\documentclass[12pt,reqno]{amsart}

\usepackage[active]{srcltx}

\usepackage{epsfig}
\newcommand{\X}{{\mathbf X_{\Gamma}}}

\newcommand{\R}{{\mathbb R}}

\newcommand{\Z}{{\mathbb Z}}

\renewcommand{\H}{{\mathbf H}}
\newcommand{\half}{{\frac{1}{2}}}

\renewcommand{\phi}{\varphi}

\newcommand{\dcal}{\mathcal{D}}

\newcommand{\gcal}{\mathcal{G}}
\newcommand{\hcal}{\mathcal{H}}

\newcommand{\lcal}{\mathcal{L}}
\newcommand{\pcal}{\mathcal{P}}
\newcommand{\mcal}{\mathcal{M}}
\newcommand{\ncal}{\mathcal{N}}

\newcommand{\rcal}{\mathcal{R}}

\newcommand{\tcal}{\mathcal{T}}

\newcommand{\VT}{\overline{V}_{T,\epsilon}}

\numberwithin{equation}{section}

\newtheorem{maintheo}{{\sc Theorem}}
\newtheorem{mainprop}{{\sc Proposition}}

\newtheorem{maincor}{{\sc Corollary}}
\newtheorem{maindefin}{{\sc Definition}}
\newtheorem{theo}{Theorem}
\newtheorem{cor}[theo]{ Corollary}
\newtheorem{lem}[theo]{ Lemma}
\newtheorem{prop}[theo]{ Proposition}

\newenvironment{rem}{\medskip\noindent{\it Remark:\/} }{\medskip}
\newenvironment{defin}{\medskip\noindent{\it Definition:\/} }{\medskip}

\title[Quantum ergodic restriction theorems, II: manifolds without boundary]
{Quantum ergodic restriction theorems, II: manifolds without
boundary}

\author{John  A. Toth}
\address{Department of Mathematics and Statistics, McGill University, Montreal, CANADA }
 \email{jtoth@math.mcgill.ca} \thanks{Research partially supported by NSERC grant \# OGP0170280 and a William Dawson Fellowship}

\author{Steve Zelditch}
\address{Department of Mathematics, Northwestern  University,
Evanston, IL 60208-2370, USA} \email{
zelditch@math.northwestern.edu}

\thanks{Research partially supported by NSF grant  \# DMS-0904252.}

\begin{document}

\begin{abstract}  We prove that if $(M, g)$ is a compact
Riemannian manifold with ergodic geodesic flow, and if $H \subset
M$ is a smooth hypersurface satisfying a generic microlocal
asymmetry condition, then restrictions $\phi_j |_H$ of an
orthonormal basis $\{\phi_j\}$ of  $\Delta$-eigenfunctions of $(M,
g)$ to $H$ are quantum ergodic on $H$. The condition on $H$ is
satisfied by geodesic circles, closed horocycles and generic
closed geodesics on a hyperbolic surface. A key step in the proof
is that matrix elements $\langle F \phi_j, \phi_j \rangle$ of
Fourier integral operators $F$ whose canonical relation  almost
nowhere commutes  with the geodesic flow must tend to zero.

\end{abstract}

\maketitle


This article is part of a series on what we call the quantum
ergodic restriction problem.  The QER problem is to determine
conditions on a hypersurface $H$ so that the restrictions
$\{\gamma_H \phi_j \}$ to $H$  of an orthonormal basis of eigenfunctions
$\{\phi_{j}\}$ of $\Delta_g$, $$\Delta \phi_j = \lambda_j^2
\phi_j$$ on a Riemannian manifold
 $(M, g)$  with ergodic
geodesic flow, are quantum  ergodic along $H$. Here, 
  $\gamma_H f = f|_H$  denotes the restriction operator to $H$.  We say that
$\{\gamma_H \phi_j \}$  is quantum ergodic along $H$ if there exists a
measure  $d\mu_H$ on $T^*H$ and a density one subsequence of
eigenfunctions so that, for any zeroth-order
 pseudo-differential operator $Op_H(a)$  along $H$,
\begin{equation}\label{QER} \langle Op_H(a) \gamma_H \phi_{j} , \gamma_H \phi_{j} 
\rangle_{L^2(H)} \to \int_{T^*H} a d\mu_H. \end{equation} Here,
the norm on  $L^2(H)$ is $||f||_{L^2(H)}^2 = \int_H |f|^2 d S$
where $dS$ is the Riemannian surface measure. We may pose the same
problem for the Neumann data $\partial_{\nu} \phi_{j}|_H$ or the
full Cauchy data $ (\gamma_H \phi_{j},  \lambda_j^{-1} \gamma_H \partial_{\nu} \phi_{j}) $ of
$\phi_{j}$ along $H$. In this article we study the QER problem for
Dirichlet data for general Riemannian manifolds without boundary
and ergodic geodesic flow. Our main result (Theorem
\ref{maintheorem}) gives a
 geometric  condition on $H$, satisfied for generic $H$,  so that the eigenfunctions of the
Laplacian of $(M, g)$ have the quantum ergodic property on $H$. In
\S \ref{HHP} it is shown that the condition is satisfied by
geodesic circles, closed horocycles and generic closed geodesics
on a hyperbolic surface. This result has applications to the
equidistribution of intersections of nodal lines and geodesics on
surfaces \cite{Z4}. In the companion paper \cite{TZ} we prove the analogue
of Theorem \ref{maintheorem} on Euclidean domains with boundary and ergodic billiards by
a quite different proof.

 In the  case of bounded domains $M \subset \R^n$  and for the special hypersurface $H= \partial M,$ it is  shown in \cite{HZ}
(see also \cite{B})  that a full asymptotic density of Neumann eigenfunctions $\{\phi_{j}|_{\partial M \}}$ are quantum ergodic.  It is important to note that these are really QER results for Cauchy data along the special boundary hypersurface where half the data happens to vanish due to the boundary conditions (ie. $\partial_{\nu} \phi_j|_{\partial M} = 0).$ In analogy with these earlier results, in \cite{TZ,CTZ} it is proved that quantum ergodicity of Cauchy
data on {\em any} interior hypersurface is inherited from quantum ergodicity of
the eigenfunctions in the ambient space.  In fact, QUE in the ambient space implies a QUE result on the hypersurface
(with respect to a certain  sub-algebra of pseudo-differential operators).  

    This article proves a much subtler  QER theorem  for the Dirichlet data alone along hypersurfaces $H$ on manifolds without boundary. 
This is  quite a different result  from the automatic QER property of Cauchy data.  In particular,  QER of Dirichlet data does not
automatically follow from quantum ergodicity in the ambient space. There
are simple examples of $H$ for which the Dirichlet data of ergodic
eigenfunctions of $(M, g)$ fail to be ergodic on $H$. For
instance, if $H$ is the fixed point set of an isometric involution
of $(M, g)$, then any odd eigenfunction with respect to the
involution will vanish on $H$. The condition given in Theorem
\ref{maintheorem} is a  microlocal {\it asymmetry} condition on
the `left' versus `right' return maps for geodesics emanating from
$H$ which rules out the existence of such an involution on the
phase space level.

To state our results, we introduce some notation.   We denote by
\begin{equation} T^*_H M = \{(q, \xi) \in T_q^* M, \;\; q\in H\} \end{equation}
 the covectors to $M$ with footpoint on $H$, and by
$T^* H=  \{(q, \eta) \in T_q^* H, \;\; q\in H\}$ the cotangent
bundle of $H$.  We further denote by $ \pi_H : T^*_H M \to T^*
H $  the restriction map,
\begin{equation} \label{RESCV}  \pi_H(x, \xi)  = \xi |_{TH}.
\end{equation}
It is a linear map whose kernel is the conormal bundle $N^* H$ to
$H$, i.e. the annihilator of the tangent bundle $TH$. In the
presence of the metric $g$,  we may identify co-vectors in $T^*M$
with vectors in $TM$ and  induce a co-metric $g$ on $T^*M$. The
orthogonal decomposition  $T_H M = T H \oplus N H$ induces an
orthogonal decomposition $T_H^* M = T^* H \oplus N^* H, $ and the
restriction map (\ref{RESCV}) is  equivalent modulo metric
identifications to the tangential orthogonal projection (or restriction)
\begin{equation} \label{piHdef} \pi_H: T^*_H M \to T^* H. \end{equation}

For any orientable (embedded) hypersurface $H \subset M$, there
exists two unit normal co-vector fields $\nu_{\pm}$ to $H$ which
span half ray bundles $N_{\pm} = \R_+ \nu_{\pm} \subset N^* H$.
Infinitesimally, they define two `sides' of $H$, indeed they are
the two components of  $T^*_H M \backslash T^* H$. We often use
Fermi normal coordinates $(s, y_n)$ along $H$ with $s \in H$ and
with $x = \exp_x y_n \nu$. We let $\sigma, \eta_n$ denote the dual
symplectic coordinates.

We also denote by $S^*_H M,$ resp. $S^* H$, the unit covectors in
$T^*_H M$, resp. $T^* H$. In general, for any subset $V \subset
T^* M$ we denote by $S V = V \cap S^*M$ the subset of unit
covectors in $V$. We may restrict \eqref{piHdef} to get  $\pi_H : S^*_H M \to B^* H,$ with
 where $B^* H$ is the unit coball bundle of $H$.
Conversely, if $(s, \sigma) \in B^* H$, then there exist two unit
covectors $\xi_{\pm}(s, \sigma) \in S^*_s M$ such that
$|\xi_{\pm}(s, \sigma)| = 1$ and $\xi|_{T_s H} = \sigma$. In the
above orthogonal decomposition, they are given by
\begin{equation} \label{xipm} \xi_{\pm}(s, \sigma) = \sigma \pm \sqrt{1 - |\sigma|^2}
\nu_+(s). \end{equation} We define the reflection involution
through $T^* H$ by
\begin{equation}\label{rHDEF}  r_H: T_H^* M \to T_H^* M, \;\;\;\;
r_H(s, \mu\; \xi_{\pm}(s, \sigma)) =(s, \mu\; \xi_{\mp}(s,
\sigma)), \,\,\,  \mu \in \R_{+}.
\end{equation} Its  fixed point
set is $T^* H$.

We denote by $G^t$ the homogeneous  geodesic flow of $(M, g)$,
i.e.  Hamiltonian flow on $T^*M - 0$ generated by $|\xi|_g$. We
then put $\exp_x t \xi = \pi \circ G^t(x, \xi)$. We emphasize that
both the geodesic flow and exponential map are homogeneous with
respect to the natural $\R_+$ action on $T^*M - 0 $, i.e. $G^t (x, r
\xi) = r G^t(x, \xi), \exp_x r \xi = \exp_x \xi$ for $|\xi| = 1$,
unlike the customary definitions in geometry.  We assume
throughout that $G^t: S^*M \to S^* M$ is ergodic (with respect to
Liouville measure $d\mu_L$). The set $S^*_H M$ of unit co-vectors
to $M$ with footpoints on $H$ forms a kind of cross-section to the
flow (see \S \ref{return times}) in the  sense that almost every
trajectory of the geodesic flow intersects $S^*_H M$
transversally. In particular, almost every trajectory from $S^*_H
M$ returns to $S^*_H M$.

 We define the {\it first return time}
$T(s, \xi)$ on $S^*_H M$ by,
\begin{equation} \label{FRTIME} T(s, \xi) = \inf\{t > 0: G^t (s,
\xi) \in S^*_H M, \;\;\ (s, \xi) \in S^*_H M)\}. \end{equation} By
definition $T(s, \xi) = + \infty$ if the trajectory through $(s,
\xi)$ fails to return to $H$. The domain of $T$ (where it is
finite) is denoted by $\lcal$ (\ref{DECAL1}).  Inductively, we
define the jth return time $T^{(j)}(s, \xi)$ to $S^*_H M$ and the
jth return map $\Phi^j$ when the return times are finite
(\ref{LCALM}). When $(x, \xi) \in S^* M \backslash S^*_H M$ the
same formula defines what we call the first `impact time' (see
(\ref{IMPACTTIMES})).

We define the first return map on the same domain by
\begin{equation} \label{FIRSTRETURN} \Phi: S^*_H M \to S^*_H M, \;\;\;\; \Phi(s, \xi) = G^{T(s, \xi)} (s, \xi) \end{equation}
When $G^t$ is ergodic, $\Phi$ is defined almost everywhere and is
also ergodic with respect to Liouville measure  $\mu_{L, H}$  on $S^*_H M$. 

\begin{maindefin} \label{ANC}  We say that  $H$ has a positive
measure of microlocal reflection symmetry if 
$$  \mu_{L, H} \left( \bigcup_{j \not= 0}^{\infty}  \{(s, \xi) \in S^*_H M : r_H G^{T^{(j)}(s, \xi)} (s, \xi)  =
 G^{T^{(j)}(s, \xi)} r_H (s, \xi)  \}\right) > 0.  $$ 
Otherwise we say that $H$ is asymmetric with respect to the geodesic flow. 

\end{maindefin}

The term ``microlocal reflection symmetry" is intended to distinguish the symmetry from a global one defined 
by a symmetry map on $M$. The  symmetry condition may be understood in terms of left and right return maps.  We use
this characterization to determine the degree of symmetry in the examples in \S \ref{HHP}. 
Since $S^*H$ disconnects $S^*_H M$, 
we have two lifts $\xi_{\pm}(s, \sigma)$ of a
covector  $(s, \sigma) \in B^* H$ to $S^*_H M$,  two almost
everywhere defined  first return maps

\begin{equation} \label{pcaldef} \pcal_{\pm}: B^* H \to B^* H,
\;\;\;\; \pcal_{\pm}(s, \sigma) = \pi_H \;\Phi \; \xi_{\pm}(s,
\sigma),
\end{equation}
and two first return times $T_{\pm}(s, \sigma)$. We define the jth
return maps similarly by
\begin{equation} \label{pcaldefj} \pcal_{\pm, j}: B^* H \to B^* H,
\;\;\;\; \pcal_{\pm,j}(s, \sigma) = \pi_H \;\Phi^j \; \xi_{\pm}(s,
\sigma),
\end{equation}
and the two jth return times by $T_{\pm}^{(j)}(s, \sigma)$ (see
(\ref{tjpm}) for the precise definition).
 Thus,  $\pcal_{\pm, j}(s, \sigma)$ is defined by lifting
 $(s, \sigma) \to \xi_{\pm}(s, \sigma)$   and following the
 trajectory
$G^t(s, \xi_{\pm}(s, \sigma))$ until it hits $S^*_H M$ for the
$j$th-time and  then projecting back to $B^* H$.
When the condition 
\begin{equation}   \label{rH}  r_H G^{T^{(j)}(s, \xi)} (s, \xi)  =
 G^{T^{(j)}(s, \xi)} r_H (s, \xi), \,\, (s,\xi) , \;\; G^{T^{(j)}(s, \xi)} r_H (s, \xi) \in S^*_H M, 
    \end{equation}
of Definition \ref{ANC}
holds, one has
$$\pcal_{+, j} (s, \sigma) = \pcal_{-, k}(s, \sigma)
$$
for a certain  $k$ which might   not equal $j$. Indeed,  \eqref{rH} can only hold if $  G^{T^{(j)}(s, \xi)} r_H (s, \xi) \in S^*_H M$,
i.e. $T^{(j)}(x, \xi)$ is a return time for $r_H(x, \xi)$. This does not necessarily imply it is the jth return time
for $r_H(x, \xi)$. Thus, the return time condition is that the $+$ and $-$ trajectories return at the same time and project to
the same covector in $B^* H$ on a set of positive measure.

We  need some further notation and background considering the `test operators' used in the limit
formula \eqref{QER}.   The result holds for both poly-homogeneous (Kohn-Nirenberg) pseudo-differential
operators and also for semi-classical
 pseudo-differential operators on $H$ with essentially the same proof.   To avoid confusion between pseudodifferential operators  on the ambient manifold $M$  and those on $H$,  we denote
 the latter by $Op_H(a)$ where $a \in S^{0}_{cl}(T^*H).$ 
 By Kohn-Nirenberg  pseudo-differential operators
we mean operators  with classical poly-homogeneous symbols   $a(s,\sigma) \in C^{\infty}(T^*H),$ 
 $$ a(s,\sigma) \sim \sum_{k=0}^{\infty} a_{-k}(s,\sigma), \,\,(a_{-k} \; \mbox{ positive homogeneous of order} 
\; -k) $$
 as $|\sigma| \rightarrow \infty$  on $T^* H$ as in \cite{HoI-IV}.  By semi-classical pseudo-differential operators
we mean $h$-quantizations of   semi-classical  symbols  $a \in S^{0,0}(T^*H \times (0,h_0]$
of the form
 $$ a_{h} (s,\sigma) \sim \sum_{k=0}^{\infty} h^k \;a_{-k}(s,\sigma), \,\,(a_{-k} \; \in  S_{1,0}^{-k}(T^* H)) $$ as in
\cite{Zw,HZ,TZ}. We choose to  emphasize  the polyhomogeneous case   because there exists a systematic reference  \cite{HoI-IV}
for the Fourier integral operator theory we require.  The  book-in-progress \cite{GuSt2} now
provides a similar systematic presentation of the semi-classical Fourier integral operator theory. The
rules for composing Lagrangian submanifolds and symbols are essentially  the same in the poly-homogeneous
and semi-classical settings, and so it is
straightforward  to adapt the proof of poly-homogeneous theorem to the semi-classical one, and 
we   do so in   Appendix \S \ref{sc}. A systematic exposition of the passage beween semi-classical
and polyhomogeneous Fourier integral operators is given in \cite{Y} (see Propositions 1.1.2-1.1.3
and 2.3.1).

We further introduce the zeroth order homogeneous function
\begin{equation}\label{gammaDEF} \gamma(s, y_n, \sigma, \eta_n) =  \frac{|\eta_n|}{\sqrt{|\sigma|^2
+ |\eta_n|^2}} = (1 - \frac{|\sigma|^2}{r^2})^{\half} ,\;\;\;(r^2 = |\sigma|^2 + |\eta_n|^2)  \end{equation}    
on $T^*_H M$ and also denote by
\begin{equation} \label{gammaBH} \gamma_{B^*H} = (1 - |\sigma|^2)^{\half} \end{equation}
its restriction to $S^*_H M = \{r = 1\}$.
 The functions (\ref{gammaBH})  are singular along  $S^*H$  and as in  \cite{HZ}   they arise in the limit measures $d\mu_H$ (we have
retained  the notation $\gamma$ from \cite{HZ} and hope that it does not conflict with the notation $\gamma_H$
for the restriction operator).
 We also use the same notation for a smooth extension of $\gamma$  to  a collar neighbourhood of $T^*_H M$ in $T^*M.$

For homogeneous pseudo-differential operators, the QER theorem is as follows:

   \begin{maintheo} \label{maintheorem} Let $(M, g)$ be a compact manifold with ergodic geodesic flow, and let  $H \subset
      M$ be a hypersurface.  Let $\phi_{\lambda_j}; j=1,2,...$ denote the
       $L^{2}$-normalized eigenfunctions of $\Delta_g$.
       If $H$ has a zero measure of microlocal symmetry, then
 there exists a  density-one subset $S$ of ${\mathbb N}$ such that
  for $\lambda_0 >0$ and  $a(s,\sigma) \in S^{0}_{cl}(T^*H)$
$$ \lim_{\lambda_j \rightarrow \infty; j \in S} \langle  Op_H(a) 
 \gamma_H \phi_{\lambda_j}, \gamma_H \phi_{\lambda_j}\rangle_{L^{2}(H)} = \omega(a), $$
 where 
 $$  \omega(a) = \frac{2}{ vol(S^*M) } \int_{B^{*}H}  a_0( s, \sigma )  \,  \gamma^{-1}_{B^*H}(s,\sigma)  \, ds d\sigma. $$

\end{maintheo}

 Alternatively, one can write  $ \omega(a) = \frac{1}{vol (S^*M)} \int_{S^*_H M} a_0(s, \pi_H(\xi)) d\mu_{L, H} (\xi).$  Note that  $a_0(s, \sigma)$ is bounded but is not
defined for $\sigma = 0$, hence $ a_0(s, \pi_H(\xi))$  is not
defined for $ \xi \in N^* H$ if  $a_0(s,\sigma)$  is homogeneous of order zero on $T^* H$. The integral can also be simplified to
$ \omega(a) = C_{M,n} \int_{S^*H} a_0 \, d\mu_L$
where, $C_{M,n} = \frac{2}{vol S^*M} \left( \int_{0}^{1} (1-r^2)^{-1/2} r^{n-2} \, dr \right)$ and $d\mu_L$ is Liouville measure on $S^*H.$ 
The analogous result for semi-classical pseudo-differential operators is:

  \begin{maintheo} \label{sctheorem} Let $(M, g)$ be a compact manifold with ergodic geodesic flow, and let  $H \subset
      M$ be a hypersurface.  
       If $H$ has a zero measure of microlocal symmetry, then
 there exists a  density-one subset $S$ of ${\mathbb N}$ such that
  for $a \in S^{0,0}(T^*H \times [0,h_0)),$
$$ \lim_{h_j \rightarrow 0^+; j \in S} \langle Op_{h_j}(a)
\gamma_H \phi_{h_j}, \gamma_H \phi_{h_j}\rangle_{L^{2}(H)} = \omega(a), $$
 where 
 $$  \omega(a) = \frac{2}{ vol(S^*M) } \int_{B^{*}H}  a_0( s, \sigma )  \,  \gamma^{-1}_{B^*H}(s,\sigma)  \, ds d\sigma.$$
\end{maintheo}

In the special case where $ a(s, \sigma) = V(s) $ is a multiplication
operator, an application of Theorem \ref{maintheorem} gives:
\begin{cor}  Under the same hypotheses as in Theorem \ref{maintheorem},  with $dS$ the surface measure on $H$,   $$ \lim_{\lambda_j \rightarrow \infty; j \in S}
\int_H V
   \left( \gamma_H \phi_{\lambda_j}\right)^2  \, d S = C_{M,n}'  \int_{H} V(s)  \,   dS,  $$  
where, $C_{M,n}' =  \frac{vol (S^{n-1})}{ vol (S^*M)}.$  \end{cor} 
This gives an asymptotic formula for the $L^2$-norms of restricted eigenfuntions in the density one subsequence,
as opposed to the $O(\lambda^{\half})$ upper bounds in \cite{R,BGT}. However, 
it  does not disqualify existence of a zero density
subsequence of eigenfunctions whose $L^2$ norms  blow up along $H$. Thus, it
is consistent with  the sequence of recent results on restrictions
of eigenfunctions to hypersurfaces in \cite{BGT,KTZ,R,T,To,So}. As
in the original work of A. I. Schnirelman \cite{Sch}, QER is
concerned with density one subsequence of eigenfunctions and thus
may exclude the `extremal' eigenfunctions with respect to $H$.

\subsection{Examples} As in Proposition 6 of  \cite{TZ}, it is possible to show that
a generic hypersurface $H$ has zero measure of microlocal
symmetry, hence that the  result is non-vacuous. But we omit the
details and concentrate on interesting examples when $(M, g) $ is
a finite area hyperbolic surface. In \S \ref{HHP} we apply Theorem
\ref{maintheorem} to prove:

\begin{maincor} \label{HYPCIRCLE} Let $(M, g)$ be a finite area hyperbolic surface and let $H$ be a geodesic
circle or a closed horocycle  of radius $r < inj(M, g)$, the
injectivity radius. Then a full density sequence of eigenfunctions
restricts to a quantum ergodic sequence on  $H$. The same is true
for generic Fuchsian groups and closed geodesics.\end{maincor}

The case of closed horocycles was numerically tested in \cite{HR}.

\subsection{Outline  of the proof of Theorem \ref{maintheorem}}

  We  denote by  $U(t) = e^{i t
\sqrt{\Delta}}$ the wave group of $(M, g)$. As is well-known it is a homogeneous Fourier integral operator
whose canonical relation is the graph of the homogeneous geodesic flow at time $t$. We denote  by  $\gamma_H^*$   the adjoint of $\gamma_H$ with
respect to the inner product on $L^2(M, dV)$ where $dV$ is the
Riemannian volume form.  Thus,
$$\gamma_H^* f = f \delta_H, \;\; \mbox{since}\;\;
\langle \gamma_H^* f, g \rangle = \int_H f g   dS, $$ where
as above $ dS $ is the surface measure on $H$ induced by the
ambient Riemannian metric. The fact that $\gamma_H^*$ does not preserve
smooth functions is due to the fact that $WF_M'(\gamma_H) = N^* H$ (see \S \ref{gammaHsect} and \cite{HoI-IV}, Ch. 8.2 for the notation).  Thus,   $\gamma_H^* Op_H(a) \gamma_H$
is not a Fourier integral operator with a homogeneous canonical relations in the sense of \cite{HoI-IV} because its wave front
relation contains  $N^*H \times 0_{T^*M} \cup 0_{T^* M} \times N^* H$  (where
$0_{T^*M}$ is the zero section of $T^* M$).

We study matrix elements of the restriction through the identity,
\begin{equation} \label{ME1} \begin{array}{lll}  \langle Op_H(a) \phi_j |_H, \phi_j |_H \rangle_{L^2(H)} & =
& \langle Op_H(a) \gamma_H \phi_j, \gamma_H \phi_j \rangle_{L^2(H)} \\ & & \\
& = &  \langle \gamma_H^* Op_H (a) \gamma_H U(t) \phi_j, U(t) \phi_j
 \rangle_{L^2(M)} \\ && \\
& = &  \langle   V(t;a)  \phi_j,
 \phi_j
 \rangle_{L^2(M)} \\ && \\ & = & \langle \bar{V}_T (a) \phi_j, \phi_j
 \rangle_{L^2(M)} \\ && \\ & = & \langle \bar{V}_{T,R} (a) \phi_j, \phi_j
 \rangle_{L^2(M)},
\end{array} \end{equation}
where
  \begin{equation}\label{VTFORM}\left\{ \begin{array}{lll}
  V(t;a) :=U(-t) \gamma_H^* Op_H(a) \gamma_H U(t), \\ \\
  \bar{V}_{T}(a) :=  \frac{1}{T}
\int_{-\infty}^{\infty} \chi ( T^{-1} t)  \, V(t;a) \, dt,\\
\\  \bar{V}_{T, R}(a) := \frac{1}{2R} \int_{-R}^R
U(r)^* \bar{V}_T(a) U(r) dr.  \end{array} \right.  \end{equation}
 The 
 double average in $r, t$ is only for convenience  (see section \ref{psdovariance}).

A further technical complication is that
$\overline{V}_T(a)$ is a Fourier integral operator with fold singularities. 
 It is closely
related to the operator $W^*W$ where   $W f = \gamma_H U(t) f$  (see \S \ref{WSECT}).  As already observed in \cite{Ta}, the canonical
relations of these operators are local canonical graphs away from  fold singularities
along directions tangent to $H$ (see below for a more precise statement).  In
the quantum ergodicity problem, it suffices to introduce pseudo-differential cutoffs to  cut off away from the fold singularity as
in \cite{HZ,TZ,SoZ}; we do not need the calculus of Fourier integral operators with
fold singularities as in  \cite{GrS,F}.  However, the fold singularity does induce singularities
in the symbols of the main operators. For instance, 
the push forward of the standard measure on $S^*_H M$ 
to $B^* H$ is  $\gamma_{B^* H}^{-1} ds \wedge d \sigma$ \eqref{gammaBH}.

A detailed description of $\overline{V}_T(a)$  is given in 
the next Proposition \ref{VTDECOMPa}. There it is proved
 that, after cutting off   from the tangential  singular set $\Sigma_T
\subset T^* M \times T^* M$ and the   the conormal sets $N^*H \times 0_{T^*M}, 0_{T^* M} \times N^*H$,  $\bar{V}_{T}(a)$ becomes  a Fourier
integral operator $\VT(a)$  with canonical relation  given by
 \begin{equation}\label{GAMMADELTA}
\begin{array}{ll} WF(\VT(a)): & = \{(x, \xi, x', \xi') \in T^*M \times T^*M :  \exists t  \in (-T,T),    \\ & \\ & \exp_x t \xi = \exp_{x'} t \xi' = s \in H ,
\, \, G^t(x, \xi) |_{T_sH} = G^t(x', \xi') |_{T_s H}, \;\; |\xi| =
|\xi'|\}.
\end{array} \end{equation}

To  define a Fourier integral operator $\VT(a)$,  we need to introduce cutoff operators 
to cutoff away from $T^*H$ and from  
 $N^*H \times 0_{T^*M}  \cup 0_{T^* M} \times N^*H.$ 
We let  $\chi \in C_0^{\infty}(\R), [0,1]$ be a cutoff  supported in $( -1-\delta, 1+ \delta),$ with
$\chi(t)=1$ for $t \in [-1+ \delta,1-\delta], \,
 \int_{-\infty}^{\infty} \chi(t) dt = 1. $
For fixed $\epsilon >0,$  we introduce  two cutoff pseudo-differential operators (see subsection \ref{CUTOFFS2} for more detail). The first,  $\chi^{(tan)}_{\epsilon}(x, D) = Op(\chi_{\epsilon}^{(tan)}) \in Op(S^0_{cl}(T^*M)),$ has 
homogeneous symbol $\chi^{(tan)}_{\epsilon}(x,\xi)$ supported in an $\epsilon$-aperture conic neighbourhood of $T^*H \subset T^*M$ with $\chi^{(tan)}_{\epsilon} \equiv 1$ in an $\frac{\epsilon}{2}$-aperture subcone. The second cutoff operator $\chi^{(n)}_{\epsilon}(x,D) = Op(\chi_{\epsilon}^{(n)}) \in Op(S^0_{cl}(T^*M))$ has its homogeneous symbol  $\chi^{(n)}_{\epsilon}(x,\xi)$ supported in an $\epsilon$-conic neighbourhood of $N^*H$ with $\chi^{(n)}_{\epsilon} \equiv 1$ in an $\frac{\epsilon}{2}$ subcone.
To simplify notation, define the total cutoff operator
\begin{equation} \label{CUTOFF}
\chi_{\epsilon}(x,D) := \chi^{(tan)}_{\epsilon}(x,D) + \chi^{(n)}_{\epsilon}(x,D),
\end{equation}
and put
\begin{equation} \label{>ep}
(\gamma_H^* Op_H(a) \gamma_H)_{\geq \epsilon} = (I - \chi_{\frac{\epsilon}{2}}) \gamma_H^* Op_H(a) \gamma_H (I - \chi_{\epsilon}), \end{equation}
and
\begin{equation} \label{<ep} (\gamma_H^* Op_H(a)
 \gamma_H)_{\leq \epsilon} =  \chi_{2 \epsilon} \gamma_H^* Op_H(a) \gamma_H \chi_{\epsilon}. \end{equation}

By standard wave front calculus (see subsection \ref{finish}), it follows that  \begin{equation} \label{op decomp}
\gamma_H^* Op_H(a)
 \gamma_H = (\gamma_H^* Op_H(a)
 \gamma_H) )_{\geq \epsilon} +  (\gamma_H^* Op_H(a)
 \gamma_H) )_{\leq \epsilon}  + K_{\epsilon},\end{equation}
 where, $\langle K_{\epsilon} \phi_j, \phi_j \rangle_{L^2(M)} = {\mathcal O}(\lambda_j^{-\infty}).$
We then   define
  \begin{equation} \label{VtFORMep}
  V_{\epsilon} (t;a) :=U(-t) (\gamma_H^* Op_H(a) \gamma_H)_{\geq \epsilon} U(t),  \end{equation}
and
 \begin{equation}\label{VTFORMep}   \VT(a) :=  \frac{1}{T}  \int_{-\infty}^{\infty} \chi ( T^{-1} t)  \, V_{\epsilon}(t;a) \, dt.    \end{equation}

We can now state the two main steps in the proof of 
Theorem \ref{maintheorem}.  Foremost is the variance result,

\begin{mainprop}\label{BIGSTEP} For all $\epsilon > 0$, 
$$\lim_{\lambda \to \infty}   \frac{1}{N(\lambda)}
\sum_{j: \lambda_j \leq \lambda} \left|  \langle \left(\gamma_H^*
Op_H(a) \gamma_H)_{\geq \epsilon}\right) \phi_j, \phi_j
 \rangle_{L^2(M)}  -\omega((1 - \chi_{\epsilon}) a)\right|^2  = 0. $$

\end{mainprop}




%
%

The beginning of the proof of   Proposition \ref{BIGSTEP}  follows the sketch in  \eqref{ME1}. We then decompose  $\VT(a)$ into  a pseudo-differential
and a Fourier integral part according to the dichotomy that 
$(x, \xi, x', \xi')$ in (\ref{GAMMADELTA}) satisfy either
\begin{equation} \label{possibilities}
\ \begin{array}{ll}
 (i) \,\, G^t(x,\xi) = G^t (x',\xi'), \; \mbox{or}\\ \\
 (ii) \,\, G^{t}(x',\xi') = r_{H} G^{t}(x,\xi), \end{array}  \end{equation}
 where $r_H$ is the reflection map of $T^* H$  in (\ref{rHDEF}). Thus, 
\begin{equation} \label{WFVT} WF(\VT(a)) = \Delta_{T^* M \times T^* M} \cup \Gamma_T, \end{equation}
where  
  \begin{equation} \label{relationi} \left\{\begin{array}{l} \Delta _{T^*M \times T^* M} := \{ (x,\xi,x,\xi) \in T^*M \times T^*M \}, \\ \\
 \Gamma_T =
\bigcup_{(s, \xi) \in T^*_H M}
 \bigcup_{ |t|  < T} \{( G^t(s, \xi),  G^t(r_H (s, \xi)) \}.
\end{array} \right. \end{equation} 
The two `branches' or components intersect along the singular set
 \begin{equation} \label{SIGMA} \Sigma_T : = \bigcup_{|t| < T}  (G^t \times G^t) \Delta_{T^* H \times T^* H}. \end{equation}
We further subscript  $\Gamma_T$ with $\epsilon$ to indicate the points $\Gamma_{T, \epsilon}$  outside the support of the tangential cutoff \eqref{GAMMATEPDEF}.

Since $ G^t(r_H (s, \xi)) = G^t r_H G^{-t} G^{t}
 (s,\xi)$,
 $\Gamma_{T, \epsilon} \subset  \Gamma_{T} \backslash \Sigma_T$ is the graph of a
 symplectic correspondence. The precise statement is in
 Proposition \ref{graph}, where we  show that for any $\epsilon > 0$,  $\Gamma_{T, \epsilon}$ is
 the union of a finite number $N_{T, \epsilon}$  of graphs of partially defined
 canonical transformations
 \begin{equation} \label{RCAL} \rcal_j(x,\xi) = G^{t_j(x,\xi)} r_H G^{-t_j(x,\xi)}(x,\xi).\end{equation}
which we term $H$-reflection maps. Here $t_j(x, \xi)$ is the jth
 `impact time', i.e. the  time to the $j$th impact with $H$. We denote its
 domain (up to time $T$) by $\dcal_{T, \epsilon}^{(j)}$
(see \S \ref{IMPACTRETURN} and Definition \ref{IMPACTTIMES}). By
homogeneity of  $G^{t}: T^*M \rightarrow T^*M$,  for all $j \in
\Z,$
\begin{equation} \label{impact}
t_{j}(x, \xi) = t_j(x, \frac{\xi}{|\xi|}); \, \, \xi \neq 0.
\end{equation}

The next proposition provides  a detailed discussion of $W^* W$ as  a Fourier integral operator
with local canonical graph away from its fold set. Furthermore, the  principal symbol is computed. It therefore
seems of interest independently of its applications to QER. For more details on its relation to $W$ we refer to
\S \ref{WSECT}. As A. Greenleaf pointed out to the authors, there are related calculations in \cite{GrU,F}.

  \begin{mainprop}\label{VTDECOMPa}  Fix $T,\epsilon >0$ and let
  $a \in S^{0}_{cl}(T^*H)$ with    $a_H(s, \xi) = a(s, \xi |_H) \in S^0(T^*_H M) ). $  Then $\VT(a)$ is a Fourier
integral operator with local canonical graph,
and possesses the decomposition
  $$\VT(a) =  P_{T,\epsilon}(a) + F_{T,\epsilon}(a) + R_{T,\epsilon}(a),$$
where, (i) $P_{T,\epsilon}(a) \in Op_{cl}(S^0(T^*M))$ is a
pseudo-differential operator of order zero with principal symbol
 \begin{equation} \label{aTFORM}  a_{T,\epsilon}(x,\xi)
 :=\sigma(P_{T,\epsilon}(a))(x,\xi) = \frac{1}{T}  \sum_{j \in {\mathbb Z}}  (1-\chi_{\epsilon})  (\gamma^{-1} a_H)(G^{t_j(x,\xi)}(x,\xi))  \, \chi (T^{-1}  t_j(x,\xi))  \end{equation}
  where,
 $t_j(x,\xi) \in C^{\infty}(T^*M)$ are the impact times of the geodesic $\exp_{x}(t \xi)$ with $H$ (see Definition
 \ref{IMPACTTIMES}), and $\gamma$ is defined by (\ref{gammaDEF}).

(ii) $F_{T,\epsilon}(a)$ is a
 Fourier integral operator of order zero  with canonical relation
 $\Gamma_{T,\epsilon}$. \begin{equation} \label{FIOsum}
F_{T,\epsilon}(a) = \sum_{j=1}^{N_{T,\epsilon}}
F_{T,\epsilon}^{(j)}(a),\end{equation} where the
$F_{T,\epsilon}^{(j)}(a); j=1,..., N_{T,\epsilon}$ are zeroth-order
homogeneous Fourier integral operators with
$$ WF'(F_{T,\epsilon}^{(j)}(a)) = \text{graph}  (\rcal_j) \cap \Gamma_{T,\epsilon},$$
and symbol
$$\sigma(F_{T,\epsilon}^{(j)})(x,\xi) =  \frac{1}{T}   (\gamma^{-1} a_H)( G^{t_j(x, \xi)}(x, \xi)) \, \chi(T^{-1} t_j(x,\xi)) \ |dx d\xi|^{\frac{1}{2}}.$$

(iii) $R_{T,\epsilon}(a)$ is a smoothing operator.

\end{mainprop}


  Given Proposition \ref{VTDECOMPa},   the proof of Proposition \ref{BIGSTEP} goes as follows:
By (\ref{ME1}), it suffices to show that
 $$\limsup_{\lambda \to \infty} \frac{1}{N(\lambda)} \sum_{j: \lambda_j \leq \lambda} \left|  \langle
\VT(a)  \phi_j,  \phi_j
 \rangle_{L^2(M)} - \omega((1-\chi_{\epsilon})a) \right|^2 = o(1), \;\; (\mbox{as}\; T\to \infty).$$
By  Proposition \ref{VTDECOMPa},  $\VT(a)$ is a sum of a
pseudo-differential part $P_{T, \epsilon}$, a Fourier integral
part $F_{T, \epsilon}$ associated to graphs of H-reflection maps,
and a small term $R_{T, \epsilon}$.
By the inequality $(a_1+ \cdots + a_n)^2 \leq n^2 (a_1^2 + \cdots + a_n^2)$
it suffices to estimate the variances of each term separately. It is simple to show that the $R_{T,\epsilon}$-term is negligible.

  The pseudo-differential term is somewhat (but not entirely)
similar to that encountered in \cite{HZ}. Following the argument
there (and in the standard argument), we use the $L^2$ ergodic
theorem to show that
\begin{equation} \label{PTVAR} \limsup_{\lambda \to \infty} \frac{1}{N(\lambda)} \sum_{j: \lambda_j \leq \lambda} \left| \langle P_{T,\epsilon}(a) \phi_j,  \phi_j
 \rangle_{L^2(M)} - \omega((1-\chi_{\epsilon})a) \right|^2 = 0, \end{equation}
 and indeed the state $a \to
\omega(a)$ is defined so that the $L^2$ ergodic theorem applies in
this way (it is calculated in Proposition \ref{LWLH}). Since $P_{T, \epsilon}$ is pseudo-differential, this is
simply the standard quantum ergodicity theorem.

This reduces us to studying the variances
 \begin{equation} \label{FIO part}
  \frac{1}{N(\lambda)} \sum_{j: \lambda_j \leq \lambda} \left|  \langle F_{T,\epsilon}(a)  \phi_j,  \phi_j
 \rangle_{L^2(M)} \right|^2.  \end{equation} It
  is here that we need the condition in Definition \ref{ANC} and where we encounter the novel aspects in the proof.
To prove (\ref{FIO part}) we first use the Schwartz inequality
\begin{equation} \label{SIE} \frac{1}{N(\lambda)} \sum_{j: \lambda_j \leq \lambda} \left|
\langle F_{T,\epsilon}(a)  \phi_j,  \phi_j
 \rangle_{L^2(M)} \right|^2 \leq \frac{1}{N(\lambda)} \sum_{j: \lambda_j \leq \lambda}   \langle F_{T,\epsilon}(a)^* F_{T,\epsilon}(a)  \phi_j,  \phi_j
 \rangle_{L^2(M)} \end{equation}
 to bound the variance sum by a trace. In \S \ref{LWLFIO}  we
recall (and extend)  the  local Weyl law for homogeneous Fourier
integral operators
  $F: C^{\infty}(M) \rightarrow C^{\infty}(M)$ of \cite{Z} and
  use it to prove that the right side of (\ref{SIE}) tends to zero
  under the assumption of Definition \ref{ANC} (see also \cite{TZ}
  for a similar argument).
In the case of  local canonical graphs, the local  Weyl law states
that
\begin{equation} \label{oldWeyl}
\frac{1}{N(\lambda)} \sum_{j: \lambda_j \leq \lambda} \langle F
\phi_{\lambda_j}, \phi_{\lambda_j} \rangle \to \int_{S \Gamma_F
\cap \Delta_{T^*M} } \sigma_{\Delta} (F) d\mu_L, \end{equation}
where $\Gamma_F$  is the canonical relation of $F$,  $S \Gamma_F$
is the set of vectors of norm one, and  $S \Gamma_F \cap
\Delta_{T^*M} $  is its intersection  with the diagonal of $T^*M
\times T^*M$. Also,  $ \sigma_{\Delta} (F)$ is the (scalar) symbol
in this set and $d\mu_L$ is Liouville measure. Thus, if $\Gamma_F$
is a local canonical graph, the  right side is zero unless the
intersection has dimension $m = \dim M$, i.e. the trace  sifts out
the `pseudo-differential part' of $F$.

 An application of  (\ref{oldWeyl}) to $F = F_{T,\epsilon}(a)^* 
F_{T,\epsilon}(a)$ gives: 
\begin{lem} \label{WeylCOR}  We have,
\begin{equation} \begin{array}{ll} \lim_{\lambda \rightarrow \infty} \frac{1}{N(\lambda)}
\sum_{\lambda_j \leq \lambda} \sum_{k, \ell = 1}^{N_{T, \epsilon}}
\langle F_{T,\epsilon}^{(\ell)}(a)^*
 F_{T,\epsilon}^{(k)}(a)  \phi_{\lambda_j}, \phi_{\lambda_j}
 \rangle \\ \\    =  \frac{1}{T^2} \int_{S^* M} \sum_{j = 1}^{N_{T, \epsilon}} \left| \chi(\frac{t_j(x, \xi)}{T}) \ (1-\chi_{\epsilon}) \gamma^{-1}
 a_H ( G^{t_j(x, \xi)} (x, \xi) ) \right|^2 d\mu_L  \\ \\ +   \frac{1}{T^2}  \int_{S \{ \rcal_j = \rcal_k \} }   \sum_{j \neq k}^{N_{T, \epsilon}}  \chi(\frac{t_j(x, \xi)}{T}) \,(1-\chi_{\epsilon}) \gamma^{-1}
 a_H ( G^{t_j(x, \xi)} (x, \xi) ) \\ \\ \chi(\frac{t_k(x, \xi)}{T}) \,(1-\chi_{\epsilon}) \gamma^{-1}
 a_H ( G^{t_k(x, \xi)} (x, \xi) ) \, d\mu_L \end{array} \end{equation}
 \end{lem}
  Since $N_{T,\epsilon} = {\mathcal O}_{\epsilon}(T)$ and $ |\chi| \leq 1,$ the  first term on the  right side in Lemma \ref{WeylCOR} is
 \begin{equation} \label{FIObound}
  {\mathcal O}_{\epsilon}\left( \frac{1}{T} \| a_H \|^2_{C^0(S^*M_H)} \right).  \end{equation}

 The fact that the second term on the RHS in Lemma \ref{WeylCOR} vanishes follows from

\begin{lem} \label{mrns} The hypersurface $H \subset M$ has zero measure of microlocal reflection symmetry if and only if for all $(T,\epsilon),$
 $\mu_L  \left(    S\{\rcal_j = \rcal_k\}_{T,\epsilon} \right) =
0$ for $j \not= k$ (cf. Definition \ref{ANC}).  \end{lem}

Indeed,
\begin{equation}  \label{RJERKintro} \begin{array}{lll} \rcal_j(x, \xi) = \rcal_k(x, \xi) && \iff  G^{-t_j(x,\xi)} r_H G^{t_j(x,\xi)}(x,\xi)=  G^{-t_k(x,\xi)} r_H G^{t_k(x,\xi)}(x,\xi)
\\ && \\
&& \iff  r_H  G^{t_j(x, \xi) -t_k(x,\xi)}  =
G^{t_j(x, \xi) -t_k(x,\xi)} r_H (x, \xi) .\end{array}
\end{equation}

Hence under the assumption of Theorem \ref{maintheorem},  after taking the $T \rightarrow \infty$ limit,   the right
side in Lemma \ref{WeylCOR} is zero, proving the theorem for cutoff symbols.
To complete the proof, we  show in \S \ref{finish}  that the  cutoffs $\chi_{\geq \epsilon}^n, \chi_{\geq \epsilon}^{tan}$ on $\gamma_H^*
Op_H(a) \gamma_H$ can be removed in the limit formula along the density one subsequence.

\subsection{\label{WSECT} Background to Proposition \ref{VTDECOMPa}}

The fact that $\VT(a)$ is a
Fourier integral operator with local canonical graph is closely related to  the fact (used in \cite{Ta,GrS,So,F,BGT,SoZ})
that the operator
\begin{equation} \label{W} W: f \in C(M) \to  \gamma_H B U(t) f \in C(\R \times H) \end{equation}
is a Fourier integral operator with local canonical graph. Here,  $B \in \Psi^0(M)$ is a poly-homogeneous pseudo-differential operator whose symbol vanishes in a
conic neighborhood of $T^* H$.   
Although we do not need such a precise description, $W$ without cutoffs is a Fourier
integral  operator with one-sided folds.  It has the canonical relation
\begin{equation} \label{GAMMAW} \Gamma_W = \{(t, \tau, q, \xi|_H, G^t(q, \xi): (q, \xi) \in T^*_H M: |\xi | = \tau\}
\subset T^*(\R \times H) \backslash 0 \times T^* M \backslash 0 \end{equation}
and in the associated diagram

\begin{equation}\label{DIAGRAMa} \begin{array}{ccccc} & & \Gamma_W \subset T^*(\R \times H) \times T^* M
 & &  \\ & & & & \\
& \pi \swarrow & & \searrow \rho & \\ & & & & \\
T^*(R \times H)  & & & & T^*M,
\end{array}
\end{equation}
the left projection is 2-1 except along the set  $\{(t, \tau, q, \xi |_H, q, \xi): |\xi| = \tau = |\xi |_H|\} $
(i.e. $\xi \in T^* H)$
where it has a fold singularity. 

Our operator $\VT(a)$ is closely related to $W^* W$: to be precise, it is $W^* \chi_T Op_H(a) W$
where $\chi_T$ is the time cutoff. 
 To prove Theorem \ref{maintheorem}, we only  need to compute symbols in the local canonical graph part, but we need to understand the singularity of the symbol
along $T^*H$.  To our knowledge, the symbols of $W$ and $W^*W$ on the canonical graph
part have not been calculated before, and we go through the calculation in \S \ref{COMPOSITIONS} - \S \ref{v operator}.

The composition with $Op_H(a)$ causes a minor complication since $a(q, \sigma)$ is not 
necessarily  a poly-homogeneous symbol 
on $T^*(\R \times H)$ because it is independent of the $\tau$ variable dual to $t$. This bad behavior of $a$ is another
way to view the role of  $N^*H \times 0_{T^*M} \cup 0_{T^* M} \times N^* H$ in the wave front set of
$\gamma_H^* Op_H(a) \gamma_H$.

\subsection{Comparison to boundary case}

We now compare the methods and the results of this article to
those in \cite{HZ,TZ}. In particular, we explain how the
difficulties caused by tangential directions $\gcal_T $ (or more precisely, $\Sigma_T$)  relate to
those when $H = \partial M$ in the boundary case.

The main difference is that the wave group $U(t)$ is a global
Fourier integral on $M$ with $\partial M = \emptyset$, but is not
one if $\partial M \not= \emptyset$. When $\partial M =
\emptyset$, there  are no boundary conditions on $U(t)$  on the
hypersurface $H$, indeed  $U(t)$ is independent of $H$.  In the
boundary case, the boundary conditions on $U(t)$ cause the Fourier
integral structure of $U(t)$ to break down microlocally near
tangent directions to $\partial M$. Moreover, the geodesic flow
$G^t$ is also independent of $H$ and we do not have the  problems
of reflections off  corners of $\partial M$ in its definition.

In \cite{HZ}, the breakdown of $U(t)$ in tangential directions to
$\partial M$ was handled by making a reduction to the boundary.
Thus, we used the symmetries of  matrix elements \eqref{QER}  when $H
= \partial M$ which came from a certain boundary integral operator
$F(\lambda)$. Conjugation with $F(\lambda)$ was an endomorphism,
not an automorphism, of the pseudo-differential operators on $H$
and that gave rise to additional singularities caused by the
factor $\gamma = \sqrt{1 - |\eta|^2}$ in \cite{HZ}. We do not make
a reduction to $H$ in this article and only use the symmetry given
by conjugation by $U(t)$. In a sense, we go in the opposite
direction of `extending from $H$ to the interior' rather than
reducing to $H$.   In revenge, we have to deal with this symmetry
on matrix elements of Fourier integral operators. But we avoid the
singularities caused by the $\gamma$ factor.

On the other hand, the symbols of $P_{T, \epsilon} $ and $F_{T,
\epsilon}$ do become singular along directions of geodesics which
touch $H$ tangentially, and we have to cut away from these
directions. More precisely, there is a failure of transversal
composition at these points, so that $P_{T, \epsilon} $ and $F_{T,
\epsilon}$ fail to be Fourier integral operators at such points.
In that sense, the problem caused by $\Sigma_T $ is somewhat similar
to the boundary case, but as mentioned above it is not quite as
serious since $U(t)$ is a global Fourier integral operator and
$\Sigma_T $ only causes mild singularities such as folds and
self-intersections in the canonical relations and corresponding
singularities in the symbols.

In the case of Euclidean domains $\Omega \subset \R^n$ with
boundary  studied in
 \cite{TZ},  we made a
reduction from an interior hypersurface $H$ to $\partial M$. In
part we could then reduce the quantum ergodicity problem to
\cite{HZ}. However, the reduction left a Fourier integral term
which was handled by a method related to that of this article. The
almost nowhere commuting condition for QER  was that the sets
where $\beta^m \beta_H = \beta_H \beta^k $ had measure zero, and
where $\beta_H$ was a certain transmission map. We  tie this
together with the condition of the present article.

In the boundary case, we also have two lifts  $\xi_{\pm}(y, \eta)$
of covectors to $H$. In addition, unlike the boundary-less case,
there are   two transfer maps $\sigma_{\pm}(y, \eta)$ to the
boundary and a billiard map $\beta: B^* \partial \Omega \to B^*
\partial \Omega$ of the boundary.  The transmission map is defined by  $\beta_H= \sigma_-
\sigma_+^{-1}.$

  The microlocal asymmetry condition of \cite{TZ} stated that
$\beta^m \sigma_- \sigma_+^{-1}= \sigma_- \sigma_+^{-1} \beta^k $
has measure zero or equivalently
$$\sigma_-^{-1}\beta^m \sigma_-
=  \sigma_+^{-1} \beta^k \sigma_+.$$ But $\sigma_-^{-1}\beta^m
\sigma_-$ is the $-$ side first return map, while $\sigma_+^{-1}
\beta^k \sigma_+$ is the $+$ side first return map, hence the
condition of \cite{TZ} is equivalent to saying that the
$\pm$-sided return maps do not agree on a set of positive measure.

\subsection{Further results}

As mentioned above, the authors together with H. Christianson have proved in \cite{CTZ} that quantum
ergodicity in the ambient manifold always implies QER of the Cauchy data on any hypersurface. Moreover, QUE (unique
quantum ergodicity) in the ambient manifold implies a kind of QUER property for $H$. Namely, there is a subalgebra
of pseudo-differential operators on $H$ for which the QER property holds without needing to extract subsequences.
However, the symbols of the operators are all multiplied by $\gamma$ in the restriction operation and therefore
vanish to first order along $T^* H$. 

A further direction is to find a spectral condition on $H$ which implies QER rather than the dynamical condition
in Definition \ref{ANC}. Suppose that $H$ is a separating hypersurface and that $M \backslash H = M_+ \cup M_-$
where $M_{\pm}$ are the disjoint components. Then one may consider the Dirichlet problem for $\Delta$ in $M_{\pm}$.
In the case where $H$ is the fixed point set of an isometry, the Dirichlet spectra  of $M_{\pm}$ has a large
overlap with the global spectrum of $\Delta$ on $M$. It is plausible that in place of Definition \ref{ANC} one can
prove QER under the assumption that the Dirichlet spectra of $M_{\pm}$ are disjoint from the spectrum of $M$. 
Indeed, in this case the Dirichlet-to-Neumann operators $\ncal_{\pm}(\lambda)$  for $M_{\pm}$ are well-defined at eigenvalue parameters
for $\Delta$. When  $\ncal_{\pm}(\lambda)$ is a semi-classical Fourier integral operator (see \cite{HZ,TZ}), it is 
plausible that the QER of the Cauchy data implies QER of both the Dirichlet and Neumann data separately. The authors
plan to investigate this on a later occasion.

\subsection{Organization}

  The first two sections \S \ref{return times} and \S \ref{IMPACTRETURN} 
consist of preliminary material on the symplectic geometry of $S^*_H M$ as a cross section to the geodesic flow
and on the fold singularity of the projection $S^*_H M \to B^* H$. This fold singularity causes the tangential
singularities in the canonical relation of $\VT$.  This material is often suppressed
for the sake of brevity, but as a result we could not find any reference for the details we need. Section
\S \ref{COMPOSITIONS} is needed to compute the symbol of $\VT$ in \S \ref {Vta operator}. The calculation is important
both to obtain the correct limit meaure in \eqref{QER} and also because it shows that  the symbol lies outside
of $L^2$. This complicates the use of the mean ergodic theorem. With these preliminaries in hand, the proof of 
Theorem \ref{maintheorem} proper begins in \S \ref{LWLFIO}.   In \S \ref{sc} we adapt the proof to semi-classical
pseudo-differential operators.  In \S \ref{nice case},  we provide a study of examples
to confirm that the condition in Definition \ref{ANC} is effective in concrete examples.

\subsection{Acknowledgements and remarks on the exposition}

Since the original posting of this article, S. Dyatlov and M. Zworski \cite{DZ} have proved a more
general version of the results of this paper,  where $\Delta$ is generalized to a semi-classical Schr\"odinger
operator. We thank them in addition for questions and comments helping us to improve the presentation. In  particular,
we  clarified    Definition \ref{ANC} and added detail on the cutoffs away from $N^* H \times 0_{T^*M}
\cup 0_{T^*M} \times N^*H$, and simplified the last section. We also thank A. Greenleaf and C. Sogge for  comments on
the exposition.

\section{ Geometry of hypersurfaces} \label{return times}

This section provides background on the  symplectic and
Riemannian submanifold geometry in the analysis of the
canonical relation of the restriction map $\gamma_H$ and the
subsequent compositions defining $V_{\epsilon}(t;a)$ and $\VT(a)$.

Let $H \subset M$ be a hypersurface in a Riemannian manifold $(M,
g)$. We consider two hypersurfaces of $T^* M$, the set $T^*_H M$
of covectors with footpoint on $H$ and the unit cotangent bundle
$S^* M$ of $g$.  We  often use metric Fermi normal coordinates
along
 $H$, i.e.  we exponentiate the normal bundle to
 $H$. We denote by $s$ the coordinates along $H$ and $y_n$ the
 normal coordinate, so that $y_n = 0$ is a local defining function for $H$.   We also let $\sigma, \xi_n$ be
 the dual symplectic Darboux coordinates. Thus the canonical
 symplectic form is $\omega_{T^* M } = ds \wedge d \sigma + dy_n
 \wedge d \xi_n. $

Let $\pi: T^* M \to M$ be the natural projection. We identity
$\pi^*y_n = y_n$ as functions on $T^* M$. Then $ f=y_n = 0$ is the
defining function of  $T^*_H M$.   The
  hypersurface $S^* M$ is defined by $g = |\xi| = 1$, the metric
  norm function.  It is clear that
  $df, dg$ are linearly independent, so that $T^*_H M, S^* M$ are
  a pair of transversal hypersurfaces in $T^* M$.

  In general, let $F, G \subset T^* M$ be two transversally intersecting hypersuraces,
   and let $f$, resp.
$g$, denote defining function of $F$, resp. $G$, so that $f = 0$
on $F$, $g = 0$ on $G$ and $df, dg$ are linearly independent. Then
their intersection $J = F \cap G$ is a submanifold of codimension
two. The intersection fails to be symplectic along the set $K =
\{x \in J: \{f, g\}(x) = 0\}. $ In certain circumstances, the
Hamiltonian flow lines of $f$ intersect $J$ in two points which
are different on $J \backslash K$. This is the so-called glancing
set or points of bicharacteristic tangency. The map taking one
intersection point to the other defines an involution of $\iota_F:
J \to J$ with fixed point set $K$. When $K$ is a hypersurface in
$J$ and $\iota_F$ is smooth, the eigenvectors of eigenvalue $-1$
of  $D \iota_F: T_k J \to T_k J$ define a line bundle over $K$
known as the reflection bundle of $\iota_F$. We refer to
\cite{HoI-IV} Section 21.4 and \cite{Me}.

\subsection{Restrictions and folds} \label{HYPERSURFACES}

We   are interested in the case, $F = T^*_H M, \;\; G = S^* M,
\;\; J = S^*_H M, \;\; K = S^* H. $ The Hamilton vector field of
 $y_n$ equals  $\frac{\partial}{\partial \xi_n}$ and  its orbits
 are vertical curves
 of the form $(s, 0, \sigma, \xi_{n 0} + t);  $  they
  define the characteristic foliation of $T^*_H M$. The
  hypersurface $S^* M$ is defined by $g = |\xi| = 1$, the metric
  norm function, and its characteristic foliation is given by
  orbits of the homogeneous geodesic flow $G^t$.  Evidently,
$$\{x_n, |\xi|_g\} = \frac{\partial}{\partial \xi_n} |\xi|_g =
|\xi|_g^{-1} \sum_{j} g^{j n}(x) \xi_j = \xi_n\;\; \mbox{on} S^*_H
M,
$$ so $\{x_n, |\xi|_g\} = 0$ defines $S^* H$. Equivalently, we have
\begin{lem} \label{FAILTRANS}   $S^* H$ is the set of points of $S^*_H M$ where $S^*_H M$
fails to be transverse to $G^t$, i.e. where the Hamilton vector
field $H_g$ of $g = |\xi|$ is tangent to $S^*_H M$. \end{lem}
Indeed, this happens when $H_g(f) =df (H_g) = 0$. One may also see
it in Riemannian terms as follows:  the generator $H_g$ is the
horizontal lift $\eta^h$ of $\eta$ to $(q, \eta)$ with respect to
the Riemannian connection on $S^* M$, where we freely identify
covectors and vectors by the metric. Lack of transversality occurs
when $\eta^h $ is tangent to $T_{(q, \eta)} (S^*_H M)$. The latter
is the kernel of $d y_n$. But $d y_n (\eta^h) = d y_n (\eta)= 0 $
if and only if $\eta \in T H$. We also note that for any
hypersurface $H$, $d y_n, d \xi_n, d |\xi|_g$ are linearly
independent.

 Two closely related  restriction maps will be important. The
 first is the linear restriction map $ \pi_H: T^*_H M \to T^* H $
 defined in (\ref{piHdef}). If we orthogonally decompose $T^*_H M
 = T^* H \oplus N^* H$, then $ \pi_H $ is the orthogonal
 projection with respect to this decomposition. It is a fiber
 bundle with fiber $ N_s^* H $. On the other hand, we consider the restriction map on $S^*_H M \to
B^* H$.  For
  $s \in H,$  the orthogonal projection map $ \gamma_H: S^*_s
M \to B_s^* H $ is the standard projection of a sphere to a ball,
which  has a fold singularity along the `equator'.

This fold singularity will play a role in all the canonical
relations to follow, so we pause to recall the definitions (see
\cite{HoI-IV} Vol. III): If $f: Y \to X$ is a smooth map then it
has a Hessian map $H f: \ker f'(y) \to \mbox{coker} f'(y). $ $f$
is said to have a fold at $y_0 \in Y$ if $\dim \ker f'(y) = \dim
\mbox{coker} f'(y) = 1 $ and $H f(y_0) \not= 0$. In this case,
there is a neighborhood of $y_0$ and an involution $\iota: Y \to
Y $ in a neighborhood of $y_0$  which is not the identity such that
$f \circ \iota = f$. It is called the involution defined by the
fold, and its fixed point set is the hypersurface $F$ where $f'$
is not bijective. The involution defines a line bundle in $T F$,
known as the reflection bundle, of eigenvectors of eigenvalue
$-1$. They are transversal to the fixed point hypersurface. If
$\iota$ is defined by a folding map $f$ then the reflection bundle
is $\ker f'$.

In our setting, the full restriction map $\gamma_H: S^*_H M \to
B^* H$ is a folding map  with fixed point set $S^* H$ and
involution given by the reflection map $r_H$ (\ref{rHDEF}). When
$H$ is orientable, $S^*H$ divides $S^*_H M$ into two connected
components, and the involution on $\R \times S^*_H M$ is given by
$r(t, x, \xi) = (t, r_H(x, \xi))$. Indeed, as observed above, this
is true for each $x \in H$, and $D \gamma_H$ is the identity in
the directions tangent to $H$. The
   reflection bundle at $(s, \sigma) \in S^* H$ is spanned by the
   Hamilton vector field $H_{y_n} = \frac{\partial}{\partial
   \xi_n}$. That
is, the reflection bundle is the family of reflection bundles for
the folding maps $ S^*_s M \to B^*_sH $ as $ s \in H $ varies.

We also need the following variant.

\begin{lem} \label{RFOLD} The maps $G: \R \times S^*_H M \to S^* M$ defined by $G: (t,
x, \xi) \to G^t(x, \xi)$, resp.  $G: \R \times T^*_H M - 0 \to T^*
M - 0$ defined by $(t, x, \xi) \to G^t(x, \xi)$, are folding maps
with folds along $\R \times S^* H$, resp. $\R \times T^* H$.
\end{lem}

\begin{proof} In both cases, the spaces are of equal dimension, so
the maps are local diffeomorphisms whenever the derivatives are
injective. By Lemma \ref{FAILTRANS}, $D G
(\frac{\partial}{\partial t} - H_g) = 0$ on $T_{(t, x, \xi)} (\R
\times S^*_H M)$ if $(x, \xi) \in S^* H$, and these are the only
vectors in its kernel. Indeed, suppose $X \in T_{x, \xi} S^*_H M$.
We note that $D G_{(t, x, \xi)} \frac{\partial}{\partial t} = H_g
(G^t(x, \xi))$ and $D G_{t, x, \xi} X$ (as $t$ varies) is a Jacobi
field along the geodesic $\gamma_{x, \xi}(t) = \pi G^t(x, \xi)$.
Since $G^t$ is a diffeomorphism, the only possible elements of the
kernel have the form  $\frac{\partial}{\partial t} + X$.  If $H_g
+ D_{x, \xi} G^t X = 0$, then  $X = - H_g$, i.e. it is the
tangential Jacobi field $\dot{\gamma}$. But by Lemma
\ref{FAILTRANS}, this implies $(x, \xi) \in S^* H$ and $X \in
T(S^* H)$.

Since $G^t$ is homogeneous on $T^*M = 0$ the same statements are
true on $\R \times T^*_H M$.

\end{proof}

\section{\label{IMPACTRETURN} $S^*_H M$ as a cross section to the geodesic flow}

As above, let $H \subset M$ be a smooth hypersurface. The purpose of this section is
to explain the sense in which  $S^*_H
M$ is a  cross-section to the geodesic flow $G^t: S^* M \to
S^* M$ and to discuss the associated return times and return maps. 

 By a cross-section, we mean a hypersurface of $S^*M$ so
that almost every orbit intersects it transversally. As discussed
in the previous section, geodesics which intersect $H$
tangentially in the base also intersect $S^* H$ tangentially in
$S^*_H M$. But $S^*_H M$ behaves sufficiently well as a cross
section so that,  in the ergodic case, almost all orbits hit
$S^*_H M$ and the first return map is almost everywhere defined
and ergodic. The first return map is similar to the billiard map
on the inward pointing unit covectors at the boundary of a domain,
except that we consider covectors pointing on both sides.

 The return
and impact times to $H$  were defined in (\ref{FRTIME}) in the
introduction. Somewhat more precisely, we define for any $(x, \xi)
\in T^*M - 0$, the forward first impact time

\begin{equation} \label{IMPACTTIME} T(x, \xi) = \left\{ \begin{array}{l} \inf\{t > 0: G^t(x,
\xi)  \in T^*_H M\}, \;\;\; \\ \\
+ \infty, \;\ \mbox{if no such t exists}.\end{array} \right.
\end{equation}
We  note that $T$ is  lower semi-continuous. If $(x, \xi) \in
T^*_H M$, then  $T(x, \xi)$ is the first return time of the orbit
$G^t(x, \xi)$ to $T^*_H M$. It may be zero if $H$ contains
geodesic arcs. If $(x, \xi) \in T^* M \backslash T^*_H M$ then
$T(x, \xi)$ is the first impact time (or hitting time) of its
orbit on $H$.  In terms of the notation in (\ref{impact}), $T(x,\xi) = t_1(x,\xi).$

Since $G^{T(x, \xi)}(x, \xi) \in T^*_H M$, the further impact
times $t_j(x, \xi)$ in (\ref{impact}) are the higher return times
of $G^{T(x, \xi)}(x, \xi)$. So it suffices to find the domains on
which the first impact time and the return times on $S^*_H M$ are
well-defined and smooth. We note that $T(x, \xi)$ is homogeneous
of degree zero in $\xi$, so suffices to consider its restriction
to $S^*M$.

 We introduce the sets
\begin{equation} \label{DECAL1}  \left\{ \begin{array}{l} \hcal =
\{(x, \xi) \in S^* M: T(x, \xi) < \infty\}, \\ \\
\lcal = \{(s, \xi) \in S^*_H M: T(s, \xi) < \infty\}. \end{array}
\right.
\end{equation} We refer to the first set as the `hitting set', i.e. the
initial directions of geodesics which intersect $H$ at some time,
and the second as the `return set', i.e. the directions along $H$
of geodesics which return to $H$. We note that $\hcal =
\mbox{Im}\;  G$ is the image of the map $G$ (\ref{G}), and the
natural domain of $T(x, \xi)$.

  On these sets we define the
first impact, resp. first return maps
\begin{equation} \label{IMPACTMAP}  \left\{ \begin{array}{l} \Phi_I:
\hcal\to  S^*_H M, \;\;\; \Phi_I(x, \xi) = G^{T(x, \xi)} (x, \xi), \\
\\ \Phi:
\lcal\to   S^*_H M, \;\;\;  \Phi(s, \xi) = G^{T(s, \xi)} (s, \xi).
\end{array} \right. \end{equation}   We use the same notation for
both maps because they differ only in their domains. The return
map was defined in (\ref{FIRSTRETURN}). The  impact map defines a
kind of fibration
$$\Phi_I: \hcal  \to S^*_H M. $$
Below we use it to describe the geodesic flow as the suspension of
$\Phi$ with height function $T$.

\subsection{Higher return times and impact times}

Once $\Phi(x, \xi) \in T^*_H M$, the further intersections of its
trajectory with $T^*_H M$ come from applying the return maps.  To
obtain invariant sets up to a fixed number of iterates we put
\begin{equation} \label{LCALM} \lcal_{M} = \bigcap_{0\leq  k \leq M}
\Phi^{-k}
 \lcal =  \{(s, \xi) \in S^*_H
 M: (s, \xi) \in \lcal, \Phi(s, \xi) \in \lcal, \dots, \Phi^M(s,\xi) \in \lcal\}.   \end{equation}
 Then the higher return times $T^{(j)}(x, \xi)$ are defined by
 \begin{equation} \label{tjpm} T^{j + 1}(x, \xi) = T(\Phi^j(x, \xi)), \;\; (x, \xi) \in \lcal_j,  \end{equation} and
 finite on $\lcal_M$ for $j \leq M$. We also put
 \begin{equation} \hcal_M = \{(x, \xi) \in \hcal: \Phi_I(x, \xi)
 \in \lcal_M\}, \end{equation} and define
the $j$th impact map by
$$\Phi_j: \hcal_j \to  S_H^*M, \;\;\; \Phi_j(x, \xi)
= G^{t_j(x, \xi)}(x, \xi) = \Phi^{j-1} \Phi_I(x, \xi), $$ and the
jth impact time $t_j(x, \xi) $ by
$$t_j(x, \xi) = T(x, \xi) + \sum_{k = 1}^{j-1} T(\Phi^k  \Phi_I (x, \xi)); \;\;\; t_1 = T. $$ In a similar way, we define
\begin{equation} \label{LCALMpm} \lcal^{\pm}_{M} = \{(s, \eta) \in
B^* H: \xi_{\pm}(s, \eta) \in \lcal_M\}. \end{equation} The
corresponding $\pm$ return times $T_{\pm}^j$ and return maps
$\pcal_{\pm, j}$ in (\ref{pcaldefj}) are well-defined on this set
for $j \leq M$.

Although we will not need it, it is useful to think of the
invariant sets
$$ \lcal_{\infty} = \bigcap_{k \in \Z} \Phi^{-k}
 \lcal, \;\;\; \hcal_{\infty} = \{(x, \xi) \in S^* M: \Phi_I(x, \xi) \in \lcal_{\infty}\}.  $$

The geodesic flow on $S^*M$ then becomes the suspension flow over
$\Phi$ with height function $T$, i.e. up to a set of measure zero,
$$S^* M = \{((s, \xi), t) \in \lcal_{\infty} \times \R: 0 \leq t \leq
T(s, \xi) \}/ \{(s, \xi, T(s, \xi)) = (\Phi(s, \xi), 0)\}, $$ and
the under this identification,   when $ T^{(n)}(s,\xi) \leq t'+t \leq T^{(n+1)}(s,\xi),$
$$G^t(s, \xi, t') = (\Phi^n(s, \xi), t' + t - T^{(n)}(s, \xi)).$$
In practice we only use a finite number of iterations of $\Phi$.
But this identification arises naturally in parameterizing the
canonical relation of $\VT(a)$.

To illustrate these notions, consider the case where $H$ is a
distance  circle $S_r$ in a hyperbolic surface $\X$, or more
precisely a distance circle in the hyperbolic disc projected to
$\X$. Since the geodesic flow is ergodic, and the circle is a
separating curve, the complement of $\hcal$ must have measure
zero in $S^* M$. If the radius is suffciently large so that the
circle surrounds a fundamental domain, then every geodesic must
intersect the circle and $\hcal = S^*M$  (since every geodesic
must pass through the fundamental domain). If the radius is small
enough, then there may exist closed geodesics of $\X$ which do not
pass through the circle and $\hcal \not= S^*M$.  As another
example, let $H$ be a closed geodesic $\gamma$ of $\X$. Again,  $|\hcal| =0$ since otherwise there would exist a $G^t$-invariant set of non-trivial Liouville measure. However, $\hcal \neq  \emptyset$ since
there exist geodesics which are forward asymptotic to $\gamma$ as
$t \to \infty$ (i.e. they spiral in towards $\gamma$ as $t \to
\infty$) but never cross $\gamma$.  Such geodesics belong to the
one-parameter family with the same forward boundary point as
$\gamma$ on the ideal boundary of the hyperbolic disc. We note
that $\gamma$ itself belongs to this one-parameter family, so the
set $\hcal$ is not closed.

\subsection{$\Phi: S^*_H M \to S^*_H M$ as a symplectic map} \label{poincare cartan}

Let $\alpha =\xi \cdot dx$ denote the canonical one-form of $T^*
M$, and let $d \alpha = \omega_{T^* M}$ be the canonical
symplectic form on $T^* M$.

We note that $\omega_{T^* M}$ restricts to $S^* M$ as a form with
a one-dimensional kernel, spanned by the Hamilton vector field
$H_g$ of the metric norm function. Since $S^*_H M$ is transverse
to $H_g$ except on $S^* H$ , $\omega_{T^* M} |_{S^*_H M}$ is
symplectic away from $S^* H$. Indeed, in Fermi coordinates it is
the form $ \gamma_H^* (ds \wedge d \sigma)$, the symplectic form of the ball
bundle $B^* H$, pulled back to $S^*_H M$ under the tangential
projection $\gamma_H: S^{*}_H M \rightarrow B^*H.$

We further note that $\alpha$ restricts to the one form $\alpha_H
  := \sigma \dot d s$ on $T^*_H M$, since $d y_n = 0$ on $T(T^*_H
  M)$. As discussed in \cite{FG}, $\Phi$ is symplectic with
  respect to $ds \wedge d \sigma$ and moreover
  $$\Phi^* \alpha_H - \alpha_H = d T, $$
  on $S^*_H M$, where as above $T$ is the return time function (the
  `Poincar\'e-Cartan identity, see (2.8) of \cite{FG}).

The symplectic volume
density $|\gamma_H^* ds \wedge d \sigma|$ is, strictly speaking,  not a
volume form since it vanishes on $S^*H,$ but we use it as an
invariant volume density. It may also be defined as follows:

\begin{defin} \label{LIOUVILLE}  We define  the Liouville volume measure $d\mu_{L,H}$ on $S^*_H M$ by
$d\mu_{L, H} = \iota_{H_g} d\mu_L,$   i.e. by inserting the Hamilton
vector field generating $G^t$ into $d\mu_L$.
\end{defin}

In terms of local Fermi symplectic coordinates, $d \mu_{L, H} = |\gamma_H^*ds \wedge d\sigma|.$

\begin{lem} \label{MULMUH}  On $\hcal,$  $d \mu_L = d T \wedge (\Phi_I^* d \alpha_H)^{n-1}.$   \end{lem}

\begin{proof} We recall that $d\mu_L = \alpha \wedge (d
\alpha)^{n-1}$ on $S^* M$. We claim that $\alpha = dT + \Phi_I^*
\alpha_H$ on $\hcal$. Indeed, since $G^t: T^*M-0 \rightarrow T^*M-0$ is homogeneous,  $(G^t)^* \alpha = \alpha.$   Also,  for the $G^t-$translated hitting time $    (G^t)^* dT = dT$ and so, it follows that
$(G^t)^* (dT + \Phi_I^* d\alpha_H) = (dT + \Phi_I^* d\alpha_H) $.
Therefore it suffices to show that $\alpha = d T + \Phi_I^*
\alpha_H$ at points of $S^*_H M$ where $\Phi_I = \mbox{id}$ and
$\alpha =  \eta_n dy_n + \sigma ds  = \eta_n dy_n + \alpha_H$. But $d T = \eta_n dy_n$ on
$S^*_H M$.  Indeed,  $\alpha(H_g) = 1$, since at $(x,
\xi) \in S^*M$, $H_g$ is the horizontal lift of $\xi$ to $(x,
\xi)$ and so $\alpha (H_g) = \xi (\pi_* H_g) = |\xi|_g^2 = 1$.  On
the other hand, $d T(H_g)(s,\xi) = 1 $ for all $(s,\xi) \in S_H^*M$ by definition. Moreover, both $d T$
and $\eta_n dy_n $ annihilate $T(T^* H)$ and so,
$$ \eta_n dy_n |_{S_H^*M} = dT |_{S_H^*M}.$$
It follows that $d\mu_L = (dT + \Phi_I^* \alpha_H) \wedge
(\Phi_I^* d \alpha_H)^{n-1} =  dT \wedge (\Phi_I^* d
\alpha_H)^{n-1}. $

\end{proof}

\subsection{Singularities of return times}

We now define smooth local branches of the return time functions.
Let $f = y_n: M \rightarrow \R$ be a local
 defining function for $H$ in $U$,  so that $H \cap U= \{ f = 0 \}$ and $df(x) \neq 0, x \in H.$
Then $d f: T M \to \R$ is a local defining function of $T H$. We
consider the maps \begin{equation} \label{FMAPS} \left\{
\begin{array}{l} (i)\;\; F: \R \times S^*_H M \to \R, \;\;\;  F(t,
s, \xi) = f(G^t(s, \xi)) = f(\exp_s t \xi),  \\ \\\;\;\;(ii)\;
F^{\pm} : \R \times B^* H \to \R, \;\;\;  F^{\pm}(t, s, \sigma) =
F(t, \xi_{\pm}(s, \sigma)).
\end{array} \right.
\end{equation} $F$ extends by homogeneity to $T^*_H M$. Here, as always, we
define $\exp_q t \eta = \pi G^t(q, \eta)$ with $G^t$ the
homogeneous geodesic flow.  The graph of the impact times $t_j(x,
\xi)$ (see (\ref{RCAL})) is given by
$$\tcal_I =   \{(t, s, \xi) \in \R \times T^*_H M:  F(t, s, \xi) =
0   \}, $$ and those of the $\pm$  return time functions (see
(\ref{pcaldefj}))  are given by
$$\tcal_{\pm}: =  \{(t, s, \sigma) \in \R \times B^* H: F (t,  \xi_{\pm}(s,\sigma)) =
0\}.
$$
Consider the diagram
\begin{equation}\label{TIMEDIAGRAM} \begin{array}{ccccc} & & \tcal_I \subset
\R \times T^*_H M
 & &  \\ & & & & \\
& \pi(t,x, \xi) = t \swarrow & & \searrow \rho(t, x, \xi) = G^t(x, \xi) & \\ & &   \\
& \R  & &   T^*_H M .
\end{array}
\end{equation}
The  set of impact times of $(x, \xi)$ is thus given by
$$\pi \rho^{-1} (x, \xi): \{(t, G^{-t}(x, \xi)) \in \tcal\} \to t \in  \R. $$
We are interested in the extent to which $\rho: \tcal_I \to T^* M$
is an (infinite sheeted) covering map.

\begin{lem}  \label{REGU} $0$ is a regular value of $F$, so $\tcal_I$ is always
a submanifold of $\R \times T^*_H M$. Let $\Sigma(F) = \{(t, x,
\xi) \in \R \times T^*_H M: \partial_t F(t, x, \xi) = \partial_t
f(exp_x t \xi) = 0\} $ Then $\Sigma(F) \subset \R \times T^* H$.
The image of $\Sigma(F)$ under $\rho (t, x, \xi) = G^t(x, \xi)$ is
$\gcal$ (\ref{tangential}).
 \end{lem}

\begin{proof} Since $d F = \partial_t F dt + df \circ DG^t, $
since $df \not= 0$ (by definition of a defining function) and $D
G^t \not= 0$, it follows that $d F \not= 0$ and $0$ is a regular
value.

The fact that $\Sigma(F) \subset \R \times T^* H$ follows from
Lemma \ref{FAILTRANS}. Indeed, $\Sigma(F)$ is defined by $d
f(G^t(x, \xi)) H_g = 0$ and that implies $(x, \xi) \in T^* H$.

\end{proof}

\subsection{Cutoffs and domains}

 We often fix $T > 0$, and then the
image of $G$ in Lemma \ref{RFOLD} restricted in time to $[0, T]$
is the closed set
\begin{equation} \label{DECALT}   \begin{array}{l} \hcal^T =
\{(x, \xi) \in S^* M: T(x, \xi) < T\}.
\end{array}
\end{equation}
We also put $\lcal^T = \{(s, \xi) \in S^*_H M: T(s, \xi) < T\}.$
The set  $\{T = \infty\}$ is in general neither open nor closed;
of course it is the countable intersection  $\bigcap_{n =
1}^{\infty} \{ T > n\}$ of open sets. Given $T > 0$ we define
\begin{equation} \label{HCALMT} \hcal_{ M}^T = \{(x, \xi) \in
 \hcal^T :
  G^{T(x, \xi)}(x, \xi) \in \lcal, \dots, \Phi^M G^{T(x, \xi)}(x, \xi) \in \lcal\}.   \end{equation}

We will need to cut off tangential directions to $H$. We put
$$\left\{ \begin{array}{l} (S^*_H M)_{\leq \epsilon} = \{(x, \xi)
\in S^*_H M: |\langle x, \eta \rangle| \leq
\epsilon, \forall \eta \in S^*_x H\}\\ \\
\;(S^*_H M)_{\geq \epsilon} = \{(x, \xi) \in S^*_H M: |\langle x,
\eta \rangle| \geq \epsilon, \forall \eta \in S^*_x H\}
\end{array} \right. $$ i.e. the covectors which make an angle $\leq \epsilon$, resp. $\geq \epsilon$
with $H$.  We homogenize by defining
\begin{equation} \label{TEP} \left\{ \begin{array}{l} (T^*_H M)_{\leq \epsilon} = \{(x, \xi) \in
T^*_H M: \frac{\xi}{|\xi|} \in (S^*_H M)_{\leq \epsilon}\}, \\ \\
(T^*_H M)_{\geq \epsilon} = \{(x, \xi) \in T^*_H M:
\frac{\xi}{|\xi|} \in (S^*_H M)_{\geq \epsilon}\}.
\end{array} \right. \end{equation}

We then define $\gcal_{T, \epsilon}$ to be the flowout of the tube $(T^*_H M)_{\epsilon}, $ i.e.
\begin{equation} \label{GCALEPDEF} \gcal_{T, \epsilon} = \bigcup_{|t| \leq T}
G^t(T^*_H M)_{\epsilon}. \end{equation} We also denote the
complement of a set $E$ by $E^c$. By Lemma \ref{REGU} and an application of the implicit function theorem,  we have

\begin{cor} For any $T, \epsilon > 0$, $T(x,\xi)$ is a smooth
function on the open set  $\hcal_T \cap (\gcal_{T, \epsilon})^c$.
\end{cor}

We also need to define the domains $\dcal_{T,\epsilon}^{(j)} \subset T^*M$ of the $j$-th impact times $t_j(x,\xi):$

\begin{defin} \label{IMPACTTIMES} We define
 $$ \dcal_{T, \epsilon}^{(j)} = \{(x, \xi) \in
\hcal_j^T  \cap (\gcal_{T, \epsilon})^c,   G^{T(x, \xi)}(x, \xi)
\in \lcal \backslash (S^*T_H)_{\leq \epsilon}, \dots, \Phi^j
G^{T(x, \xi)}(x, \xi) \in
 \lcal  \backslash (S^*T_H)_{\leq \epsilon}\} $$

\end{defin}

As above,  $\dcal_{T, \epsilon}^{(j)}$   is a conic open subset of
$T^*M$ with  $t_j \in C^{\infty}(\dcal_{T,\epsilon}^{(j)}).$

\subsection{Ergodicity of the return map}

As mentioned  in subsection \ref{poincare cartan},  $\Phi$ is a $d\mu_{L,H}$ and $ |ds d\sigma|$-measure preserving transformation
on $\lcal_{\infty}$.
\begin{lem}\label{IMG}  For any $T$, the image $G \left( (- T, T) \times
(T^*_H M \backslash T^* H \right)$ is an  open homogeneous set  in
$T^* M$. If $G^t$ is ergodic, then
$$   \limsup_{T \rightarrow \infty} | (S^* M \backslash (G \left( (- T, T) \times (S^*_H M
\backslash S^* H \right)| =  0.  $$
\end{lem}

\begin{proof} $G$ is an open map on the given domain since $D G$
is everyhwere non-singular. The complement of the image is
obviously decreasing. If its volume were bounded below by some
$\delta > 0$, the complement of the image $G \left( \R \times
(S^*_H M \backslash S^* H \right) )$  would be a closed invariant
set of positive measure for $G^t$, contradicting its ergodicity.

\end{proof}

Since $\mu_H(S^*_H M \backslash \lcal_{\infty}) = 0$ we also
regard it as a measure preserving transformation on $S^*_H M$. We
add the obvious comment that in the ergodic case, almost all
geodesics hit $H$.  Since $G^t$ is the suspension of $\Phi$, we have

\begin{lem} \label{ERG} The  return map $ \Phi:
S^*_H M \rightarrow S^*_H M$ is ergodic on $S^*_H M$ with respect
to $d\mu_{L, H}$  if and only if
 $G^t$ is ergodic on $S^*M$ with respect to $d\mu_L$. \end{lem}

  \begin{proof}

We have just seen that $d\mu_{L, H}$ is an invariant measure for
$\Phi$. If there exists an invariant set $A \subset S^*_H M$ with
$0 < \mu_{L, H}(A)  < 1$, then its flowout, $\bigcup_{s \in \R} G^s A,$
is an
 invariant set for $G^t$ and by Lemma \ref{MULMUH} it satisfies $0 < \mu_L(A) < 1$.
This contradicts ergodicity of $G^t$. The converse is similar.

   \end{proof}

We also need an $\epsilon$-refinement of Lemma \ref{IMG}:

\begin{lem}\label{PREV}  For any $T, \epsilon$,  $G \left( (- T, T) \times (T^*_H M \backslash (T^*_H
M)_{\leq \epsilon})  \right)$ is an  open homogeneous set  in $T^*
M$. If  $G^t$ is ergodic then,
$$ \limsup_{T \to \infty} | (S^* M \backslash (G \left( (- T, T) \times (S^*_H M
\backslash (S^* H)_{\leq \epsilon} \right)| = o(1) \;\; \mbox{as}
\;\; \epsilon \to 0;
$$
Similarly,
$$ \limsup_{T \to \infty} | (S^* M \backslash \left(G  (- T, T) \times (S^*_H M
 \right)\;\; \backslash \gcal_{T, \epsilon}| = o(1) \;\; \mbox{as}
\;\; \epsilon \to 0.
$$

\end{lem}

\begin{proof} In each case the image in question is the image of
an open set, hence open. For the volume estimates, we note that as in Lemma \ref{IMG},  by $\mu_L$-ergodicity of $G^t:S^*M \rightarrow S^*M,$
the union $\bigcup_{T, \epsilon > 0} G ( (-T,T) \times (S_H^*M/S^*H)_{\leq \epsilon} )$ of the image has full
measure, hence the measure of the complement is zero. A similar argument applies to the second set in Lemma \ref{PREV}.

\end{proof}

\section{\label{COMPOSITIONS} Compositions of  canonical relations}

In order to study the Fourier integral properties of $ V_{\epsilon}(t; a)$ and
$\VT(a)$, we need to understand the compositions of the canonical
relations underlying various operators. In essence we prove here that the
canonical relation $\Gamma_W^* \Gamma_W$ (cf. \eqref{GAMMAW})
is a  local canonical graphs,  determine the graph and relate it to the first return
times and maps. We choose to work with operator kernels on $M \times M$ rather
than with $W$ itself.

We refine \eqref{SIGMA} to

\begin{equation}\label{SIGMATEPDEF}  \Sigma_{T, \epsilon} = \bigcup_{|t| \leq T}
 ( G^t\times G^t ) \Delta_{(T^*_H M)_{\leq \epsilon} \times (T^*_H
M)_{\leq \epsilon}},
\end{equation}
and put
\begin{equation} \label{GAMMATEPDEF} \Gamma_{T, \epsilon} =
\Gamma_T \backslash \Sigma_{T, \epsilon}. \end{equation} In this section we prove that \eqref{WFVT},
cutoff away from the singular set, is a good canonical relation:

\begin{prop}\label{CANON}  For any $\epsilon > 0$,  $\Delta_{T^* M \times T^* M} \cup\Gamma_{T, \epsilon} \subset T^* M \times T^* M$  is
smoothly immersed homogeneous canonical relation. \end{prop}

The self-intersection locus is described in Lemma \ref{PUSHPULL}.
In the next section \S \ref{RCALSECTION} we show that it is a local canonical graph and
determine the branches.

We recall that a canonical relation is a Lagrangian submanifold
with respect to the difference symplectic form $\pi_1^*
\omega_{T^*M } - \pi_2^* \omega_{T^*M}$ where $\omega_{T^*M}$ is
the canonical symplectic form and $\pi_{k}: T^*M \times T^*M \rightarrow T^*M; k=1,2$ are the projecttions onto the two component $T^*M.$ We prove the Proposition as a
series of Lemmas. The final Lemma \ref{PUSHPULL} is more precise
and describes the singularities at $\epsilon = 0$.

\subsection{\label{CHSECTION} The canonical relation $C_H$}

We define \begin{equation} \label{CH} \left\{ \begin{array}{l} C_H
= \{(s, \xi, s, \xi') \in T^*_H M \times T^*_H M: s \in H, \,
\xi|_{TH} = \xi'|_{TH} \}, \\ \\
 \hat{C}_H
= \{(s, \xi, s, \xi') \in T^*_H M \times T^*_H M: s \in H, \,
\xi|_{TH} = \xi'|_{TH}, \;\;\; |\xi| = |\xi'| \}, \\ \\
S \hat{C}_H = \{(s, \xi, s, \xi') \in T^*_H M \times T^*_H M: s
\in H, \, \xi|_{TH} = \xi'|_{TH}, \;\;\; |\xi| = |\xi'| = 1 \}
\end{array} \right.
\end{equation} As above,  $S F$ denotes  the unit vectors
in any set $F$. Thus, $\hat{C}_H = \R_+ S \hat{C}_H$.
 As will be
seen below, $C_H$ is the canonical relation of $\gamma_H^* Op_H(a)
\gamma_H$, and $\hat{C}_H$ arises in the canonical relation of
$\VT(a)$.

 We  recall that
the fiber product of two fiber bundles $\pi: X \to Z$ and $\rho: Y
\to Z$ is the submanifold $X \times_Z Y \subset X \times Y$ equal
to $(\{(x, y): \pi(x) = \rho(y)\}. $ We apply the same terminology
with $X = Y = S^*_H M$, $Z = B^*H$ and   $\pi, \rho = \gamma_H$,
but as just observed, the restriction map is not a fiber bundle
projection but a folding map.

\begin{lem}\label{CSUBH} We have:

\begin{itemize}

\item  $C_H \simeq T^*_H M \times_{T^*H} T^*_H M$ is the fiber
 square of $T^*_H M$  with respect to the
 restriction map $\gamma_H: T^*_H M \to T^* H$. It is an embedded
Lagrangian submanifold  of $T^* M \times T^* M$.

\item $\hat{C}_H := \R S\hat{C}_H \simeq T^*_H M \times_{S^*H}
T^*_H M$ is an immersed homogeneous isotropic  submanifold of
dimension $2n - 1$ with transveral crossings on the
self-intersection locus $\R_+ \Delta_{S^*H \times S^* H} =
\Delta_{T^* H \times T^* H}$. Also, $\hat{C}_H \cap (T^* H \times
T^* H) = \Delta_{T^* H \cap T^* H}. $

\item $S\hat{C}_H \simeq S^*_H M \times_{S^*H} S^*_H M$ is the
`fiber  square' of $S^*_H M$  with respect to the (folding)
restriction map $\gamma_H: S^*_H M \to S^* H$. It is an immersed
isotropic  submanifold of dimension $2n - 2$ with transveral
crossings on the self-intersection locus $\Delta_{S^*H \times S^*
H}$.

\end{itemize}
\end{lem}

\begin{proof}

The defining equations of $C_H \subset T^*_H M \times T^*_H M$ are
given by equating the map $(v,w) \to v |_{T H} - w |_{TH} \in T^*
H$ to zero. This map is a submersion.  Suppressing the $s \in H$ variable, it is just the map
$(\sigma, y_n, \sigma', y_n') \to \sigma - \sigma'$ with $\sigma,
\sigma' \in \R^{n-1}, y_n \in \R$. Thus,  the zero set is a
regular level set, hence an embedded submanifold of codimension $n
= \dim M$.

 We observe that $S\hat{C}_H$  is the union
$S\hat{C}_H = \mbox{gr}(Id) \cup \mbox{gr}(r_H)$ of the identity
and reflection maps. The graphs  intersect transversally along the
diagonal $\Delta_{S^* H \times S^* H} \subset S^*H \times S^* H$,
since the tangent space to the identity graph is the diagonal and
the tangent space to the reflection map is the  `anti-diagonal'
$(v, - v) \in T (S^* H \times S^* H)$. That is, the equation
$\pi_H(\zeta, \zeta') : = \gamma_H(\zeta) - \gamma_H(\zeta') = 0$
in $S^*_H M \times S^*_H M$ defines a submanifold of codimension
$n - 1$  on the dense open set where  $D_{\zeta} \gamma_H,
D_{\zeta'} \gamma_H$ spans $ T B^*H $. Suppressing the variable
along $H$, the singularities at each $x \in H$ are those of the
map $  \pi: S^{n-1} \times S^{n-1} \to \R^{n-1}, \pi(\sigma, y_n; \sigma',
y_n') = \sigma - \sigma', $ where $ (\sigma, y_n)  \in \R^{n-1} \times \R,   |\sigma|^2
+ y_n^2 = 1.$  Thus, $y_n =   \pm \sqrt{1-|\sigma|^2} $ and $\pi^{-1}(0) =  \{\sigma, y_n,
\sigma, y_n)\} \cup \{(\sigma, y_n, \sigma, - y_n)\}.$ Here, we fix $ s \in H$
and identify $T^*_s M \simeq \R^n, T^*_s H \simeq \R^{n-1}.$

 Since $\R_+ S
\hat{C}_H $ is the homogenization, we only need to homogenize the
results for $S \hat{C}_H$. In more detail, we again fix $x$ and
consider the map $\pi (r, s, y_n, s', y_n') = r (s - s') $ from
$\R_+ \times S^{n-1} \times S^{n-1} \to \R^{n-1}$. The zero set is
again defined by $s = s'$. The radial tangent direction is in the
kernel of $D \pi$ along $\pi^{-1}(0)$. Finally, we note that if
$(x, \xi, x, \xi') \in \hat{C}_H \cap (T^* H \times T^* H)$, then
$\xi = \xi'$.

\end{proof}

\subsection{\label{GCHG} The canonical relation $\Gamma^* \circ C_H \circ \Gamma$}
It is well known (see \cite{HoI-IV}, vol. IV)  that $U(t) \in
I^{-\frac{1}{4}}(\R \times M \times M, \Gamma), $ with $\Gamma =
\{(t, \sigma, x, \xi, G^t(x, \xi)): \sigma + |\xi| = 0\}. $ As in
\cite{DG}, the half density symbol of $U(t, x, y)$ is the
canonical volume half density $ \sigma_{U(t, x, y)} = |dt \otimes
dx \wedge d \xi |^{\frac{1}{2}}$ on $\Gamma$.

Here, \begin{equation} \label{GammaCHGamma}\begin{array}{lll}
 \Gamma^* \circ C_H \circ \Gamma  &= &   \{(t',-|\xi'| ,t, |\xi|, G^{t'}(s,
\xi'), G^{t}(s, \xi))  \\ && \\ && \in T^* \R \times T^* \R \times T^*M \times T^* M, \, , \, \xi|_{TH} = \xi'|_{TH}
 \}.
\end{array}\end{equation}

\begin{lem}\label{GammaCHGammaLEM}  The (set-theoretic) composition $\Gamma^* \circ C_H \circ \Gamma
$ is transversal, and  $\Gamma^* \circ C_H \circ \Gamma \subset T^*
\R \times T^* M \times T^* M $ is  the Lagrangian submanifold  parametrized by the embedding  $$\begin{array}{l} \iota_{\Gamma^* C_H \Gamma}:
\R \times \R \times T^*_H M \to T^* (\R \times \R \times M \times
M), \;\; \\ \\\iota_{\Gamma^* C_H \Gamma}(t', t, s, \xi, \xi') =
(t', - |\xi'|, t, |\xi|, G^{t'}(s, \xi), G^t(s, \xi') ), \,\,  \xi|_{TH} = \xi'|_{TH}.
\end{array}$$

\end{lem}

\begin{proof} This follows from the following observation: if
$\chi: T^*M - 0 \to T^*M - 0$ is a homogeneous canonical
transformation and $\Gamma_{\chi} \subset T^*M \times T^*M$ is its
graph, and if $\Lambda \subset T^*M \times T^*M$ is any
homogeneous Lagrangian submanifold with no elements of the form
$(0, \lambda_2)$, the $\Gamma_{\chi} \circ \Lambda$ is a
transversal composition with composed relation
$\{(\chi(\lambda_1), \lambda_2): (\lambda_1, \lambda_2) \in
\Lambda\}. $ The condition that $\lambda_1 \not= 0$ is so that
$\chi(\lambda_1)$ is well-defined.

 We recall that transversality refers to the intersection
$$\Gamma_{\chi} \times \Lambda \cap T^* M \times \Delta_{T^*M \times T^* M}
\times T^* M. $$ Now, the tangent space at any intersection point
to $T^* M \times \Delta_{T^*M \times T^* M} \times T^* M$ contains
all vectors of the form $(v, 0, 0, 0)$ and $ (0, 0, 0, v') $ with $ v,v'
\in T(T^* M)$. Hence to prove transversality it suffices to fill
in the middle two components. The diagaonal $T \Delta_{T^*M \times
T^* M}$ contributes all tangent vectors of the form $(w,w)$. On
the other hand, the middle components of $\Gamma_{\chi} \times
\Lambda$ have the form $(w, \delta \lambda_1)$ where $w \in T(T^*
M)$ is arbitrary. The sum of such vectors with the diagonal
contains all vectors of the form $(w + v, \delta \lambda + v' )$ and
therefore clearly spans the middle $T(T^*M \times T^*M)$.

We apply this observation  in two steps. First, we compose
$$  C_H  \circ \Gamma = \{(s, \xi', G^t(s, \xi), t, -|\xi|): (s, \xi)
\in T^*_H M, \xi|_{TH} = \xi'|_{TH} \}  \subset T^* M \times T^*M \times T^* \R \backslash
0. $$ By the first part of Lemma \ref{CSUBH}, $C_H$ is a
Lagrangian submanifold, so the argument about graphs applies to
show that this composition is transversal (including the innocuous
$T^* \R$ factor.) We then apply the same argument to the left
composition with $\Gamma$. It is straightforward to determine the
composite as stated above.

\end{proof}

\subsection{\label{PBDELTA} The pullback   $\Gamma_H : = \Delta_t^*\;
 \Gamma^* \circ C_H \circ \Gamma$}

We now consider the pullback of $\Gamma^* \circ C_H \circ \Gamma$
under the time diagonal embedding $\Delta_t(t, x, y) = (t, t, x,
y): \R \times M \times M \to \R \times \R \times M \times M $. We
define \begin{equation} \label{GtCH} (G^t \times G^t) (C_H) =
\{(G^t(s, \xi), G^t(s, \xi')): (s, \xi, s, \xi') \in C_H \},
\end{equation}
and
 \begin{equation} \label{GAMMAH} \begin{array}{ll} \Gamma_H & :=
\{t,  |\xi| - |\xi'|, (G^t(s, \xi), G^t(s, \xi'))
 \in T^* \R \times   (G^t \times G^t)(C_H))\}.  \end{array}\end{equation}

\begin{lem}\label{DELTATRANS} The map $\Delta_t$ is transversal to $(\Gamma^*
\circ C_H \circ \Gamma)$, hence $$\Delta_t^* (\Gamma^* \circ C_H
\circ \Gamma) = \Gamma_H$$ is a smoothly embedded canonical
relation, under the Lagrange embedding
$$\begin{array}{l} \iota_{\Gamma_H}:
\R \times   T^*_H M \to T^* (\R \times M \times M), \;\; \\
\\\iota_{\Gamma_H}(t, s, \xi, \xi') =   (t, |\xi| - |\xi'|, G^t(s, \xi), G^t(s, \xi') ), \,\,  \xi|_{TH} = \xi'|_{TH}. \end{array}$$
\end{lem}

\begin{proof}   The explicit formula for the composition is simple
to verify.  We recall  that a map $f: X \to Y$ is said to be
transversal to $W \subset T^* Y$ if $df^* \eta \not= 0$ for any
$\eta \in W$. By (see \cite{GuSt}, Proposition 4.1), if   $f: X
\to Y$ is smooth and $\Lambda \subset T^* Y$ is Lagrangian, and if
$f:X \rightarrow Y$ and $\pi|_{\Lambda}: \Lambda \rightarrow Y$ are transverse then $f^* \Lambda \subset T^*X$ is
Lagrangian. It is clear from the explicit formula for the pullback
that transversality holds.

 Since $G^t \times G^t$ is a homogeneous
diffeomorphism, $G^t \times G^t(C_H)$ is a smooth embedded
manifold, and the map $\iota_{t, C_H} :  T^*_H M \times_{T^*H}
T^*_H M \to G^t \times G^t(C_H) \subset T^*M \times T^* M$ is a
smooth embedding.

\end{proof}

\subsection{\label{PF} The pushforward  $\pi_{t*} \Delta_t^* \Gamma^* \circ C_H \circ \Gamma$}

We now consider the map $\pi_t: \R \times M \times M \to M \times
M$ and push forward the canonical relation $ \Delta_t^* \; \Gamma^*
\circ C_H \circ \Gamma$. We recall that $\VT(a)$ is cutoff in time
(by  $\chi_T$) to $|t|\leq T$ and thus  \eqref{WFVT}, 
\begin{equation} \label{PushGtCH}  \begin{array}{lll} \Delta_{T^* M \times T^*M} \cup\Gamma_T &
=& \bigcup_{|t| \leq T} \{(G^t(s, \xi), G^t(s, \xi')): (s, \xi, s,
\xi') \in C_H, |\xi| = |\xi'| \} \\ && \\
& = &  \bigcup_{|t| \leq T} (G^t \times G^t) \hat{C}_H .
\end{array} \end{equation}
is the   proper pushforward
\begin{equation} \label{GAMMATDEF} \Gamma_T = \pi_{t*} \Delta_t^* \Gamma^* \circ C_H \circ \Gamma, \;\;\; \pi_t: [-T, T] \times M \times M \to M \times
M. \end{equation} Of course, the sharp cutoff to $[-T, T]$ puts a
boundary in $\Gamma_T$, but it causes no problems since all of our
operators are smooth in a neighborhood of the boundary and since
we use the smooth cutoff $\chi(\frac{t}{T})$ in the definition of
$\bar{V}_T$.

 We recall that the pushforward of $\Lambda \subset T^*
X$ under a map $f: X \to Y$ is defined by $f_* \Lambda =\{(y,
\eta): \exists x, \;\;y = f(x), (x, f^* \eta) \in \Lambda)\}. $ As
discussed in  (\cite{GuSt}, Proposition 4.2,page 149), if  $f: X
\to Y$ is a smooth map of constant rank and $H^*(X) $ is the
bundle of horizontal covectors, and if $\Lambda \cap H^*(X)$ is
transversal then $f^*(\Lambda)$ is a Lagrangian submanifold. Here,
$H^*(X) = f^* T^* Y$ is the set of covectors which annihilate the
tangent space to the fibers.

 In
our setting, $\pi_t^* T^*(M \times M)$ is the co-horizontal space
$H^* \subset T^*(\R \times M \times M)$ which is co-normal to the
fibers of $\pi_t$, i.e. its elements have the form $(t, 0, x, \xi,
y, \eta)$. Let $\tau: T^* \R \times T^* M \times T^* M \to \R$ be
the projection onto the second component of $T^*\R = \{(t, \tau)
\}.$ Thus,  \begin{equation} \label{INTER}  \Gamma_H \cap H^*(M \times M) = \Delta_t^*
\Gamma^* \circ C_H \circ \Gamma \cap H^*(M \times M) = \{z \in
\Delta_t^* \Gamma^* \circ C_H \circ \Gamma: \tau(z) = 0\},
\end{equation} and the pushforward relation is
Note that  $\bigcup_{|t| \leq T} G^t(T^*_x M)$ projects (for small
t) to $M$ to the ball of radius $t$ around $x$.

 By Lemma \ref{CSUBH}, (\ref{PushGtCH}) is the flow-out of an
immersed Lagrangian submanifold with transversal crossings on
$\R_+ \Delta_{S^*H \times S^*H}. $ Equivalently, the  pushforward
 relation is parameterized by the Lagrange mapping
\begin{equation} \label{IOTA} \iota: \R \times \hat{C}_H \to T^*M
\times T^* M:   (t, s, \xi, \xi') \mapsto (G^t(s, \xi), G^t(s,
\xi')).  \end{equation}

The following Lemma is the final step in the proof of Proposition
\ref{CANON}, and indeed is more precise than necessary for the
proof.

\begin{lem} \label{PUSHPULL}  We have,

\begin{itemize}

\item  $d \tau \not= 0$ on (\ref{INTER}) except on the  set of
points of $\R \times \Delta_{S^* H \times S^*H}$. Consequently,
(\ref{INTER}) is a smooth manifold except at these points and the
pushforward $$\pi_{t * } \Delta_t^* \left(\Gamma^* \circ C_H \circ
\Gamma \backslash T^* \R \times \R^+(\Delta_{S^* H \times S^*
H})\right)$$
  is an (immersed) Lagrangian submanifold.

\item   $\iota|_{ \R \times (\hat{C}_H \backslash \R \Delta_{S^* H
\times S^*H})}$ is a Lagrange immersion, with self-intersections
corresponding to `return times'.

\end{itemize}

 \end{lem}

 \begin{proof}

As noted above,  if $\Gamma_H =  \Delta_t^*\; \Gamma^* \circ C_H
\circ \Gamma$ intersects $0_{\R} \times T^*M \times T^*M$
transversally, then $\pi_{t*} \Delta_t^* \Gamma^* \circ C_H \circ
\Gamma$ is Lagrangian. Since $H^*(M \times M)$ is of co-dimension
one, $\Delta^* \; \Gamma^* \circ C_H \circ \Gamma$ fails to be
transverse at an intersection point only if its tangent space is
contained in $T(H^*(M \times M))$. Thus, it fails to be transverse
only at points where $d \tau = \tau = 0$. Since $ \tau(t,s,\xi,\xi') = |\xi| - |\xi'| =
\sqrt{\sigma^2 + \eta_n^2} - \sqrt{\sigma^2 + (\eta_n')^2}, $ we
see that $\tau = 0$ if and only if $\eta_n = \pm \eta_n'$ and $d
\tau = 0$ on this set  if and only if $\eta_n = \eta_n' = 0$. This
proves that the intersection (\ref{INTER}) is transversal except
on the set $0_{\R} \times \Delta_{T^* H \times T^*H}$ and that it
fails to be transversal there. Consequently, the pushforward is a
smoothly immersed Lagrangian submanifold away from this singular
set.

We now consider $\iota$ and first  restrict it  to $\R \times
(\hat{C}_H \backslash \Delta_{T^* H \times T^*H})$ since
$\hat{C}_H$ does not have a well-defined tangent plane on the
critical locus.
 The map $\iota$ is then an immersion as long as $(G^t \times G^t)
(\hat{C}_H)$ is transverse to the orbits of $G^t \times G^t$. As
noted in \S \ref{HYPERSURFACES}, $S^* H$ is the set of points of
$S^*_H M$ where the Hamilton vector field $H_g$ of $g = |\xi|$ is
tangent to $S^*_H M$. Hence, $\iota|_{ \R \times (\hat{C}_H
\backslash \R \Delta_{S^* H \times S^*H})}$ is a Lagrange
immersion.  It follows that  $\pi_{t*} \Delta_t^* \Gamma^* \circ C_H
\circ \Gamma$ is an immersed canonical relation away  from the set
$ \R_+ \bigcup_{|t| \leq T} (G^t \times G^t) (S^* H \times S^* H)$.

We next consider self-intersection set of this immersion. The
fiber of $\iota$ over a point in the image,
\begin{equation} \label{FIBER} \iota^{-1}(x_0, \xi_0, y_0, \eta_0) =   \{(t,  s, \xi,
\xi') \in \R \times \hat{C}_H:
  (G^{-t}(s, \xi), G^{-t}(s, \xi')) = (x_0, \xi_0, y_0, \eta_0) \},
  \end{equation}
corresponds to simultaneous hitting times of $(x_0, \xi_0)$ and
$(y_0, \eta_0)$ on $T^*_H M$. Thus, the  self-intersection locus
of $\Gamma_{T, \epsilon}$  consists of the image of pairs   $ (t, s,
\xi, \xi'), (t', s', \eta, \eta')$ such that
$$G^t(s, \xi) = G^{t'}(s', \eta), \;\;\;G^t(s, \xi') = G^{t'}(s',
\eta') \iff G^{t - t'}(s, \xi), G^{t - t'}(s, \xi') \in T^*_H M.
$$
If $\xi = \xi'$ then $(s, \eta) = (s', \eta')$ and the
self-intersection points correspond to the return times and
positions of $(s, \xi)$ to $T^*_H M$. If $\xi' = r_H \xi$, then
the self-intersection points correspond to the times where the
left and right times are the same.  Away from $T^* H \times T^* H$
the set of return times is discrete.

This concludes the proof of the Lemma and hence of Proposition
\ref{CANON}.

\end{proof}

\section{\label{RCALSECTION} Return times and reflection maps $\rcal_j$}

In Proposition \ref{CANON},  $\Gamma_{T,\epsilon}$  is shown to be  a
canonical relation. In this section, we study the diagram
\begin{equation}\label{DIAGRAM} \begin{array}{ccccc} & & \Gamma_T \subset T^*M \times T^* M
 & &  \\ & & & & \\
& \pi \swarrow & & \searrow \rho & \\ & & & & \\
T^*M  & & & & T^*M .
\end{array}
\end{equation}
Our aim is to show that the map $\rho \pi^{-1}$  defines a
finitely multi-valued  symplectic correpondence.
 Underlying the projections is the map
\begin{equation} \label{G} G : \R  \times T^*_H M
\to T^*M - 0, \;\;\;   G (t, s, \xi) = G^t(s, \xi),
\end{equation} which was introduced in Lemma \ref{RFOLD}. We often restrict $G$ to $[- T, T] \times T^*_H M$
and then denote it by $G_T$. In Lemma \ref{RFOLD}, we determined
the  singular set of $G$.

We note that $D G  \frac{\partial}{\partial t} = H_g$ and that
$  D_{s, \xi}   G^t = D G^t |_{T^*_H M}$. Hence $D G $ is injective
(and surjective) as long as $H_g$ is linearly independent from $D
G^t |_{T^*_H M}$. As discussed in \S \ref{HYPERSURFACES}, $H_g = D
G^t X$ with $X \in T (T^*_H M)$  if and only if $X \in T(T^* H)$.
We now restrict the time domain to $(- T, T)$ (it is immaterial
whether we use the closed or open interval).

\subsection{Definition of the maps $\rcal_j$}

We now  define the correspondences (\ref{RCAL}) and the associated
return times. We consider the subset  $\pi_{t*} |_{[-T, T]} \;
\Delta_t^* \Gamma^* \circ C_H \circ \Gamma $ of $\pi_{t*} \;
\Delta_t^* \Gamma^* \circ C_H \circ \Gamma $ where we restrict the
time interval to $[-T, T]$.  We define $\hat{\Gamma}_{T \epsilon}
$ to consist of  $\Gamma_{T, \epsilon}$ together with a subset of
the diagonal $\Delta_{T^* M \times T^* M}$.

\begin{prop}
\label{graph}    The canonical relation $\hat{\Gamma}_{T \epsilon}
$ is the disjoint union of

\begin{itemize}

\item The diagonal graph over the image $(S^* M \backslash \left(G
(- T, T) \times (S^*_H M)
 \right)\;\; \backslash \gcal_{T, \epsilon}$;

\item  $\Gamma_{T,\epsilon} \subset T^*M \times T^*M$, which  is a
finite union of (transversally intersecting)  canonical graphs,
$$ \Gamma_{T,\epsilon} = \bigcup_{j=1}^{N_{T,\epsilon}} \{ ( x,\xi; \rcal_j(x,\xi) ): (x, \xi) \in \dcal^{(j)}_{T,\epsilon}\}.$$

\item The graph $\{ ( x,\xi; \rcal_j(x,\xi) )\}$ intersects the
graph $  \{ ( x,\xi; \rcal_k(x,\xi) )\} $ when \eqref{RJERKintro} holds.

\end{itemize}

\end{prop}

\begin{proof} We consider  the   projections $\pi, \rho$ in
the diagram (\ref{DIAGRAM}) restricted to $\pi_{t*} \Delta_t^*
\Gamma \circ C_H \circ \Gamma $, and   use the description of the
latter in  (\ref{PushGtCH}).  We thus define
$$  \pi(G^t(s, \xi), G^t(s, \xi')) = G^t(s, \xi), \;\;\;\rho(G^t(s, \xi), G^t(s,
\xi')), \;\; ((s, \xi, \xi') \in \hat{C}_H).  $$ The compositions
$\pi \circ \iota, \rho \circ \iota$ are just the map $G$ studied
in Lemma \ref{RFOLD}.  As shown there, each of these maps has a
bijective differential on $\R \times (T^*_H \backslash T^* H)$.
Since  the flowout of $\R \; \Delta_{S^* H \times S^*H}$ is
removed from $\Gamma_{T, \epsilon}$, and since  $\hat{C}_H \cap
T^* H \times T^* H = \Delta_{T^* H \times T^* H}$, there are no
$  (t, s, \xi, \xi')$ with $\xi$ or $\xi'$ in $T_s^* H $ in the part
of the domain of $\iota$ parameterizing $\Gamma_{T, \epsilon}$.

The new aspect is that we are considering $\pi, \rho$ directly as
maps on $\pi_{t*} \Delta_t^* \Gamma^* \circ C_H \circ \Gamma $,
which is an immersed rather than embedded relation.

On $[-T, T] \times (S^*_H M)_{\geq \epsilon}$, $G$ is a proper
submersion and hence a finite covering map. The  domains
$\{\dcal^{(j)}_{T, \epsilon}\}_{j = 1}^{N_{T, \epsilon}} $ defined
in Definition \ref{IMPACTTIMES} are fundamental domains for $G$,
and we have
\begin{equation} \label{DDCALJDEF} \begin{array}{l} , \\ \\
\dcal^{(1)}_{T, \epsilon} = \bigcup_{(x, \xi) \in (S^*_H M)_{\geq
\epsilon} } \; (0,  T^{(1)}(x, \xi)) \times \{(x, \xi)\},\\ \\ \cdots \\ \\
\dcal^{(j)}_{T, \epsilon} = \bigcup_{(x, \xi) \in (S^*_H M)_{\geq
\epsilon} } \; ( T^{(j-1)}(x, \xi)), T^{(j)}(x, \xi) ) \times
\{(x, \xi)\}
\end{array}  \end{equation} Thus,  $\dcal^{(j)}_{T, \epsilon}
\subset [-T, T] \times (S^*_H M)_{\geq \epsilon}$ are disjoint
open subsets whose union is $ [-T, T] \times (S^*_H M)_{\geq
\epsilon}$, such that $G$ is a diffeomorphism of $\dcal^{(j)}_{T,
\epsilon}$ to its image. The closures of the $\dcal^{(j)}_{T,
\epsilon}$ intersect at the points where \eqref{RJERKintro}  holds by the calculation in Lemma \ref{PUSHPULL}.

We now consider the smooth components of $\hat{\Gamma}_{T
\epsilon} $. In the parametrizing map (\ref{IOTA}), we have
removed the separating hypersurface $\R \times \Delta_{T^* M
\times T^* M}$ from the parameter space. Hence it has two
connected components, one is which is $\R \times \Delta_{T^*_H M
\times T^*_H M}$ and the other of which is $\R$ times the graph of
$r_H : T^*_H M \to T^*_H M$. Under $\iota$ these two components
map to disjoint canonical relations in $T^* M \times T^* M$. The
first is of course $\Delta_{T^* M \times T^* M}$, and the second
is $\Gamma_{T, \epsilon}$. We define $\rcal_j$ to be the partial
symplectic map defined on $\dcal_{T, \epsilon}^{(j)}$ by $\rho
\circ \pi^{-1} $.
 Thus,  $\rcal_j$ is well-defined and smooth $\dcal_{T,\epsilon}^{(j)}$.
  For $(x, \xi) \in \dcal_{T, \epsilon}^{(j)}$ the jth
$H$-reflection map $\rcal_j$ is given by (\ref{RCAL}).

 By definition of the  $N_{T,\epsilon}$  smooth functions $t_j:
   \dcal_{T,\epsilon}^{(j)} \rightarrow (-T,T), j =1,...,N_{T,\epsilon}$, we have
\begin{equation} \label{claim}
 \Gamma_{T,\epsilon} = \bigcup_{j=1}^{N_{T,\epsilon}}   \bigcup_{(x,\xi) \in
 \dcal^{(j)}_{T,\epsilon}} (x,\xi; G^{t_j(x,\xi)} r_H G^{-t_j(x,\xi)}(x,\xi) ). \end{equation}

\end{proof}

\section{Analysis of  $ V_{\epsilon}(t; a)$} \label{Vta operator}

So far, we have studied the symplectic geometric aspects of the
compositions underlying $V_{\epsilon} (t; a)$ and $\VT(a)$.  In \S
\ref{HYPERSURFACES} and \S \ref{RCALSECTION}, we studied the
composition of canonical relations underlying the composition of
operators in $ V_{\epsilon}(t; a)$ and $\VT(a)$. The compositions studied in
the previous section \ref{COMPOSITIONS} imply that these operators
are Fourier integral operators. The purpose of this section (and
the next)  is to calculate the principal symbol of $ V_{\epsilon}(t;a) $ (and
of $\VT(a)$). The results are again valid for any Riemannian
manifold and hypesurface; we again do not assume ergodicity of
$G^t$ in these sections. 

Our analysis begins with the operator  $\gamma_H^*  Op_H(a)
 \gamma_H$ and its cutoff  $(\gamma_H^*  Op_H(a)
 \gamma_H)_{\epsilon}$ away from the singular sets. It then  proceeds to conjugation by $U(t)$
and integration in $t$. We could equally well have begun with the analysis of $W$ \eqref{W} and
then $W^*W$, which would be closer to the analysis in \cite{SoZ}.

As recalled in \S \ref{FIOS}, the principal symbol of a Fourier
integral distribution
$$ I(x, y) = \int_{\R^N} e^{i \phi (x, y, \theta) }
a( x, y, \theta) d \theta $$
with non-degenerate homogeneous phase function $\phi$ and amplitude $a \in S^{0}_{cl}(M \times M \times \R^N),$
 is the transport to the Lagrangian $\Lambda_{\phi} =
 \iota_{\phi}(C_{\phi})$ of the square root of the density
$$ d_{C_{\phi}}: = \frac{|d \lambda|}{|D(\lambda,
\phi_{\theta}')/D(x,y,\theta)|} $$ on $C_{\phi}$, where $\lambda=(\lambda_1,...,\lambda_n)$ are local coordinates on the critical manifold $C_{\phi} =\{ (x,y,\theta); d_{\theta}\phi(x,y,\theta) = 0\}. $

\subsection{\label{CUTOFFS2} Pseudo-differential cutoffs}

As mentioned in the introduction, we wish to cutoff operators away
from $\Sigma_{T}$ and $N^*H-0 \times 0_{T^*M} \cup 0_{T^*M} \times N^H-0$. 
 As above, let  $x=(s,x_n)$ be Fermi normal coordinates along
$H$, i.e. let $x = \exp_{q_H(s)} x_n \nu_{s}$ where $s \mapsto
q_H(s) \in H$ denotes a local parametrization of $H.$ Then $H = \{
x_n=0 \}$. Let
  $\xi = (\sigma, \xi_n) \in T^*M$ denote the corresponding symplectically dual fibre
 coordinates. Here, we describe these pseudodifferential cutoffs (introduced in (\ref{CUTOFF}) in more detail in terms of Fermi coordinates.

 Let $\psi_{\epsilon} \in C^{\infty}_0(\R)$, $\psi_{\epsilon} \equiv 1$ on $[-\epsilon/2,
\epsilon/2]$ and $\psi_{\epsilon} \equiv 0$ on $(-\infty, -\epsilon] \cup [\epsilon, \infty).$ In Fermi normal coordinates, we may
take the cutoff  $\chi_{\epsilon}^{(tan)} \in C^{\infty}(T^*M)$ (see  also (i)-(iii) in the Introduction) to be
\begin{equation} \label{CHIEP} \chi_{\epsilon}^{(tan)}(s, y_n, \sigma, \eta_n) =  
\psi_{\epsilon}\left( \frac{|\eta_n|^2}{|\sigma|^2 + |\eta_n|^2} \right) \cdot \psi_{\epsilon}( y_n).
\end{equation}  which is equal to one in  a conic neighbourhood of $T^*H= \{ y_n= \eta_n = 0 \}.$
 We
further introduce a homogeneous cutoff $\chi_{\epsilon}^{(n)} \in C^{\infty}(T^*M)$ given by
\begin{equation} \label{PSIEP} \chi_{\epsilon}^{(n)}(s, y_n, \sigma, \eta_n) = 
\psi_{\epsilon}\left( \frac{|\sigma|^2}{|\sigma|^2 + |\eta_n|^2} \right) \cdot \psi_{\epsilon}(y_n)
\end{equation} 
which equals one on a conic neighborhood of $N^*H = \{y_n = \sigma = 0\}.$  More precisely,
we multiply \eqref{CHIEP} and \eqref{PSIEP}  by  a bump function $\psi(\xi)$ which vanishes identitically
 near the zero section.

As in (\ref{CUTOFF}) we  introduce the combined smooth homogeneous cutoff
\begin{equation}\label{totalcutoff}
\chi_{\epsilon} := \chi_{\epsilon}^{(tan)} + \chi_{\epsilon}^{(n)} \end{equation}
and denote the corresponding pseudo-differential
operator by $\chi_{\epsilon}(x, D)$ or by $Op(\chi_{\epsilon})$.


%
%


\subsection{\label{gammaHsect} $ Op_H(a) \gamma_H (1-\chi_{\epsilon}) $ }
  In Fermi  coordinates,
  \begin{equation} \label{OPAH}    Op_H(a) \gamma_H (s; x_n, s') = C_n \int e^{i \langle s-s', \sigma \rangle  - i x_n \xi_n} a(s,\sigma)  (1-\chi_{\epsilon}(s',x_n,\sigma',\xi_n) ) \, d\xi_n d\sigma.
  \end{equation}
  The phase $\phi(s, x_n, s', \xi_n, \sigma) =  \langle s - s', \sigma
   \rangle - x_n \xi_n  $   is linear and non-degenerate, the number  of phase variables is $N = d $ and $n = 2d-1$,
  where $d = \dim M$, so $ \frac{N}{2} - \frac{n}{4}=
  \frac{1}{4}. $ Then $C_{\phi} =\{(s, x_n, s', \sigma,
   \xi_n): s = s', x_n = 0, \} $ and
   $\iota_{\phi}(s, 0, s, \sigma, \xi_n) \to (s, \sigma, s, \sigma,
0,\xi_n). $

The complication arises that elements of the form $(s, \xi, s, 0)$ arise when $\xi \in N^*H$ in the canonical relation of $\gamma_H$
and similarly $(s, 0, s, \xi)$ arises in that of $\gamma_H^*$. Hence they are not homogeneous canonical relations in the sense
of \cite{HoI-IV}, i.e. conic  canonical relations $C \subset (T^*X \backslash 0) \times (T^* Y \backslash 0)$.  We  introduced the cutoff  $(1-\chi_{\epsilon})$ in \eqref{OPAH}
so that no such elements occur in the support of the cutoff and then
$$\gamma_H (1-\chi_{\epsilon}) \in I^{ \frac{1}{4}}(M \times H,
\Lambda_H), $$ where $\Lambda_H =   \{(s, \xi, s, \sigma) \in T^*_H M
\times T^* H: \xi|_{TH} = \sigma\}.  $ Its adjoint $ (1-\chi_{\epsilon}) \gamma_H^* $ then
lies in $ I^{ \frac{1}{4}}(H \times M, \Lambda_H^*), $ where
$\Lambda_H^*=  \{(s, \sigma, s, \xi) \in   T^* H \times T^*_H M:
\xi|_{TH} = \sigma \}.  $

\subsection{$(\gamma_H^* Op_H(a)
\gamma_H)_{ \geq \epsilon}$}
The composition $\gamma_H^* Op_H(a) \gamma_H$ also fails to be a Fourier integral operator with homogeneous canonical relation
for the same reason. We recall that
 (\cite{HoI-IV} Theorem 8.2.14) that the general composition of wave front sets has the form:
Let $A: C_0^{\infty}(Y) \to \dcal'(X), B: C_0^{\infty} (Z) \to \dcal'(Y)$. Then if
$WF_Y'(A) \cap WF'_Y(B) = \emptyset$,  then $A \circ B: C_0^{\infty} (
Z) \to \dcal'(X)$ and
$$WF'(A \circ B) \subset WF'(A) \circ WF'(B) \cup (WF'_X(A) \otimes 0_{T^*Z} ) \cup (0_{T^*X} \times WF'_Z(B)). $$
 Thus,
$$\begin{array}{lll} WF'(\gamma_H^* Op_H(a) \gamma_H)  & \subset &  \{(q, \xi, q, \xi'): \xi |_H = \xi' |_H),  (q, \xi), (q, \xi') \in T^*_H M - 0\} \\ &&\\ &&\cup
 \{(q, \nu, q, 0):  (q, \nu) \in N^*H - 0\} \cup \{(q, 0, q, \nu): \nu \in N^*H - 0\}. 
. \end{array}$$
With the cutoff   $(I - \chi_{\frac{\epsilon}{2}})$  on the left and $(1-\chi_{\epsilon})$ on the right of $\gamma_H^* Op_H(a)
\gamma_H,$ the last two sets are erased. Observing that $(1-\chi_{\frac{\epsilon}{2}}) (1-\chi_\epsilon) = 1-\chi_{\epsilon},$ we have proved 

\begin{lem}\label{gammaH*gammaH}  If $a \in S^0_{cl}(T^*H)$, then $(\gamma_H^*  Op_H(a)
 \gamma_H)_{\epsilon}\in I^{\half} (M \times M, C_H). $ In the Fermi normal
 coordinates  the symbol is given by
 $$\sigma_{(\gamma_H^* Op_H(a) \gamma_H)_{ \geq \epsilon}}(s, \sigma, \eta_n, \eta_n') =  (1 - \chi_{\epsilon})   a(s,
\xi|_{TH}) |\Omega|^{\half}, $$ where $\Omega = | ds \wedge d
\sigma \wedge d \eta_n \wedge d
 \eta_n'|$.
\end{lem}

\begin{proof} In Lemma \ref{CSUBH}, we showed that $C_H$ is an
embedded Lagrangian submanifold of $T^* M \times T^*M$. The proof
shows that the composition of $
\Lambda_H^* \circ \Lambda_H$ is transversal. Since the order of
$\gamma_H^*$ equals that
  of $\gamma_H$ and the orders add under transversal  composition,
  the order of $(\gamma_H^*  Op_H(a)
 \gamma_H)_{ \geq \epsilon}$ is $\half$.  Hence, for any
homogeneous pseudo-differential operator $Op_H(a)$ on $H$,
 \begin{equation} \label{WFGAMMA} (\gamma_H^* Op(a) \gamma_H)_{\geq \epsilon} \in
 I^{\half}(M \times M, C_H).
 \end{equation}

Next we compute its principal symbol.
  By Lemma
 \ref{CSUBH}, $C_H$  is the fiber product $T^*_H M \times_{T^*H} T^*_H
 M$, hence it  carries  a canonical  half-density (associated to
 the fiber map).
 As discussed in \cite{GuSt} (p. 350), on any fiber product $A
 \times_B C$, half-densities on $A, C$ together with
 a negative density on $B$ induce a half density on the
 fiber product. In our setting, the canonical half-density on $T^*_H M$ is
 given by the square root of the quotient $\frac{\Omega_{T^*M}}{d
 y_n} = ds \wedge d \sigma \wedge d \eta_n$ of the symplectic volume density on $T^* M$ by the
 differential of the defining function $y_n$ of $T^*_H M$. We also
 have a canonical density $|ds \wedge d \sigma|$ on $T^* H$, which
 induces a  canonical $-1$-density. The induced half-density on
 $C_H$ is then $| ds \wedge d \sigma \wedge d \eta_n \wedge d
 \eta_n'|^{\half}$.

 We compute the principal symbol and order  using the
special oscillatory integral formula,
\begin{equation} \label{SYMBOLORDER} \begin{array}{lll} \gamma_H^* Op(a) \gamma_H (s, x_n;
s', x_n') & = &   C_n
 \delta_0(x_n) \int e^{i \langle s-s', \sigma \rangle  - i x_n' \xi_n'} a(s,\sigma) d\xi_n'
 d\sigma  \\ && \\
 & = &   C_n
  \int_{\R^n \times \R \times \R}    e^{i \langle s-s', \sigma \rangle  + i x_n \xi_n - i x_n' \xi_n'} a(s,\sigma) d\xi_n
 d\sigma d \xi_n' .  \end{array}
\end{equation} If we  compose on left and right by  $(1 -\chi_{\epsilon})$  and $(1-\chi_{\frac{\epsilon}{2}})$ respectively then we further obtain  factor
of  \ $(1 -\chi_{\epsilon}(s, x_n, \sigma, \xi_n) )$  under the integral.
The phase is $ \phi(s, x_n, s', x_n', \xi_n,
 \xi_n', \sigma) =  \langle s - s', \sigma
   \rangle + x_n \xi_n - x_n' \xi_n'  $
   with phase variables $(\xi_n, \xi_n', \sigma)$, and
   $$C_{\phi} = \{(s, x_n, s', x_n', \sigma, \xi_n, \xi_n'): s = s',
   x_n = 0, x_n' = 0\}. $$
   Also,
   $$\iota_{\phi}(s, 0, s, 0, \sigma, \xi_n, \xi_n') = (s, \sigma, s, \sigma,
   0, \xi_n, 0, \xi_n') \in T^* M \times T^* M. $$
   Thus, $(s, \sigma, \xi_n, \xi_n')$ define coordinates on
   $C_{\phi}$.

   As discussed in \S \ref{FIOS}, the delta-function on $C_{\phi}$
   is given by
   $$ d_{C_{\phi}} = \frac{| ds \wedge d \sigma \wedge d \xi_n
   \wedge d \xi_n' |}{D(s,  \sigma, \xi_n, \xi_n', \phi'_{\xi_n},  \phi'_{\xi_n'}
   \phi'_{\sigma})/D(s, s', \sigma, x_n , x_n', \xi_n, \xi_n')|}.
   $$ Since
    $$  | \, D(   \phi'_{\xi_n}, \phi'_{\xi_n'},
   \phi'_{\sigma})/D( s', x_n, x_n')  \,  | = 1,
   $$
   the lemma follows.

\end{proof}

We further recall from the introduction the operator $(\gamma_H^*
Op_H(a) \gamma_H)_{ \geq \epsilon}$ in  (\ref{>ep}). We have:

\begin{cor}  With the same notation and assumptions as above, $$(\gamma_H^*
Op_H(a) \gamma_H)_{\geq \epsilon}  \in I^{\half} (M \times
M, C_H),$$ and its symbol is given by
 $$\sigma_{(\gamma^* Op_H(a) \gamma})_{\geq \epsilon} (s, \sigma, \eta_n, \eta_n') =  (1 - \chi_{\epsilon}) a(s,
\xi|_{TH}) |\Omega|^{\half}, $$ where $\Omega = | ds \wedge d
\sigma \wedge d \eta_n \wedge d
 \eta_n'|$. 

\end{cor}

\begin{rem}  In the case of semi-classical pseudo-differential
operators on $H$ in \cite{HZ}, we could use a cutoff on $B^* H$
away from its boundary $S^* H$. No such cutoff exists for
homogeneous pseudo-differential operators on $H$. The closest
analogue is to introduce the  cutoff on $\gamma_H^* Op_H(a)
\gamma_H$.
\end{rem}

\subsection{  $U(t_1)^* (\gamma_H^* Op_H(a) \gamma_H)_{\geq \epsilon} U(t_2)$ }

The next step is to right and left compose with the wave group.
The canonical relation was determined in Lemma
\ref{GammaCHGammaLEM}. We now work out the symbol.

\begin{lem} \label{GAMMAHLEMa} If $a \in S^0_{cl}(T^*H)$, then $$  U(-t_1) (\gamma_H^*  Op(a)
 \gamma_H)_{\geq \epsilon} U(t_2)  \in I^{0} (\R \times M \times \R \times M,
 \Gamma^* \circ C_H \circ \Gamma). $$

 Under the embedding $\iota_{\Gamma^* C_H \Gamma}$ (of Lemma \ref{GammaCHGammaLEM}), the  principal symbol
 pulls back to the homogeneous function on $
\R \times \R \times  T^*_H M$ given by
 $$  (1 - \chi_{\epsilon}) a_H(s, \xi) |dt \wedge dt_1 \wedge \Omega|^{\half},
 $$
 where $|dt \wedge dt_1 \wedge \Omega|^{\half}$ is the canonical volume
 half-density on $\Gamma^* \circ C_H \circ \Gamma$ (defined in the proof).

\end{lem}

\begin{proof} It is well known (see \cite{HoI-IV}, vol. IV)  that $U(t) \in I^{-\frac{1}{4}}(\R
\times M \times M, \Gamma), $ with $\Gamma = \{(t, \tau, x, \xi,
G^t(x, \xi)): \tau + |\xi| = 0\}. $ As in \cite{DG}, the half
density symbol of $U(t, x, y)$ is the canonical volume half
density $ \sigma_{U(t, x, y)} = |dt \otimes dx \wedge d \xi
|^{\frac{1}{2}}$ on $\Gamma$. By Proposition
\ref{GammaCHGammaLEM}, the composition is $\Gamma^* \circ C_H \circ
\Gamma$ is transversal for any hypersurface $H$, hence $ U(-t_1)
(\gamma_H^*  Op_H(a)
 \gamma_H)_{\geq \epsilon} U(t_2) $ is a
Fourier integral operator with the stated canonical relation.
Under transversal composition the orders add, and the stated order
follows from Lemma \ref{gammaH*gammaH} together with the fact that
$U(t) \in I^{-\frac{1}{4}}(\R \times M \times M, \Gamma)$.

To prove the formula for the symbol,  we observe that $ \Gamma^*
\circ C_H \circ \Gamma$ is parameterized  by $$  \iota_1 (t, t', s,
\sigma, \eta_n, \eta_n') = (t, |\xi|, t', - |\xi'|, G^t(s, \xi),
G^{t'}(s, \xi')): s \in H, \xi, \xi' \in T_x M, \;\; \xi|_{TH} =
\xi' |_{TH} \}, $$ where  $ \xi = (\sigma,  \eta_n), \xi' = (\sigma, \eta_n'))$ are dual Fermi coordinates  in the orthogonal decomposition of
$T_H^* M = T^* H \oplus N^* H $. The natural volume half density on
parameter domain of $\Gamma \circ C_H \circ \Gamma$  is $  |dt_1
\wedge dt_2 \wedge d s \wedge d \sigma \wedge d \eta_n \wedge d
\eta_n'|^{\half}  $ where $ds \wedge d \sigma$ is the symplectic volume
form on $T^*H$, where $(\eta_n, \eta_n')$ are the normal components of
$(\xi, \xi')$ and where $d\eta_n$ is the Riemannian density on $N^*_s
H$. The stated symbol then follows by transversal composition from
the symbols of $U(t)$ and of $\gamma^*_H Op_H(a) \gamma_H$
determined in Lemma \ref{GammaCHGammaLEM}).
\end{proof}

\subsection{$V_{\epsilon}(t; a)$}

The purpose of this section is to prove

\begin{lem} \label{GAMMAHLEM} If $ a \in S^0_{cl}(T^*H) $, then $$V_{\epsilon}(t; a) = U(-t) (\gamma_H^*  Op(a)
 \gamma_H)_{\geq \epsilon} U(t) \in I^{\frac{1}{4}} (\R \times M \times M,
 \Gamma_H). $$

 Under the embedding $\iota_{\Gamma_H}$ (of Lemma \ref{DELTATRANS}), the  principal symbol
 pulls back to the homogeneous function on $
\R \times   T^*_H M$ given by
 $$ \iota_{\Gamma_H}^* \sigma_{V_{\epsilon}(t; a) }(t, s, \xi,  \xi')  = (1 - \chi_{\epsilon})(s, \xi) a_H(s, \xi) |dt \wedge \Omega|^{\half},
   $$
 where $|dt \wedge \Omega|^{\half}$ is the canonical volume
 half-density on $\Gamma_H$ (defined in the proof).

\end{lem}

\begin{proof} The new step beyond Lemma \ref{GAMMAHLEMa} is to
pull back the canonical relation and symbol under the
time-diagonal embedding  $\Delta_t (t, x, y) = (t, t,  x, y)$  of
$\R \times M \times M \to \R \times \R \times M \times M$. In \S
\ref{GCHG} and Lemma \ref{GammaCHGamma}, together with \S
\ref{PBDELTA} and Lemma \ref{DELTATRANS}, we showed that the
compositions are transversal.
 Hence, for any hypersurface $H$, $ V_{\epsilon}(t; a)$ is a
Fourier integral operator with the stated canonical relation.

As mentioned above, orders add under transversal composition.
Before pulling back under the diagonal relation the composition
has order $0$ by Lemma \ref{GAMMAHLEMa}. Setting  $t = t'$ is
composition with the pullback $\Delta_t^*$, which has order $
\frac{1}{4}$ \cite{DG}. Hence the order is now $\frac{1}{4}$.

To compute the symbol,  we use that  the pullback of  $\Gamma_H$ under $\Delta_t$
may be parameterized by
$$  \iota_{\Gamma_H}: (t, s,
\sigma, \eta_n, \eta_n')   \in \R \times T^*H \times T_H^* M \times T_H^*M
\to (t, |(\sigma,\eta_n)| - |(\sigma,\eta_n')|, G^{- t}(s, \sigma, \eta_n), G^{-t}( s, \sigma, \eta_n')),
$$ in the notation of Lemma \ref{GAMMAHLEMa}. We need to verify that
\begin{equation}\label{QUO}   \Delta_t^* |dt_1 \wedge d t_2 \wedge d s \wedge d \sigma
\wedge d \eta_n \wedge d \eta_n'|^{\half} = |dt \wedge ds \wedge d
\sigma \wedge d \eta_n \wedge d \eta_n'|^{\half}.  \end{equation} We use
the pullback diagram
$$\begin{array}{ccc} \Gamma^* \circ C_H \circ \Gamma  & \leftarrow & F  \to
\Delta^*\Gamma^* \circ C_H \circ \Gamma
\\ &&\\ \downarrow &  & \downarrow \\ && \\
T^*(\R \times \R \times M \times M)  & \leftarrow &
\ncal^*(\mbox{graph}(\Delta_t))
\end{array}$$
Here, $F$ is the fiber product, $\ncal^*(\mbox{graph}(\Delta_t))$
is the co-normal bundle to the graph, and  $\alpha: F \to
\Delta^*\Gamma^* \circ C_H \circ \Gamma$  is the map to the
composition (see \cite{DG,GuSt}. Since the composition is
transversal, $D \alpha$ is an isomorphism (loc. cit.). The graph
of $\Delta_t$ is the set  $\{(t, t, x, y,  t, x,  y)\}$
and its conormal bundle is
$$\ncal^*(\mbox{graph}(\Delta_t)) = \{(t, \tau_1, t, \tau_2, x,
\xi, y, \eta, t, - (\tau_1 + \tau_2), x, - \xi, y, - \eta)\}. $$
The canonical half-density on this graph is $|dt \wedge d \tau_1
\wedge d \tau_2 \wedge \Omega|^{\half}$, where $\Omega = |dx
\wedge d \xi \wedge dy \wedge d \eta|$ is the canonical volume
density on $T^* M \times T^* M$. The half density produced by the
pullback diagram takes the product of the half densities $ |dt_1
\wedge dt_2 \wedge d s \wedge d \sigma \wedge d \eta_n \wedge d
\eta_n'|^{\half}$ and $|dt_2 \wedge d \tau_1 \wedge d \tau_2 \wedge
\Omega|^{\half}$ on the two factors $  \Gamma_H$ and  $\ncal^*(\mbox{graph}(\Delta_t)) $ and divides by the canonical
half density $|dt_1 \wedge d \tau_1 \wedge dt_2 \wedge d \tau_2
\wedge \Omega |^{\half}$ on $T^* \R \times T^* \R \times T^* M
\times T^* M$. The factors of $| dt_1 \wedge d \tau_1 \wedge dt_2
\wedge d \tau_2 \wedge \Omega|^{\half}$ cancel in the quotient
half-density, leaving  the one stated in (\ref{QUO}).

Finally, the presence of the factor of  $(1 - \chi_{\epsilon}(x, \xi)) a_H(s, \xi)$ 
follows immediately from the representation (\ref{SYMBOLORDER}).

\end{proof}

\section{Analysis of  $\VT(a)$} \label{v operator}

 The purpose of this section
is to prove Proposition \ref{VTDECOMPa}.
To define $\VT(a)$ we need to integrate in $t$, i.e. pushforward
from $\R \times M \times M \to M \times M$. It is in this step
that the composition becomes non-transversal due to the tangential
geodesics and requires a cutoff. We begin by defining it more
precisely. Then we decompose $\VT(a)$ into its branches and define
the principal symbol on each branch. In other words, we compute the
symbol of $W^* W$ or more precisely of $W^* \chi_T Op_H(a) W$ away
from the fold singularity.

\begin{lem}\label{VTALEM} For all $T,\epsilon >0,$ we have that   $\VT(a) \in I^{0}(M \times M; \hat{\Gamma}_{T,\epsilon}).$  \end{lem}

\begin{proof} This follows from Lemma \ref{PUSHPULL}.

We denote  by $ \pi_t: \R \times M \times M \to M \times M $ the
natural projection  $ \pi_t(t,x,y) = (x,y) $ and define
$$\pi_{T *} K(t, x, y) = \int_{\R} \chi(T^{-1} t) K(t, x, y) dt.
$$
Then, \begin{equation} \label{COMPO}  \VT(a) = \pi_{T*} \circ U(-t)
\circ (\gamma_H^* Op_H(a) \gamma_H)_{\geq \epsilon}\circ U(t),
\end{equation} or equivalently,  the Schwartz kernel of $\VT(a)$
is
 \begin{equation}\begin{array}{ll}
 \VT(a)(x,y) & =\frac{1}{T} \int_{-T}^{T} \int_H \int_H U_{\frac{\epsilon}{2}}^*(t, x,
s) \\ & \\ & \cdot Op_H(a)(s,s') \cdot U_{\epsilon}(t,
s',y) \chi(T^{-1}t) \, d\sigma(s) d\sigma(s') dt,
\end{array}\end{equation}
where, $U_{\epsilon}(t):= (1-\chi_{\epsilon}) U(t).$ 
Integration in $t$ pushes forward the canonical relation $
\Delta_t^*\Gamma^* \circ C_H \circ \Gamma $ to the canonical
relation $ \Gamma_{T,
\epsilon} $ studied in \S \ref{PF} and in Lemma \ref{PUSHPULL}.
$$WF'( \VT(a)) =   \{ (x,\xi,x',\xi'): (t,0,x,\xi,x',\xi') \in  WF'( V_{\epsilon}(t,a))  \}.$$
Thus $\VT(a)$ is a Fourier integral operator as long as the
composition is transversal. With the cutoff in place, transversal
composition was proved in Lemma \ref{PUSHPULL}.

We compute the order by the argument of \cite{DG}. We note that
$\pi_{T *} \Delta_t^*$ maps half densities on $\R \times \R \times
M \times M$ to half densities on $M \times M$ and its Schwartz
kernel is then a half density on $(\R \times \R \times M \times M)
\times (M \times M)$ which coincides with the Schwartz kernel of
the identity operator under an interchange of order of the
variables. Hence, $\pi_{T *} \Delta_t^* \in I^0((\R \times \R
\times M \times M) \times (M \times M), \Gamma)$ where $\Gamma$ is
the identity graph. As a result, applying it to   $U(-s) (\gamma_H^*
Op_H(a)
 \gamma_H)_{\geq \epsilon} U(t)$  preserves the order. But as noted above, the
 latter operator has order zero.

\end{proof}

\subsection{Symbol of $\VT(a)$}

We now calculate the symbol of $\VT(a)$. The symbol is a section
of the bundle of half-densities (tensor Maslov factors) on the
canonical relation $\pi_{t*} \Delta_{t}^* \Gamma^* \circ C_H \circ
\Gamma$. We parametrize the canonical relation by (\ref{IOTA}) and
view the symbol as a half-density on the parameter space $\R
\times \hat{C}_H$. The symbol of $\Delta^* \Gamma^* \circ C_H \circ
\Gamma$ is a half-density (tensor Maslov factor) on the parameter
space $\R \times C_H$ (see Lemma \ref{GAMMAHLEM}). To calculate
it, we use the pushforward diagram,  $$\begin{array}{ccc}
\Delta_t^*\Gamma^* \circ C_H \circ \Gamma & \leftarrow & F  \to
\pi_{t *} \Delta^*\Gamma^* \circ C_H \circ \Gamma
\\ &&\\ \downarrow &  & \downarrow \\ && \\
T^*(\R \times M \times M)  & \leftarrow &
\ncal^*(\mbox{graph}(\pi_t)).
\end{array}$$
Here, $\mbox{graph}(\pi_t) = \{(x, y; t, x, y)\} $ and
$$\ncal^*(\mbox{graph}(\pi_t)) = \{(x, \xi, y, \eta; t, 0, x, -
\xi, y, - \eta)\}, $$ which is naturally parameterized by $(t, x,
\xi, y, \eta) \in \R \times T^*(M \times M). $

 We note that
$\R \times \hat{C}_H$ is the set $\{\tau = 0\} \cap \R \times C_H$
where $ \tau = |\xi| - |\xi'| = |(\sigma,\eta_n)| - |(\sigma,\eta_n')| $. We claim that
\begin{lem}
\begin{equation}   \pi_{t*} |dt \wedge ds \wedge d
\sigma \wedge d \eta_n \wedge d \eta_n'|^{\half}  = |\frac{\sqrt{\sigma^2
+ \eta_n^2}}{\eta_n}  dt \wedge d s \wedge d
  \sigma \wedge d \eta_n|^{\half}, \end{equation} \end{lem}

\begin{proof} Away from the tangential directions, the pushforward
is a transversal composition, and we only calculate the
half-density symbol on that set. The map from half-densities on
the fiber product to half-densities on the composition is then a
canonical  isomorphism.

On the fiber product, we have the half-density given by tensoring
$  |dt \wedge ds \wedge d \sigma \wedge d \eta_n \wedge d
\eta_n'|^{\half} $ with the canonical half density $|dt_1 \otimes
\Omega|^{\half}$ where $\Omega$ is the symplectic volume density
on $T^* M \times T^* M$. When we divide by the canonical
half-density $|dt_1 \wedge d\tau_1 \wedge \Omega|^{\half}$ on $T^*
(\R \times M \times M)$ we obtain
$$ \frac{|dt \wedge ds \wedge d \sigma \wedge d \eta_n \wedge d
\eta_n'|^{\half}}{|d \tau_1|^{\half}}. $$

 $$\begin{array}{lll}   d \tau |_{\tau=0}  &= &  d ( |(\sigma,\eta_n)| - |(\sigma,\eta_n')| ) |_{\eta_n = \pm \eta_n'} \\&&\\
 &=& \half (\sigma^2 + \eta_n^2)^{-\half} (d \sigma^2 + d \eta_n^2) -
 (\sigma^2 + (\eta_n')^2)^{-\half} (d \sigma^2 + d (\eta_n')^2) |_{\eta_n = \pm \eta_n'} \\ &&\\
 & = &  \eta_n (\sigma^2 + \eta_n^2)^{-\half} ( d \eta_n
 \mp d (\eta_n')) |_{\eta_n = \pm \eta_n'}.  \end{array} $$
 Then the quotient density  is
  \begin{equation} \label{QUOTIENT} \sqrt{\sigma^2 + \eta_n^2}   \frac{  dt \wedge d s \wedge d
  \sigma \wedge d \eta_n \wedge d \eta'_n}{ d \eta_n^2 -
  d (\eta_n')^2} = \mp \frac{\sqrt{\sigma^2 + \eta_n^2}}{\eta_n}  dt \wedge d s \wedge d
  \sigma \wedge d \eta_n.  \end{equation}

  The presence of the $\pm$ is due to the fact that the canonical
  relation underlying $\VT(a)$ has both a diagonal and a reflection branch.
Moreover,  it  is immersed rather than embedded, so the symbol is
a collection of half-densities (tensor Maslov factors) on the
union of canonical graphs.

\end{proof}

To complete the calculation, we have
\begin{lem} \label{QUOTCALC} $|dt \wedge d s \wedge d
  \sigma \wedge d \eta_n|^{\half} = |\frac{\eta_n}{\sqrt{\sigma^2  + \eta_n^2}}|^{-\half} \left|  \pi^*
  \Omega_{T^* M} \right|^{\half}, $ where $\pi: \Gamma_H \to T^* M$ is
  the natural projection. \end{lem}

  \begin{proof} In terms of the parametrizing coordinates  $(t, s, \sigma,
  \eta_n, \eta_n')$, the map $\pi$ is given by $\pi(t, s, \sigma,
  \eta_n, \eta_n') = G^t(s, \sigma, \eta_n)$. Hence
  $$\begin{array}{lll} \frac{\pi^* \Omega_{T^* M}}{dt \wedge d s \wedge d
  \sigma \wedge d \eta_n} & = & \Omega_{T^* M} (\frac{d}{dt} G^t(s,
  \sigma, \eta_n), d G^t \frac{\partial}{\partial s_j},  d G^t \frac{\partial}{\partial
  \sigma_j},  d G^t \frac{\partial}{\partial \eta_n} ) \\ && \\
  & = & \Omega_{T^* M} (H_g,  \frac{\partial}{\partial s_j},  \frac{\partial}{\partial
  \sigma_j},   \frac{\partial}{\partial \eta_n} )  \\ && \\
  & = &  \frac{\eta_n}{\sqrt{\sigma^2 + \eta_n^2}}  \Omega_{T^* M} (\frac{\partial}{\partial y_n},  \frac{\partial}{\partial s_j},  \frac{\partial}{\partial
  \sigma_j},   \frac{\partial}{\partial \eta_n} ) \\ && \\
  & = & \frac{\eta_n}{\sqrt{\sigma^2  + \eta_n^2}}.  \end{array} $$
since $\frac{d}{dt} G^t(s,
  \sigma, \eta_n) = H_g =  \frac{\eta_n}{\sqrt{\sigma^2  + \eta_n^2}} \frac{\partial}{\partial y_n} + \cdots $ is the Hamilton vector
  field of $g^2 = \eta_n^2 + (g')^2$ where $\cdots$  represent vector fields in
  the span of $\frac{\partial}{\partial s_j},  \frac{\partial}{\partial
  \sigma_j},   \frac{\partial}{\partial \eta_n}$. Finally, we use that   $d G^t$ is symplectic
  linear and that $s, y_n, \sigma, \eta_n$ are symplectic
  coordinates.

  \end{proof}

  \begin{cor}\label{BIGFACTOR}
  \begin{equation} \pi_{t*}  |dt \wedge ds \wedge d
\sigma \wedge d \eta_n \wedge d \eta_n'|^{\half} = |\frac{\sqrt{\sigma^2
+ \eta_n^2}}{\eta_n}  |^{\half} |\frac{\eta_n}{\sqrt{\sigma^2  +
\eta_n^2}}|^{-\half} \left|  \pi^*
  \Omega_{T^* M} \right|^{\half} = |\frac{\sqrt{\sigma^2
+ \eta_n^2}}{\eta_n}  | \left|  \pi^*
  \Omega_{T^* M} \right|^{\half} .  \end{equation}
  \end{cor}

\subsection{The $P_{T, \epsilon} + F_{T, \epsilon}$ decomposition}

We first define the pseudo-differential part $P_{T,\epsilon}(a)$
of $\VT(a)$. As discussed in Lemma \ref{graph},
 $\R \times \Delta_{T^* H \times T^* H}$   is a separating
 hypersurface in the parameter space $\R \times \hat{C}_H$
  Hence $\R \times \hat{C}_H \backslash  \R \times \Delta_{T^* H \times T^* H} $    has two connected components, $\R \times \Delta_{T^*_H M \times T^*_H M}$ and  $\R \times \mbox{graph} (r_H : T^*_H M \to T^*_H
   M)$,
which map respectively to  $\Delta_{T^* M \times T^* M}$, and the
second is $\Gamma_{T, \epsilon}$.

To separate these pieces of the canonical relation in the support
of the cutoff, we introduce a finite  conic open cover $\{U_j\}$
of $T^* M$ with sufficiently small sets $U_j$ so that the image
under $\iota$ of $\R \times \mbox{graph}(r_H: T^*_H M \to T^*_H
M)$ does not intersect $\bigcup_{j} U_j \times U_j$. This is
possible since $d(\xi, r_H \xi) \geq \epsilon$ in the support of
the cutoff. We then define:

\begin{defin}
 \label{psdopart}
$
 P_{T,\epsilon}(a):= \sum_{j=1}^{N_{T,\epsilon}} \psi_{U_j} \VT(a)
 \psi_{U_j},$ and
 $F_{T,\epsilon}^{(j)}(a) = \VT(a) - P_{T, \epsilon}(a).$

 \end{defin}

Since   $WF'( P_{T,\epsilon}(a)) \subset \Delta_{T^*M \times
T^*M},$  $P_{T,\epsilon}(a)$ is a pseudo-differential operator. We
make a further decomposition of $F_{T, \epsilon}$ below.

\subsection{Principal symbol of $P_{T, \epsilon}$}

The following is a key calculation in the proof of Theorem
\ref{maintheorem}.

\begin{lem} \label{egorov} Let $t_j(x,\xi); j  \in {\mathbb Z}$ denote the impact times of the geodesic $G^{t}(x,\xi)$
with the hypersurface $H$ in Lemma \ref{graph}.  Then the
principal symbol $a_{T, \epsilon}$ of $P_{T,\epsilon}(a) $ is given by
$$ \begin{array}{lll}  a_{T, \epsilon}(x,\xi) &= & \frac{1}{T}  \sum_{j \in {\mathbb Z} }
( (1 -\chi_{\epsilon}) \gamma^{-1} a_H)(\Phi^j  \Phi_I  (x,\xi))
\cdot \chi (\frac{ t_j(x,\xi)}{T}).  \end{array}$$
\end{lem}

\begin{proof}

By Lemma \ref{VTALEM}, $P_{T,\epsilon}$ is a pseudodifferential
operator. To compute  its symbol, we use that
 the formula (\ref{psdopart}) is equivalent to,
$$  P_{T,\epsilon}(a) = \sum_{j=1}^{N_{T,\epsilon}}  \pi_{T*} \psi_{U_j} V_{\epsilon}(t;a) \,
\psi_{U_j}.  $$ By Lemma \ref{GAMMAHLEM},  $ \psi_{U_j}
V_{\epsilon}(t;a)\psi_{U_j} $ is a Fourier integral operator,
and the symbol of $P_{T,\epsilon}$ is obtained from that of $
V_{\epsilon}(t;a)$ by multiplying by $ \sigma_{\psi_{U_j}} =
\psi_{U_j}\left( \frac{\xi}{|\xi|} \right)$ and pushing forward
under $ \pi_{T*}$.

The pushforward symbol at $(x, \xi, x, \xi) \in \Delta_{T^* M
\times T^* M}$  is obtained by summing contributions from each
point of
 the  `fiber'
\begin{equation} \label{fiber} \begin{array}{ll}
\iota^{-1} (x, \xi, x, \xi) = \{(t, s, \xi', \xi') \in [-T, T]
\times \hat{C}_H : G^t(s, \xi') = (x, \xi)\}.
\end{array} \end{equation} The fiber thus consists of the impact times  $t_j(x,
\xi)$ and impact points.  By Lemma \ref{GAMMAHLEM}, and taking
into account the normalizing factor $\frac{1}{T}$ in $\VT(a)$, at
each point we get the scalar  $$  \frac{1}{T} ( (1-\chi_{\epsilon})
 a_H )(G^{t_j(x, \xi)} (x, \xi)) \chi(T^{-1} t_j(x,
\xi)) $$ times the target half-density (and Maslov factor)
calculated in (\ref{QUOTIENT}) and Corollary (\ref{BIGFACTOR}),
$$- \gamma^{-1} = - |\frac{\sqrt{\sigma^2
+ \eta_n^2}}{\eta_n}  |$$ times the symplectic volume half
density.
  Here, the minus sign is due to the fact that $\eta_n = \eta_n'$
  in this diagonal component.

\end{proof}

\begin{rem} We note that $T^*_H M $ embeds as the subset $\{(s, \xi, \xi) \in  \hat{C}_H\}$.
Hence the diagonal branch of $\Gamma_T$ may be parametrized by $j:
\R \times S^*_H M \to S^* M \times S^* M, $ $j(t, s, \xi) =
(G^t(s, \xi), G^t(s, \xi))$. This is  similar to the description
of $G^t$ as the suspension of $\Phi$ with height function $T$,
except that it does not identify $(s, \xi, T(s, \xi)) \sim
(\Phi(s, \xi), 0)$.  In terms of this parametrization,
$$ \begin{array}{lll}  a_{T, \epsilon}(s, \xi, t) &= & \frac{1}{T}  \sum_{j \in {\mathbb Z} }
 ( (1 - \chi_{\epsilon}) \gamma^{-1} a_H )(\Phi^j  (s,\xi))  \cdot
\chi (\frac{t +  T^{(j)}(s,\xi))}{T}), 
\end{array}$$ where we implicitly use the identification
$G^{t'}(s, \xi, t) = (\Phi^n(s, \xi), t' + t - T^{(n)}(s, \xi))  \,\,;  T^{(n)}(s,\xi) \leq t + t' \leq T^{(n+1)}(s,\xi).  $
\end{rem}

\subsection{Symbol of $F_{T,\epsilon}(a)$: Proof of (ii)}

By   Lemmas \ref{graph} and \ref{VTALEM}, $\Gamma_{T,\epsilon} =
 \bigcup_{j=1}^{N_{T,\epsilon}} \text{graph}  (\rcal_j)$ is  a union of canonical graphs. By definition of
 $\epsilon$ they are disjoint. Let $U_j  \times V_j\subset T^* M\times T^*M, j=1,...,N_{T,\epsilon}$
 be conic open sets which separate the sets$\text{graph}
 (\rcal_j).$ Add the conic open sets $U_{j} \times U_{j}; j=1,...,N_{T,\epsilon}$ containing the diagonal components of
  $\Gamma_{T,\epsilon}$  and let $U_0 \times V_0$ denote an additional open set so that
\begin{equation} \label{cover}
\bigcup_{j=1}^{N_{T,\epsilon}} (U_j \times V_j) \cup (U_j \times
U_j) \cup (U_0 \times V_0) = T^* M \times T^* M. \end{equation}
 Let $ \psi_{U_j} \in Op(S^0(T^*M))$ and $ \psi_{V_j} \in Op(S^0(T^*M));
  \, j=1,...,N_{T,\epsilon}$  with the property that $\psi_{U_j} \times \psi_{V_j}, \psi_{U_j}
  \times \psi_{U_j} j=1,...,N_{T,\epsilon}$ together with $\psi_{U_0} \times \psi_{V_0}$;  form a pseudo-differential partition of unity subordinate to
 the cover (\ref{cover}).

 \begin{defin} We put

 $$F_{T,\epsilon}^{(j)}(a) = \psi_{V_j} F_{T,\epsilon}(a)
 \psi_{U_j}. $$
 $F_{T,\epsilon}^{(0)}(a)$ is a smoothing operator, and (iii) holds.
 \end{defin}

 Since the canonical relation of  each term $F_{T,\epsilon}^{(j)}(a)$ is the
 graph of a canonical transformation, it carries a canonical graph
 $1/2$-density $|dx \wedge d \xi|^{1/2}$ pulled back from the
 projection to the domain of $\rcal_j$. Hence, we can identify the
 symbol of $F_{T,\epsilon}^{(j)}(a)$ with a scalar function on $T^* M$.

The symbol of $\VT(a)$ is the pushforward under
 $\pi_{t*}$ of the symbol of $V(t; a) \chi(\frac{t}{T})$.
Hence, the calculation of this symbol is analogous to that of the
 pseudo-differential part, except that now it is
only the  elements of  $\{(s, \tau, \xi_n; s, \sigma,
  \xi_n')\}$ with $\xi_n = -\xi_n'$ which  contribute to the
  composition. This
 canonical relation (with boundary) is parameterized by
 $$(t, s, \xi) \in [-T, T] \times T_H^* M \to (G^t(s,
\xi), G^t(s, r_H \xi)). $$ The symbol of $U(-t) \circ \gamma_H^*
Op(a) \gamma_H \circ U(t)$ as Fourier integral kernel in
$\dcal'(\R \times M \times M)$ is computed in  Lemma
\ref{GAMMAHLEM}.

The difference to the diagonal calculation lies with the push
forward of the symbol and canonical relation to $\bigcup_j
\mbox{graph} \rcal_j \subset T^* M \times T^*M$.  The `fiber' over
$(x, \xi, \rcal_j(x, \xi)) $ is the discrete set
$$\{(t, s, \sigma, \xi_n) \in [-T, T] \times T_H^* M : G^t(s, \sigma, \xi_n) = (x,
\xi), G^{t}(s, \sigma, - \xi_n) = \rcal_j(x, \xi)\}. $$ The second
condition only holds when  $t = t_j(x, \xi)$ and then follows from
the first. Hence, the symbol is given in graph coordinates ($t =
t_j(x, \xi), (s, \sigma, y_n, \eta_n) = \Phi_j(x, \xi)$) as the
scalar factor
\begin{equation} \label{fiosymbol}
 \sigma(F_{T,\epsilon}^{(j)})(x,\xi)  =\frac{1}{T} \sum_{j} ( (1 - \chi_{\epsilon}) \gamma^{-1} a_H) (G^{t_j(x, \xi)}(x, \xi))
\chi(\frac{t_j(x,\xi)}{T}) 
\end{equation}
times the half-density $|dx \wedge d\xi|^{\frac{1}{2}}$.

 This completes the proof of Proposition \ref{VTDECOMPa}. \qed

\section{\label{LWLFIO} Local Weyl law for homogeneous Fourier integral
operators}

In this section, we collect together the instances of the local
Weyl laws we need in the proof of Theorem \ref{maintheorem}.
We only state the first one, since a proof can be found in
\cite{Z}.

\begin{prop}  \label{LWLFIOs}  Let $C_{F} \subset T^*M-0 \times T^*M-0$ be a local canonical graph and  $F \in I^0(M \times M;C_F).$ Then,
$$\lim_{\lambda \to \infty} \frac{1}{N(\lambda)} \sum_{j: \lambda_j \leq \lambda} \langle F  \phi_j,
\phi_j
\rangle =  \frac{1}{vol(S^*M)}  \int_{S (C_F \cap \Delta_{S^*M \times S^* M})} \sigma_F
d\mu_L. $$
\end{prop}

Here, $S (C_F \cap \Delta_{S^*M \times S^* M})$ is the set of unit
vectors in the diagonal part of $C_F$. The proof is similar to
that of the next Proposition, which we use to determine the limit
state $\omega(a)$.

\begin{prop} \label{LWLH} We have,
$$\lim_{\lambda \to \infty} \frac{1}{N(\lambda)} \sum_{j: \lambda_j \leq \lambda} \langle Op_H(a)  \gamma_H \phi_j,  \gamma_H \phi_j \rangle
=  \frac{2}{vol(S^*M)}  \int_{B^*H} a_0 \gamma_{B^*H}^{-1}  |ds \wedge d \sigma|, $$ where $|ds
\wedge d \sigma|$ is symplectic volume measure on $B^*H$, and
$a_0$ is the principal symbol of $Op_H(a)$.
\end{prop}

\begin{rem} When $a = V$ is a multiplication operator, then this
follows from the pointwise Weyl asymptotic, $$\frac{1}{N(\lambda)}
\sum_{j: \lambda_j \leq \lambda}  |\phi_j(x)|^2 = 1 +
O(\lambda^{-1}).
$$
The pointwise asymptotics imply that the $L^2$-norm squares of $\gamma_H \phi_j$ are bounded on average,
 \begin{equation} \label{L2H} \frac{1}{N(\lambda)}
\sum_{j: \lambda_j \leq \lambda}  ||\gamma_H \phi_j||_{L^2(H)}^2  = Vol^{n-1}(H)  +
O(\lambda^{-1}). \end{equation}

In fact, by \cite{HoI-IV} Proposition 29.1.2, for any pseudo-differential operator $B$ of order zero on $M$,
the Schwartz kernel $K_B(t, x)$  of $U(t) B$ or $B U(t)$ on the diagonal $\Delta_{M \times M}$ is conormal with respect to $\R \times  \Delta_{M \times M}$ and if
$$\frac{\partial A(\lambda, x) }{\partial \lambda} = {\mathcal F}_{t \to \lambda} K_B(t, x)$$
the $A(\lambda, x)$ is a symbol of order $n$ with\begin{equation} \label{HORM} A(\lambda, x)  = \sum_{j: \lambda_j \leq \lambda} \phi_j(x) A \phi_j(x) \sim (2 \pi)^{-n} \int_{|\xi| < \lambda} a_0 d \xi + O(\lambda^{n-1}) \end{equation}
in the case where $A = A^*$. There is an analogous statement for $A U(t) B$.  Integrating \eqref{HORM} over $H$ gives
\begin{equation} \label{GENLWL}  \sum_{j : \lambda_j \leq \lambda} \langle \gamma_H A \phi_j(x) , \gamma_H B \phi_j \rangle_{L^2(H)} \sim C_n \lambda^n \int_{B^*_H M} a_0 b_0 ds  d \xi. \end{equation}

\end{rem}


\begin{proof} We first   prove the  local Weyl law for  $(\gamma_H^*
Op_H(a) \gamma_H)_{\geq \epsilon}$ \eqref{>ep}  on $M$, that is, we prove
\begin{equation} \label{LWLM} \lim_{\lambda \to \infty} \frac{1}{N(\lambda)} \sum_{j: \lambda_j \leq \lambda} \langle 
(\gamma_H^*
Op_H(a) \gamma_H)_{\geq \epsilon} \phi_j, \phi_j \rangle = \frac{2}{vol(S^*M)}  \int_{B^*H} a_0 (1 - \chi_{\epsilon}) \gamma_{B^*H}^{-1}  |ds \wedge d \sigma| . \end{equation}



As  in the proof of Lemma \ref{GAMMAHLEM}, \ $U(t) (\gamma_H^*
Op_H(a) \gamma_H)_{\geq \epsilon}$  is a Fourier integral operator of order $
\frac{1}{4}$ associated to the (clean) composition of the
canonical relation $\Gamma$ of $U(t)$ and $C_H$. The trace is the
further composition with $\pi_* \Delta^*$ as in \cite{DG}, where
$\Delta: \R \times M \to \R \times M \times M$ is the embedding
$(t, x) \to (t, x, x)$ and $\pi: \R \times M \to \R$ is the
natural projection. Then   $\pi_* \Delta^* U(t) \circ (\gamma_H^*
Op_H(a) \gamma_H)_{\geq \epsilon} $ has singularities at times $t$ so that $G^t (x,
\xi) = (x, \xi)$ with $(x, \xi) \in S^*_H M$. By the standard
Fourier Tauberian theorem (see \cite{HoI-IV}, vol. III) the growth
rate of the sums above are determined by the singularity at $t =
0$ of the trace, where of course all of $S^*_H M$ is fixed. Hence
the fixed point set is a codimension one submanifold of $S^*M$. If
$n = \dim M$,   $\pi_* \Delta^* U(t) \circ (\gamma_H^* Op_H(a)
\gamma_H)_{\geq \epsilon} \in I^{\frac{1}{4} + n - 1 - \frac{1}{4}}(T_0^* \R)$. 
Note that, due to the drop of one in codimension, the singularity
of the  trace loses a degree of $\half$, but due to the extra
$\half$ in the order  of  $(\gamma_H^* Op_H(a) \gamma_H)_{\geq \epsilon}$  (compared
to a pseudo-differential operator), it gains it back again. Hence
the order of the singularity is the same as for
pseudo-differential operators, and so the spectral asymptotics
have the same order in $\lambda$. The principal symbol of the
trace is determined by the symbol composition   and Lemma \ref{gammaH*gammaH}.  
. Except for the
factor of $\gamma_H^* a$, the half-density symbol is the canonical
Liouville volume form on $S^*_H M$. Since $\gamma_H^* a$ is a
pullback from $B^* H$, we can project the measure to $B^* H$ and
then we obtain the stated formula.

It remains to show that

$$\begin{array}{l} \lim_{\lambda \to \infty} \frac{1}{N(\lambda)} \sum_{j: \lambda_j \leq \lambda} \langle Op_H(a)  \gamma_H \phi_j, \gamma_H  \phi_j\rangle = \eqref{LWLM}  + o(1), \;\; \mbox{as}\;\; \epsilon \to 0, \end{array}, $$
 and in view of (\ref{op decomp}), it is enough to prove 
$$\begin{array}{l} \lim_{\lambda \to \infty} \frac{1}{N(\lambda)} \sum_{j: \lambda_j \leq \lambda} \langle 
\chi_{2 \epsilon} \gamma_H^* Op_H(a) \gamma_H \chi_{\epsilon}  \phi_j, \phi_j \rangle =  o(1), \;\; \mbox{as}\;\; \epsilon \to 0. \end{array} $$

By \eqref{CUTOFF}, there are three types of terms: one with the tangential cutoff in both cutoff positions, 
one with the normal cutoff  in both positions and two mixed ones with one tangential and one normal cutoff. Successive applications of the inequality $ab \leq \frac{1}{2}(a^2 + b^2),$ Cauchy-Schwarz and $L^2$-boundedness of $Op_H(a)$ implies that
\begin{equation} \label{tn decomp} \begin{array}{ll}
\left| \frac{1}{N(\lambda)} \sum_{\lambda_j \leq \lambda} \langle (\gamma_H^* Op_H(a) \gamma_H)_{\leq \epsilon} \phi_j, \phi_j \rangle \right| \\ \\ \leq \frac{C}{N(\lambda)} \sum_{\lambda_j \leq \lambda} \left( \| \gamma_H \chi_{\epsilon}^{(tan)} \phi_j \|_{L^2(M)}^2 + \| \gamma_H \chi_{2\epsilon}^{(tan)} \phi_j \|_{L^2(M)}^2  + \| \gamma_H \chi_{\epsilon}^{(n)} \phi_j \|_{L^2(M)}^2 + \| \gamma_H \chi_{2\epsilon}^{(n)} \phi_j \|_{L^2(M)}^2 \right). \end{array} \end{equation}

Finally, one applies the pointwise local Weyl law (\ref{GENLWL}) on $M$ to estimate the right side in (\ref{tn decomp}). It follows that
\begin{equation} \label{pt bounds}
 \frac{1}{N(\lambda)} \sum_{\lambda_j \leq \lambda}  \| \gamma_H \chi_{\epsilon,2\epsilon}^{(tan)} \phi_j  \|_{L^2(H)}^2= {\mathcal O}(\epsilon), \end{equation}
 and the same is true for the other cutoff operators $ \chi_{\epsilon,2\epsilon}^{(n)}.$

\end{proof}

We further prove a local Weyl law for $P_{T, \epsilon}$. It is a
special case of the general local Weyl law for pseudo-differential
operators, but we include as a check on the formula for
$\omega(a)$.

In the following we put $V = \mu_L(S^* M)$. 

\begin{lem} \label{LWLPT}  For any $T,\epsilon>0$  we have,  $$\begin{array}{lll} \lim_{\lambda \to \infty} \frac{1}{N(\lambda)} \sum_{j: \lambda_j \leq \lambda} \langle P_{T, \epsilon}(a)  \phi_j,
\phi_j \rangle & = & \frac{1}{V}  \int_{S^* M} \left( \frac{1}{T} \sum_{j \in \Z} ((1 - \chi_{\epsilon}) \gamma^{-1}
a_H)(\Phi^j
\Phi_I (x, \xi)) \chi(\frac{t_j(x, \xi)}{T})\right) d\mu_L  \\ && \\
& = & \frac{1}{V} \int_{S^*_H M} ((1 - \chi_{\epsilon}) \gamma^{-1} a_H) )(s,
\xi)) d\mu_{L, H} \\ && \\
 & = &  \frac{2}{V} \int_{B^*H} ((1 - \chi_{\epsilon})\gamma_{B^*H}^{-1} a_0 )(s,\sigma) ds d\sigma. 
\end{array}$$ 
\end{lem}

\begin{proof} By the local Weyl law, the limit equals $\frac{1}{V}\int_{S^*
M} \sigma_{P_{T, \epsilon}(a)} d \mu_L$. We then use Lemma
\ref{egorov} to evaluate  $\sigma_{P_{T, \epsilon}(a)}.$

The fiber $\Phi^{-1}(s, \xi)$ is the backwards orbit $G^{-t}(s,
\xi)$ for the interval $t \in [0, T^{-1}(s, \xi)]$. So we may
re-write the integral as
$$\int_{S^*_H M} \left( \frac{1}{T} \sum_{j = 1}^{N_{T, \epsilon}}
((1 - \chi_{\epsilon}) \gamma^{-1} a_H)(\Phi^j (s, \xi)) \left[
\int_{0}^{T^{(-1)}(s, \xi)}  \chi(\frac{t_j(G^{- t}(s, \xi))}{T})
dt \right] \right)   d\mu_{L, H}.
$$
Now $t_j(G^{- t}(s, \xi)) = t + T^{(j)}(s, \xi)$ when $t \in [0,T^{-1}(s,\xi)]$   and so the inner
integral equals \begin{equation} \label{TILDECHI}
 \int_{0}^{T^{(-1)}(s, \xi)}
\chi(\frac{t +  T^{(j)}(s, \xi)}{T}) dt.
\end{equation}
By the $\Phi$-invariance of $ d\mu_{L, H}$ we change variables in the
$j$th term to $(s', \xi') = \Phi^j(s, \xi)$ (and then drop the
primes) to get
$$\begin{array}{l} \int_{S^*_H M}
((1 - \chi_{\epsilon}) \gamma^{-1} a_H)((s, \xi))
\overline{\chi_T}(s, \xi)  d\mu_{L, H},\\ \\\;\;
\mbox{where}\;\;\overline{\chi_T}(s, \xi) := \sum_j \frac{1}{T}
\int_{0}^{T^{(-1)}(\Phi^{-j}(s, \xi))} \chi(\frac{t +
T^{(j)}(\Phi^{-j}(s, \xi))}{T}) dt.\end{array}
$$

We claim that $\overline{\chi_T}(s, \xi) \equiv 1$.  Indeed,
$$\begin{array}{lll}
\int_{0}^{T^{(-1)}(\Phi^{-j}(s, \xi))} \chi(\frac{t +
T^{(j)}(\Phi^{-j}(s, \xi))}{T}) dt  & = &   \int_{
T^{(j)}(\Phi^{-j}(s, \xi))}^{ T^{(j)(\Phi^{-j}(s, \xi)) +
T^{(-1)}(\Phi^{-j}(s, \xi))}} \chi(\frac{t }{T}) dt.
\end{array}$$
We observe that  $$[T^{(j)}(\Phi^{-j}(s, \xi)),
T^{(j)}(\Phi^{-j}(s, \xi)) + T^{(-1)}(\Phi^{-j}(s, \xi))]  = [
T^{(-j)}(s, \xi), T^{(-j - 1)}(s, \xi)], $$ since $T^{(j)}
\Phi^{-j}(s, \xi) = T^{(-j)}(s, \xi), $ and $T^{(j)}(\Phi^{-j}(s,
\xi)) + T^{(-1)}(\Phi^{-j}(s, \xi)) = T^{(-j - 1)}(s, \xi). $
Hence,
$$\overline{\chi_T}(s, \xi) = \frac{1}{T}
\sum_j \int_{[ T^{(-j)}(s, \xi), T^{(-j - 1)}(s, \xi)]}
\chi(\frac{t }{T}) dt = \frac{1}{T} \int_{\R} \chi(\frac{t }{T})
dt = 1.$$

\end{proof}

\section{Quantum ergodic restriction: Proof of Proposition \ref{BIGSTEP}} \label{nice case}

In this section, we prove the main result. It is the first section
in which we assume $G^t$ is ergodic.  To prove quantum ergodicity
for the eigenfunction restrictions $ \phi_{\lambda_j}|_{H},
j=1,2,...$ we follow the outline in \eqref{ME1}. 

\subsection{Proof of Proposition \ref{BIGSTEP}}

In outline the proof is as follows:

\begin{equation} \label{variance}
\begin{array}{llll}  \frac{1}{N(\lambda)}
\sum_{j: \lambda_j \leq \lambda} \left|  \langle \left(\gamma_H^*
Op_H(a) \gamma_H)_{\geq \epsilon}\right) \phi_j, \phi_j
 \rangle_{L^2(M)}  -\omega((1-\chi_{\epsilon})a)\right|^2  \\ \\
 =\frac{1}{N(\lambda)} \sum_{\lambda_j \leq \lambda} | \langle
 [\VT(a)-  \omega((1-\chi_{\epsilon})a )] \phi_{\lambda_j}, \phi_{\lambda_j} \rangle |^{2} \\ \\
=\frac{1}{N(\lambda)} \sum_{\lambda_j \leq \lambda} |  \langle [P_{T,\epsilon}(a) - \omega((1-\chi_{\epsilon})a)
 + F_{T,\epsilon}(a) ]\phi_{\lambda_j}, \phi_{\lambda_j} \rangle |^{2} + {\mathcal O}(\lambda^{-n})\\ \\
\leq  \frac{2}{N(\lambda)} \sum_{\lambda_j \leq \lambda} | \langle
(P_{T,\epsilon}(a) - \omega((1-\chi_{\epsilon})a)) \phi_{\lambda_j}, \phi_{\lambda_j}
\rangle |^{2}\\ && \\
+ \frac{2}{N(\lambda)} \sum_{\lambda_j \leq \lambda} | \langle
 F_{T,\epsilon}(a) \phi_{\lambda_j}, \phi_{\lambda_j} \rangle |^{2} + {\mathcal
O}(\lambda^{-n})
\end{array} \end{equation}

In \S \ref{psdovariance}, Corollary \ref{PQE}, we show that the $P_{T, \epsilon}$ term tends to zero and in \S \ref{g operator},
Proposition \ref{FIOdecomp},
we show that the $F_{T, \epsilon}$ term tends to zero.  The fact that $\omega((1 - \chi_{\epsilon}) a)$
is the correct constant follows from Lemma \ref{LWLPT}.


\subsection{Contribution of $ \frac{2}{N(\lambda)} \sum_{\lambda_j \leq \lambda} | \langle
(P_{T,\epsilon}(a) - \omega((1-\chi_{\epsilon})a)) \phi_{\lambda_j}, \phi_{\lambda_j}
\rangle |^{2}$ to the variance} \label{psdovariance}

It follows from the standard quantum ergodicity theorem \cite{Sch,
CV,Z3} and Lemma \ref{LWLPT} that this term tends to zero. We
briefly go over the proof using the additional time average in $r$
(see the last line of \eqref{VTFORM}). By the above decomposition,
\begin{equation} \bar{V}_{T, R, \epsilon}(a)= P_{T, R, \epsilon}(a) +
F_{T, R, \epsilon}(a), \;\; \mbox{where} \end{equation}
\begin{equation} \left\{ \begin{array}{l} P_{T, R, \epsilon}  =
\frac{1}{2 R} \int_{-R}^R U(r)^* P_{T, \epsilon} U(r) dr, \\ \\
F_{T, R, \epsilon}(a)  = \frac{1}{2 R} \int_{-R}^R U(r)^* F_{T,
\epsilon}(a) U(r) dr. \end{array} \right. \end{equation}  With the
same notation as in Lemma \ref{egorov}, we denote  the principal
symbol of $P_{T, R, \epsilon}(a) $ by  $a_{T, R, \epsilon}$.

\begin{prop}\label{PTERGODIC}  Let  $a_{T, R, \epsilon}$ be the principal symbol of $P_{T,
R, \epsilon}(a) $.  Assume that $G^t$ is ergodic. Then for any $T
>0$ and any  $\epsilon > 0$, there exists $R_0$ so that for $R
\geq R_0$,  $$ \int_{S^*M} |
  a_{T,R, \epsilon
  }(x,\xi) - \omega((1-\chi_{\epsilon})a)|^{2} d\mu_L < \epsilon. $$ \end{prop}

  \begin{proof} We first claim:

\begin{lem} \label{egorovb} With the same notation as in Lemma
\ref{egorov}, the principal symbol $a_{T, R, \epsilon}$ of $P_{T,
R, \epsilon}(a) $
 is given by
$$ \begin{array}{lll} a_{T, R, \epsilon}(x,\xi): &= & \frac{1}{2 R} \int_{-R}^R
\sigma_{P_{T, \epsilon}}(G^r(x, \xi))dr \\ && \\ & = & \frac{1}{T}
\sum_{j \in {\mathbb Z} } \frac{1}{2 R} \int_{-R}^R ((1 -
\chi_{\epsilon})\gamma^{-1} a_H)(\Phi^j \Phi_I G^r (x,\xi)) \cdot
 \chi (\frac{ t_j(G^r(x,\xi))}{T}) dr \\ && \\ & = & \frac{1}{T}
\sum_{j \in {\mathbb Z} } ((1 - \chi_{\epsilon}) \gamma^{-1}
a_H)(\Phi^j \Phi_I (x,\xi))  \cdot \left(\frac{1}{2 R} \int_{-R}^R
 \chi (\frac{ t_j(x,\xi) - r)}{T}) dr \right).
\end{array}$$
\end{lem}

The proof of the first formula is immediate from the standard
Egorov theorem combined with Lemma \ref{egorov}. In the second
line, we rewrote the formula as follows: $a_{T, R, \epsilon}$ is
the principal symbol of
$$\frac{1}{2 R} \int_{-R}^R \int_{\R} \chi(\frac{t}{T})   V_{\epsilon}(t + r;a)
dt dr  = \frac{1}{2 R} \int_{-R}^R \int_{\R} \chi(\frac{t - r}{T})  V_{\epsilon}(t;a
) dt dr.  $$ By the same symbol calculation as for $\VT(a)$,
$$\int_{\R} \chi(\frac{t - r}{T})  V_{\epsilon}(t; a
) dt  = \sum_{j \in \Z}  ((1 - \chi_{\epsilon}) \gamma^{-1}
a_H)(\Phi^j \Phi_I (x,\xi)) \chi(\frac{t_j(x, \xi) - r}{T}), $$
hence
$$a_{T, R, \epsilon}(x, \xi)= \frac{1}{T} \sum_{j \in \Z}  ((1 - \chi_{\epsilon}) \gamma^{-1} a_H)(\Phi^j
\Phi_I  (x,\xi)) \left( \frac{1}{2 R} \int_{-R}^R
\chi(\frac{t_j(x, \xi) - r}{T}) dr\right).$$

 Granted the Lemma, it follows by the mean ergodic theorem that
  $$ \lim_{R \rightarrow \infty}  \int_{S^*M} \left| a_{T,R,\epsilon}(x,\xi) -  \frac{1}{vol(S^*M)}  \int_{S^*M} \sigma_{P_{T,\epsilon}} d\mu_L  \right|^2 \, d\mu_L =0. $$
 From the formula above,
\begin{equation} \label{useful}
 \| a_{T,R,\epsilon} \|_{C^0} = {\mathcal O}_{T,\epsilon}(1)
\end{equation}
uniformly for $R>0.$
 From (\ref{useful}) and dominated convergence,  we may take the limit
$\lim_{R \to \infty}$ under the integral
  sign.  Hence, to complete the proof of the Proposition,   it suffices to show that
   $$  \omega((1-\chi_{\epsilon})a) =  \frac{1}{vol(S^*M)}  \int_{S^* M} \sigma_{P_{T, \epsilon}} d\mu_L.$$ But the last identity
  is proved in Lemma \ref{LWLPT}.

\end{proof}

By the standard quantum ergodicity argument (cited above), we then
have:

\begin{cor} \label{PQE} We have,  $$\lim_{\lambda \to \infty}  \frac{2}{N(\lambda)} \sum_{\lambda_j \leq \lambda} | \langle
(P_{T,\epsilon}(a) - \omega((1-\chi_{\epsilon})a)) \phi_{\lambda_j}, \phi_{\lambda_j}
\rangle |^{2} = 0. $$

\end{cor}

\begin{rem}

The purpose of the second time average in $r$ is just to ensure
that the maps
\begin{equation} \label{aTFORMa} a \to  a_{T,\epsilon}(x,\xi)
 = \frac{1}{T}  \sum_{j \in {\mathbb Z}} (\gamma^{-1} a_H)(G^{t_j(x,\xi)}(x,\xi))  \, \chi (T^{-1}  t_j(x,\xi))  \end{equation}
 defined in \eqref{aTFORMa}
are time averages in the sense that ${\bf 1}_{T, \epsilon} \equiv
{\bf 1}$ for all $T$. In fact, one could prove this by applying
the decomposition of Proposition \ref{VTDECOMPa} to the identity
operator and restricting to $(x, \xi)$ away from intersections of
the diagonal canonical relation from that of $F_{T, \epsilon}$.
But since it follows so quickly and easily from the second time
averaging, we presented the proof in this way.

\end{rem}

\subsection{Analysis of $F_{T,\epsilon}(a)^* F_{T,\epsilon}(a)$}
\label{g operator}

It remains to show that the limit of the  second term,
$$\limsup_{\lambda \to \infty} \frac{2}{N(\lambda)} \sum_{\lambda_j \leq \lambda} | \langle
 F_{T,\epsilon}(a) \phi_{\lambda_j}, \phi_{\lambda_j} \rangle
 |^{2}$$
  in (\ref{variance}) tends to zero as $T \to \infty$. We do not
  need to average in $r$ for this term.

By the  Schwartz inequality
  $$| \langle
 F_{T,\epsilon}(a) \phi_{\lambda_j}, \phi_{\lambda_j} \rangle
 |^{2} \leq  \langle
 F^*_{T,\epsilon}(a) F_{T,\epsilon}(a) \phi_{\lambda_j}, \phi_{\lambda_j}
 \rangle, $$ so it suffices to prove
  \begin{prop}  \label{FIOdecomp}
 $$\limsup_{\lambda \to \infty} \frac{1}{N(\lambda)} \sum_{\lambda_j \leq \lambda}  \langle
 F^*_{T,\epsilon}(a) F_{T,\epsilon}(a) \phi_{\lambda_j}, \phi_{\lambda_j}
 \rangle = o(1), \;\;\;(T \to \infty). $$

 \end{prop}

 \begin{proof} The main step is to prove the

\begin{lem} Under the measure zero microlocal reflection symmetry condition in Definition \ref{ANC},
\begin{equation} \label{fiobound}
\begin{array}{l} \limsup_{\lambda \to \infty} \frac{1}{N(\lambda)} \sum_{\lambda_j
\leq \lambda}  \langle
 F^*_{T,\epsilon}(a) F_{T,\epsilon}(a) \phi_{\lambda_j}, \phi_{\lambda_j}
 \rangle \\ \\ =  \frac{1}{T^2}
\sum_{j }  \int_{S^* M} \left| (\gamma^{-1}  (1-\chi_{\epsilon})  a_H (
G^{t_j(x,\omega)}(x,\omega))  \chi (T^{-1} t_j(x,\omega))
\right|^2 d\mu_L.
 \end{array} \end{equation}
 \end{lem}

 \begin{proof}  From Proposition \ref{VTDECOMPa}
 $ F_{T,\epsilon}(a) = \sum_{j =1}^{N_{T,\epsilon}} F_{T,\epsilon}^{(j)}(a)$ where for each $j,$
 $  F_{T,\epsilon}^{(j)}(a) \in I^0 (M \times M; \Gamma_{T,\epsilon}^{(j)}), $  with $\Gamma_{T,\epsilon}^{(j)} = \text{graph}  \, {\mathcal R}_{j}$.  We then write
 $$ F_{T,\epsilon}(a)^* F_{T,\epsilon}(a) = I_{T,\epsilon}(a) + II_{T,\epsilon}(a)$$
 where we break up the sum into diagonal, resp. off-diagonal
 parts,

\begin{equation} \label{I, II} \begin{array}{ll}
I_{T,\epsilon}(a) = \sum_{|j| \leq N_{T,\epsilon}} F_{T,\epsilon}^{(j)}(a)^* F_{T,\epsilon}^{(j)}(a), \\ \\
II_{T,\epsilon}(a) = \sum_{j \neq k; |j| \leq N_{T,\epsilon}, |k|
\leq N{T,\epsilon}} F_{T,\epsilon}^{(j)}(a)^*
F_{T,\epsilon}^{(k)}(a). \end{array} \end{equation}

From the symbol computations in Proposition \ref{VTDECOMPa}  (see
(\ref{fiosymbol}) ) and  the fact that $WF'(
F_{T,\epsilon}^{(j)}(a)) =  \Gamma_{T,\epsilon}^{(j)},$ a
canonical graph, it follows that
$$ F_{T,\epsilon}^{(j)}(a)^* F_{T,\epsilon}^{(j)}(a) \in Op(S^{0}_{cl}(T^*M)),$$ with
\begin{equation} \label{I1} \begin{array}{ll}
\sigma  ( \, F_{T,\epsilon}^{(j)}(a)^* F_{T,\epsilon}^{(j)}(a) \,
) (x,\xi)= & \frac{1}{T^2} \,  | \, a_H \gamma^{-1} (1 -
\chi_{\epsilon}) (\Phi^j \Phi(x, \xi))   \chi (T^{-1} t_j(x,\xi))
\, |^{2}
 \end{array}
 \end{equation} By the local Weyl law,
\begin{equation} \label{LWLI}  \begin{array}{l} \lim_{\lambda \rightarrow \infty}
\frac{1}{N(\lambda)} \langle   I_{T,\epsilon}(a) \phi_{\lambda_j},
\phi_{\lambda_j} \rangle  = \frac{1}{T^2} \sum_j \int_{S^*M}  \, |
\, a_H \gamma^{-1} (1 - \chi_{\epsilon}) (\Phi^j \Phi(x, \xi))
\chi (T^{-1} t_j(x,\xi)) \, |^{2} d\mu_L.  \end{array}
\end{equation}
This is the desired limit.

To complete the proof of the Lemma we need to show that the limit
of the off-diagonal sum
$$II_{T,\epsilon}(a) = \sum_{j \neq k; |j| \leq N_{T,\epsilon}, |k| \leq
N_{T,\epsilon}} \left( \lim_{\lambda \rightarrow \infty}
\frac{1}{N(\lambda)} Tr \Pi_{[0, \lambda]}
F_{T,\epsilon}^{(j)}(a)^* F_{T,\epsilon}^{(k)}(a) \right),
$$ is zero. To prove this,
we recall that
$$ WF'(F_{T,\epsilon}^{(j)}(a)) = \text{graph}  (\rcal_j),$$ hence the fixed point set  in the local Weyl
law integral  (\ref{oldWeyl}) is the set $\{\rcal_j = \rcal_k\}$
and that
\begin{equation} \label{II1}\begin{array}{lll}
\sigma  ( \, F_{T,\epsilon}^{(j)}(a)^* F_{T,\epsilon}^{(k)}(a) \,
) (x,\xi) & = & \frac{1}{T^2} \,   \,(1 - \chi_{\epsilon})
\gamma^{-1}
a_H ( G^{t_j(x,\xi)}(x,\xi)) \chi (T^{-1} t_j(x,\xi)) \\ && \\
&\cdot & (1 - \chi_{\epsilon}) \gamma^{-1} a_H (
G^{t_k(x,\xi)}(x,\xi)) \chi (T^{-1} t_k(x,\xi)).
\end{array}
\end{equation}

It follows from the local Weyl law for Fourier integral operators
(\ref{oldWeyl}) that,
\begin{equation} \begin{array}{lll}\lim_{\lambda \to \infty} \frac{1}{N(\lambda)} Tr \Pi_{[0, \lambda]} F_{T,\epsilon}^{(j)}(a)^*
F_{T,\epsilon}^{(k)}(a) & = &  \frac{1}{T^2} \int_{S\{\rcal_j =
\rcal_k\}_{T,\epsilon}} (\gamma^{-1} (1 - \chi_{\epsilon})  a_H) (
G^{t_j(x,\omega)}(x,\omega)) \chi (T^{-1} t_j(x,\omega))\\ && \\
&&  \, (\gamma^{-1} (1 - \chi_{\epsilon})  a_H) (
G^{t_k(x,\omega)}(x,\omega)) \chi (T^{-1} t_k(x,\omega)) d\mu_L.
\end{array}\end{equation}

We now show that  the domain of integration is empty when $j \not=
k$ when condition (\ref{ANC}) is satisfied, hence that this term
is zero.


As discussed in the introduction (see \eqref{RJERKintro}), 
for $(x, \xi) \in \dcal_{T,
 \epsilon}^{(j)} \cap  \dcal_{T,
 \epsilon}^{(k)} \cap S^*M$,
the condition $ \rcal_j(x, \xi) = \rcal_k(x, \xi)$ on a set of positive measure is equivalent to the condition
in Definition \ref{ANC}.


%


\end{proof}

We now complete the proof of the Proposition: The right side of
Lemma \ref{fiobound} differs from that of Proposition
\ref{PTERGODIC} in three ways. First, and most importantly, it is
normalized by $\frac{1}{T^2}$ rather than $\frac{1}{T}$. Second,
it is a sum $\sum_j$ of squares and not the square of the sum; and
third, we do not subtract $\omega(a)$. Due to the last two
properties, the limit estimate is {\it not} due to ergodicity.

Rather, we estimate the right side of (\ref{LWLI}) by
$$ \,  \leq \frac{1}{T} ||
\, a_H \gamma^{-1} (1 - \chi_{\epsilon})||_{C^0} \int_{S^*M}
\left( \frac{1}{T} \sum_j \chi (T^{-1} t_j(x,\xi)) \right) d\mu_L
\leq \frac{1}{T} || \, a_H \gamma^{-1} (1 -
\chi_{\epsilon})||_{C^0},
$$
where aside from bounding the $a_H$ factor we also use that
$\chi^2 \leq \chi$ since $0 \leq \chi \leq 1$. Here we used that,
as in Lemma \ref{LWLPT}, the evaluation $$ \int_{S^*M}  \left(
\frac{1}{T} \sum_j \chi (T^{-1} t_j(x,\xi)) \right) d\mu_L  =
\frac{1}{T} \sum_j \int_{S^*_H M} \left(\int_{0}^{T^{(-1)}(s,
\xi)} \chi(\frac{t_j(G^{- t}(s, \xi))}{T}) dt) \right) d\mu_{L, H} = 1.
$$ Hence the term $I$ satisfies the limit estimate of Proposition
\ref{FIOdecomp}, completing its proof. 

\end{proof}

This completes the proof of Proposition \ref{BIGSTEP}. 







\section{Proof of Theorem \ref{maintheorem}}

\subsection{\label{finish} Completion of proof of Theorem \ref{maintheorem}}
\subsubsection{Decomposition of matrix elements}

First, we prove the asymptotic decomposition formula in (\ref{op decomp}) for matrix elements. We have the operator decomposition
\begin{equation} \label{op decomp2}
\gamma_H^* Op_H(a)
 \gamma_H = (\gamma_H^* Op_H(a)
 \gamma_H) )_{\geq \epsilon} +  (\gamma_H^* Op_H(a)
 \gamma_H) )_{\leq \epsilon}  + K_{\epsilon}, \end{equation}
 where
 \begin{equation} \label{Kop1}
 K_{\epsilon} := \chi_{\epsilon/2} \gamma_H^* Op_H(a) \gamma_H (1-\chi_{\epsilon}) +  (1-\chi_{2\epsilon}) \gamma_H^* Op_H(a) \gamma_H \chi_{\epsilon}. \end{equation}
 By wave front calculus, we can further decompose $K_{\epsilon} = K_{\epsilon}' + K_{\epsilon}''$ where, $WF'(K_\epsilon') \subset T^*M-0 \times T^*M-0$ and 
 $WF(K''_{\epsilon}) \subset 0_{T^*M} \times N^*H \cup N^*H \times 0_{T^*M}. $  To estimate the matrix elements 
 $\langle K''_{\epsilon} \phi_j, \phi_j \rangle$ we let  $\chi_{0} \in C^{\infty}_{0}(T^*M)$ and note that for any $N>0,$  the operator $ \chi_0 (-\Delta + 1)^{N}  \in \Psi^0(M).$   $L^2$-boundedness of $ \chi_0 (-\Delta + 1)^{N} $ implies that $\| \chi_0 \phi_j \|_L^2 = {\mathcal O}(\lambda_j^{-\infty})$ and by replacing $\chi_0$ with $\Delta^m \chi_0 \in \Psi^{-\infty}(M)$ for any $m >0,$ it follows by an application of the Garding inequality that $\| \chi_0 \phi_j \|_{C^k(M)} = {\mathcal O}_{k}(\lambda_j^{-\infty})$ for any $k \in \Z^+.$ Consequently, by $L^2$-boundedness of $Op_H(a) \in \Psi^0(H),$
$$  | \langle  Op_H(a) \gamma_H \chi_0 \phi_j, \gamma_H  \phi_j \rangle_{L^2(H)} | \leq C \| \gamma_H \chi_0 \phi_j \|_{L^2(H)} \| \gamma_H \phi_j \|_{L^2(H)} = {\mathcal O}(\lambda_j^{-\infty}). $$
 Here, one can use the universal restriction bound  $\| \gamma \phi_j \|_{L^2(H)} = {\mathcal O}(\lambda_j^{1/4})$ \cite{BGT} to bound  the $\| \gamma_H \phi_j \|_{L^2(H)}$-term on the right side of the Schwarz inequality, but any crude polynomial bound in $\lambda_j$ will suffice.  The argument for $\chi_0 (\gamma_H^* Op_H(a) \gamma_H)$ is very similar and also gives the ${\mathcal O}(\lambda_j^{-\infty})$ bound for matrix elements. As a result, 
 \begin{equation} \label{matrixelements2}
\langle K_{\epsilon}'' \phi_j, \phi_j \rangle = {\mathcal O}_{\epsilon}(\lambda_j^{-\infty}). \end{equation}
To estimate the matrix elements $\langle K'_{\epsilon} \phi_j, \phi_j \rangle$, we note that by time-averaging,
\begin{equation} \label{timeaverage1} \begin{array}{ll}
\langle K_{\epsilon}' \phi_j, \phi_j \rangle_{L^2(M)} =  \left\langle \frac{1}{T} \left( \int_{0}^T U(-t) \chi_{\epsilon/2} \gamma_H^* Op_H(a) \gamma_H (1-\chi_{\epsilon}) U(t) dt \right) \phi_j, \phi_j \right\rangle \\ \\
+ \left\langle \frac{1}{T} \left( \int_{0}^T U(-t) (1-\chi_{2\epsilon}) \gamma_H^* Op_H(a) \gamma_H \chi_{\epsilon} U(t) dt \right) \phi_j, \phi_j \right\rangle, \end{array} \end{equation}
and by general wave front calculus (\cite{HoI-IV} Theorem 8.2.14),
\begin{equation} \label{wf} \begin{array}{lll}
WF'  \left( \int_{0}^T U(-t) \chi_{\epsilon/2} \gamma_H^* Op_H(a) \gamma_H (1-\chi_{\epsilon}) U(t) dt \right) \\ \\
 \subset \{ (x,\xi; x',\xi') \in T^*M-0 \times T^*M-0;  \exists t  \in (-T,T),    \,\, \exp_x t \xi = \exp_{x'} t \xi' = s \in H, \\ \\
G^t(x, \xi) |_{T_sH} = G^t(x', \xi') |_{T_s H}, \, G^{t}(x,\xi) \in \text{supp} (\chi_{\epsilon/2}), \,\, G^t(x',\xi') \in \text{supp}(1-\chi_\epsilon),  \;\; |\xi| =
|\xi'|\} \,\, \,  = \,\,\, \emptyset. \end{array} \end{equation}
We note that the wave front in (\ref{wf}) is empty  since in addition to the requirement that $G^t(x,\xi)|_{T_sH} = G^t(x',\xi')|_{T_sH},$  there is the condition that $|G^{t}(x,\xi)| = |G^{t}(x',\xi')|$ which is imposed by the integration over $t \in (0,T).$  Since supp $(\chi_{\epsilon/2})  \cap $ supp $(1-\chi_{\epsilon}) = \emptyset$ and $r_H^* \chi_{\epsilon} = \chi_{\epsilon},$  this is impossible. Similarily, the second time-averaged operator on the RHS of (\ref{timeaverage1}) also has empty wave front. Thus,
\begin{equation} \label{matrixelements1}
\langle K_{\epsilon}' \phi_j, \phi_j \rangle = {\mathcal O}(\lambda_j^{-\infty}) \end{equation}
and  in view of (\ref{matrixelements2}), this proves the decomposition formula in (\ref{op decomp}). 

To   complete the proof of Theorem \ref{maintheorem} we note that by (\ref{op decomp}) and Cauchy-Schwarz,
\begin{equation} \label{l1} \begin{array}{lll}
\limsup_{\lambda \rightarrow \infty} \frac{1}{N(\lambda)} \sum_{\lambda_j \leq \lambda} | \langle Op_H(a) \gamma_H \phi_j, \gamma_H \phi_j\rangle - \omega(a) | \\ \\
 \leq  \limsup_{\lambda \rightarrow \infty} \left( \frac{1}{N(\lambda)} \sum_{\lambda_j \leq \lambda} | \langle (\gamma_H^* Op_H(a) \gamma )_{\geq \epsilon} \phi_j, \phi_j\rangle_M  - \omega((1-\chi_{\epsilon})a) |^2 \right)^{1/2} \\ \\
+ \limsup_{\lambda \to \infty} \frac{1}{N(\lambda)} \sum_{\lambda_j \leq \lambda} | \langle (\gamma_H^* Op_H(a) \gamma_H)_{\leq \epsilon} \phi_j, \phi_j \rangle_M  | + | \omega( (1-\chi_{\epsilon})a) - \omega(a) |. \end{array} \end{equation}
By Proposition \ref{BIGSTEP}, the first term on the RHS of the inequality (\ref{l1}) vanishes.
The second term is just
\begin{equation} \label{sing} \begin{array}{lll}
\limsup_{\lambda \to \infty} \frac{1}{N(\lambda)} \sum_{\lambda_j \leq \lambda} | \langle Op_H(a) \gamma_H \chi_{\epsilon} \phi_j,  \gamma_H \chi_{2\epsilon}   \phi_j \rangle_H  | \\ \\
 \leq C \limsup_{\lambda \to \infty} \frac{1}{N(\lambda)} \sum_{\lambda_j \leq \lambda}  \|  \gamma_H \chi_{\epsilon} \phi_j \|_{L^2(H)}  \, \|  \gamma_H \chi_{2\epsilon}   \phi_j \|_{L^2(H)}, \end{array} \end{equation}
which follows from the $L^2$-boundedness of $Op_H(a).$  We use the $ab \leq \frac{1}{2} (a^2 +b^2)$ inequality to get that  the last line in (\ref{sing}) is bounded by
\begin{equation}\label{sing2}
 \frac{C }{2} \limsup_{\lambda \to \infty} \frac{1}{N(\lambda)} \sum_{\lambda_j \leq \lambda}   \, \left( \, \|  \gamma_H \chi_{\epsilon} \phi_j \|^2_{L^2(H)} \, +  \, \|  \gamma_H \chi_{2\epsilon}   \phi_j \|^2_{L^2(H)} \, \right) = {\mathcal O}(\epsilon),\end{equation}
where the last estimate follows immediately from  the local Weyl law  in (\ref{GENLWL}). Since $\epsilon >0$ is arbitrary, we finally take the $\epsilon  \to 0^+$ limit in (\ref{l1}) and that completes the proof of Theorem \ref{maintheorem}.

\section{\label{HHP} Curves in $\H^2/ \Gamma$ with zero measure of microlocal symmetry}

We now illustrate Theorem \ref{maintheorem} in some important
examples.

\subsection{Geodesic circles of hyperbolic  surfaces : Proof of Corollary \ref{HYPCIRCLE} }

We first consider the case where $H = C_r$ is an embedded geodesic
circle of a small radius $r$ in a hyperbolic surface $\H^2/\Gamma$
where $\H^2$ is the hyperbolic plane and $\Gamma \subset PSL(2,
\R)$ is a co-compact Fuchsian group.  A geodesic circle is a
separating curve, so there are two global sides of $H$
corresponding to the interior an exterior. Corollary
\ref{HYPCIRCLE} follows from the following

\begin{lem} For any finite area hyperbolic surface, no distance
circle can have positive measure of microlocal symmetry.
\end{lem}

\begin{proof}  We uniformize and consider the
$\Gamma$-orbit $\Gamma C_r$ of the distance circle. It  is a union
of disjoint geodesic circles of $\H^2$. If $C_r$ has a positive
measure of microlocal reflection symmetry, then there exist  two
components, $\sigma_+ H, \sigma_- H$ with $\sigma_{\pm} \in
\Gamma$, and an open set $U^* \subset B^* C_r$, so that geodesics
$\xi_{\pm}(s, \sigma)$ defined by $(s, \sigma) \in U$ hit the
components  $\sigma_{\pm} H$ at the same time and at points which
are equivalent modulo $\Gamma$. Since the return maps are real
analytic on their open sets of definition, the left and right
return maps must coincide for all $(s, \sigma)$ so that $\exp_s
\xi_{\pm} (x, \sigma)$ hits $\sigma_{\pm} H$. That is,  $\exp_s t
\xi_+(s, \sigma)$ hits $\sigma_+ H$ at the same time that $\exp_s
t \xi_-(s, \sigma)$ hits $\sigma_- H$ and the points are
equivalent under the action of $\Gamma$.

Let $U$ denote the set of footpoints of $U^*$.  With no loss of
generality, we position the closest   point of   $U$ to $\sigma_+
C_r$ at the origin $i \in \H^2$.  We also rotate the configuration
so that $T_i C_r $ is the vertical axis, so that the minimizing
geodesic from $i$ to $\sigma_+ C_r$ has horizontal initial tangent
vector.  We denote the distance between $C_r$ and $\sigma_+ C_r$
by $d$. By the assumption that the return maps agree, the geodesic
ray in the reflected horizontal direction  hits $\sigma_- C_r$ at
time $d$.

We now claim that  $\sigma_+ C_r = \epsilon \sigma_- C_r$ where
$\epsilon$ is the Euclidean reflection through the tangent line
$T_- H$.

To see this, we consider the    two extreme geodesic rays
emanating from tangent vectors at the center $i$ that hit
$\sigma_+ C_r$ tangentially.  Their reflections through $T_i C_r$
must be rays that hit $\sigma_- C_r$ tangentially,  since the
domains of the left/right return maps coincide and so the the
reflection of extreme rays (on the boundary of the domains of
analyticity) must be invariant under reflection. Also, the
distances to the tangential intersections with $\sigma_{\pm} C_r$
are equal. Since a circle is determined by three points,
$\sigma_{\pm } C_r$ are determined by the nearest points to $C_r$
and the points where the tangential rays based at $i$ intersect
it. Since this data is $\epsilon$ invariant, we must have
$\sigma_+ C_r = \epsilon \sigma_- C_r$.

We use the same analysis to prove that it is impossible for a
circle  $C_r$ to have a positive measure of reflection symmetry.
In fact, this is quite easy to see because the configuration
cannot be $\epsilon$-symmetric due to the fact that $C_r$ lies on
the left side of the symmetry axis. Here is a more formal proof.

For  each $s \in C_r$, there is a maximal interval $I_s^+ \subset
S^*_s \H^2$ of unit tangent vectors pointing to the exterior of
$C_r$ whose geodesics hit $\sigma_+ C_r$. There is also a maximal
domain $I_s^-$ of inward pointing unit vectors whose geodesics hit
$\sigma_- C_r$. If the $\pm$ return maps coincide, one must have
$r_{C_r} I_s^{\pm} = I_s^{\mp}$ for all $s \in U$. The boundary of
the domain of analyticity for the $\pm$ return map consists of the
exterior directions whose geodesic rays hit $\sigma_{\pm} C_r$
tangentially. It is clear that $I_s^{\pm}$ shrinks to a point when
the tangent line to $C_r$ at $s$ hits $\sigma_{\pm} C_r$
tangentially. There are two such points $s_j^+$ for $\sigma_+
C_r$, and $U = [s_1^+, s_2^+]$ is the arc of $C_r$ with boundary
points $s_j^+$ which contains the origin. If the $\pm$ return maps
were the same, this would have to be the same as the corresponding
interval $[s_1^-, s_2^-] \subset C_r$. In particular, the geodesic
tangent to $C_r$ at $s_1^+$ would have to be simultaneously
tangent to $\sigma_{\pm} C_r$ and to hit them at the same
distance. In fact the pair of circles has only two common tangent
circles/lines. We ignore the point that they need not be geodesics
of $\H^2$, i.e. need not hit the boundary orthogonally.  Since the
configuration of two circles $\sigma_{\pm}$ is invariant under
$\epsilon$, the common tangent must be $\epsilon$-invariant.
Define the midpoint of the common tangent to be the point where
the distance along the tangent is the same to $\sigma_- C_r$ and
$\sigma_+ C_r$. Then the midpoint must be $\epsilon$-invariant.
But also the distance must be the same on the tangent from its
intersection with $C_r$. Hence, the midpoint is the intersection
point of the common tangent to $\sigma_- C_r, C_r, \sigma_+$. But
this midpoint cannot be $\epsilon$-invariant since $C_r$ lies on
one side of the fixed line of $\epsilon$, and the intersection of
the common to tangent with
 $H$ occurs at a point in the left half-plane with
respect to $T_i C_r$.

This contradiction shows  that the $\pm$ return maps for $ C_r$
cannot coincide for any pair of components $\sigma_{\pm} C_r$,
hence that $C_r$ satisfies the condition of Definition \ref{ANC}.

\end{proof}

The same proof generalizes to distance spheres of hyperbolic
quotients of any dimension. It generalizes also to certain
negatively curved manifolds, for which there exists an isometric
involution fixing the tangent plane of a distance sphere at some
point.

\subsection{Closed geodesics in hyperbolic surfaces: }

\begin{prop} Suppose that $\gamma $ is a closed geodesic of $ \H^2/ \Gamma$
with a  positive measure of microlocal symmetry. Then there exists
an orientation reversing involution $\epsilon$ of $\H^2$
preserving the axis of $\gamma$, and a  three generator subgroup
$\langle \gamma,\sigma_+, \sigma_- \rangle$ such that $\sigma_- =
\epsilon \sigma_+ \epsilon. $
\end{prop}

\begin{proof}

In the universal cover, we pick one component of the orbit $\Gamma
\mbox{Axis}\;(\gamma)$  of the axis of the geodesic. With no loss
of generality we may assume it is the vertical geodesic $i \R_+$.
We orient the geodesic so that it moves towards $\infty$,  and we
choose its left and right sides as the $\mp $ sides.

Let $\epsilon: \H \to \H$ be the orientation-reversing isometric
involution $\epsilon(x, y) = (-x, y)$, i.e.  $\epsilon z = -
\bar{z}$.

\begin{lem} If the left  return map corresponding to
$\sigma_- H$ coincides at a common time on a set of positive measure of $B^*H$ with
the right return map corresponding to $\sigma_+ H$, then $\epsilon
\sigma_- \epsilon = \sigma_+$. \end{lem}

We now consider what happens if
\eqref{rH} or equivalently \eqref{RJERKintro}  holds  on a set of positive measure
 of $B^* i \R_+$. Since the maps are real analytic, this implies
 that  $\pcal_{j, +} \equiv \pcal_{k, -}$ where the return indices are determined by
the condition that the return times are the same.  The return maps are
 given by $d \sigma_{\pm} \circ \alpha'_{\xi_{\pm}}(T_j(\xi_{+}))$
 where $\sigma_{\pm}^{-1}$  are the
elements of $\Gamma$ taking $\mbox{Axis}(\gamma) = i \R_+$ to the
components hit by $\alpha_{\xi_{\pm}}$, and $T_j(\xi_{+})$  is the common
times  when $\alpha_{\xi_{\pm}}$ hit the respective components.


Clearly $\epsilon \alpha_{\xi_{\pm}} = \alpha_{\xi_{\mp}}$. Hence
if $\pcal_{j, +} \equiv \pcal_{k, -}$ where $j, k$ are related as above,  then $\epsilon \sigma_-
\mbox{Axis}(\gamma)) = \sigma_+ \mbox{Axis}(\gamma)$. But then $
\sigma_- \epsilon \sigma_+^{-1}$ is an isometry of $\H$ which
fixes $\mbox{Axis}(\gamma)$ pointwise. The only such possible
isometries are the identity and $\epsilon$ and by considering
orientations it is clear that $ \sigma_- \epsilon \sigma_+^{-1} =
\epsilon.$

We now consider any component $\sigma \mbox{Axis}(\gamma)$ with
$\sigma \notin \Gamma_{\gamma}$.  Given one of its points we find
the closest point of $\mbox{Axis}(\gamma)$. The minimizing
geodesic then intersects $\mbox{Axis}(\gamma)$ and $\sigma
\mbox{Axis}(\gamma)$ orthogonally and on $\mbox{Axis}(\gamma)$
projects to the zero covector. Then by assumption, $\epsilon$ of
this minimizing geodesic is the minimizing geodesic from this
point to another component $\tau \mbox{Axis} (\gamma)$. But then
$\epsilon \sigma \epsilon = \tau$.

 \end{proof}

 Note that the quotient of $\H^2$ by the two-generator subgroup
$\langle \gamma,\sigma_+
 \rangle$ is an infinite area pair of pants with three simple
 closed geodesics corresponding to the axis $\mbox{Axis}(\gamma)$ of
 $\gamma$ and its translates by $\sigma_+, \gamma \sigma_+$. The
 quotient by $\langle \gamma,\sigma_-
 \rangle$ is a second pair of pants. If we truncate each pair of
 pants at the simple closed geodesic $\gamma$ and glue them
 together, we obtain the quotient by the three element subgroup.
 Thus, there exists a locally isometric $\Z_2$-infinite sheeted cover
 $\pi: \H^2/ \langle \gamma,\sigma_+, \sigma_- \rangle \to
 \H^2/\Gamma$. To our knowledge, such a cover may exist without $ \H^2/\Gamma$
 possessing a $\Z_2$ symmetry.

However, a generic compact hyperbolic surface does not have a
triple of elements $\gamma, \sigma_+, \sigma_-$ with the property
above. Indeed, it suffices to show that for any closed geodesic
$\gamma$, and any pair of elements $\sigma_{\pm}$ satisfying the
relation above, there exist  infinitesimal deformations which
destroy the relation. Such a deformation is given by twisting
along $\gamma$, in  twist-length coordinates on moduli space.

\subsection{Closed horocycles for $\Gamma = SL(2, \Z)$}

 Now we consider the case where $H$ is a closed horocycle $H$ of the
 modular curve $\H/ SL(2,\Z)$. Numerical studies of the quantum
 ergodic property of restrictions of eigenfunctions to horocycles
 are given in \cite{HR}.

  For simplicity we assume $H$ is an
 embedded horocyle in the parabolic end. It is a separating curve,
 and if we orient the end `upwards' the two sides of $H$ may be
 visualized as upward pointing and downward pointing.

 Except for the upward vectors orthogonal to $H$, all upward
 vectors define geodesics which return to $H$ after a sojourn in
 the end. The orthogonal geodesics to $H$ run out to infinity and never
 return. In the standard tesselation of $\H$ by fundamental
 domains of $SL(2, \Z)$, the horocycle is a horizontal line $y =
 C$ and the upward geodesics correspond to half-circles orthogonal
 to $\R$ which intersect the horizontal line in two points.

\begin{prop} Suppose that $H $ is a closed horocycle  of $ \H^2/ SL(2,
\Z)$. Then $H$ has a zero measure of microlocal symmetry.
Consequently, restrictions of eigenfunctions to $H$ are quantum
ergodic.
\end{prop}

\begin{proof} We argue by contradiction again. If $H$ had a
positive measure of microlocal symmetry, there would have to exist
horocycles $\sigma_+ H, \sigma_- H$ such that the hitting times
and return maps from some open set of $S^*_H \H^2$ and its
reflection through $H$ were the same modulo the action of
$\Gamma$.

Since $\H^2$ is a symmetric space, there exists an inversion
symmetry $s_p$ at each point $p$, i.e. an involutive isometry that
fixes $p$ and reverses all geodesics through $p$. In the case of
$i$ it is given by $s_i(z)=  -\frac{ 1}{z}$. If $p \in H$ and
compose $s_p$  with the reflection symmetry $\epsilon_p$  with
respect to the vertical geodesic through $p$, then $\epsilon_p
\circ \sigma_p$ is an isometry of $\H^2$ which reflects $T_p \H^2$
through $T_p H$; at i, it is $w(z) = \frac{1}{z}$.

As above, we can reconstruct $\sigma_- H$ from $H, \sigma_+ H$
using only the geodesics from one point of $H$, which we take to
be $i$ again with no loss of generality (so that $H$ is $y = 1$).
Since $\epsilon_i \circ \sigma_i$ takes the upward `interval' of
geodesics which hit $\sigma_+H$ to the `downward interval' that
hits $\sigma_- H$ and since the hitting times and positions are
the same, we must have $(\epsilon_i \circ \sigma_i)  \sigma_+
(\epsilon_i \circ \sigma_i)^{-1} = \sigma_-$. But the same
argument applies to any point $p \in H$ for which there exists an
interval of geodesics hitting $\sigma_+ H$. Then we get
$(\epsilon_p \circ \sigma_p)  \sigma_+ (\epsilon_p \circ
\sigma_p)^{-1} = \sigma_-$. But this implies $(\epsilon_p \circ
\sigma_p)  \sigma_+ (\epsilon_p \circ \sigma_p)^{-1} = (\epsilon_q
\circ \sigma_q)  \sigma_+ (\epsilon_q \circ \sigma_q)^{-1}$ for
all $(p, q)$ in some interval on $H$. If $\gamma_{p, q} =
(\epsilon_p \circ \sigma_p)^{-1} (\epsilon_q \circ \sigma_q)$ then
we would have $g_{p,q} \sigma_+ = \sigma_+ g_{p,q}$, which implies
that $g_{p,q} \in G_{\sigma_+}$, the centralizer of $\sigma_+$ in
$G = PSL(2, \R)$. This is a group of hyperbolic elements which is
conjugate to the real diagonal matrices, and in particular must
fix the endpoints of the axis of $\sigma_+$. Concretely, if $N =
\{n_x =
\begin{pmatrix} 1 & x \\ & \\ 0 & 1\end{pmatrix} \}$ is the
unipotent subgroup, and if $p = n_x i, q = n_u i$ then $g_{pq} =
n_x w n_{u - x} w^{-1} n_u. $ It is easy to see that the elements
$n_x w n_{u - x} w^{-1} n_u$ cannot all fix the same two points of
$\R = \partial \H^2$. Indeed, if $t$ were such a fixed point then
$n_x w n_{u - x} w^{-1} n_u t = \frac{u + t}{1 - (u - x) (u +t)} +
x $ would equal $t$ for all $x, u$. This is absurd since as $x \to
u$ it becomes $2x + t$.

This contradiction concludes the proof.

\end{proof}

\section{Proof of Theorem \ref{sctheorem}} \label{sc}

In this appendix, we 
convert the proof of Theorem \ref{maintheorem} into the
semi-classical version Theorem \ref{sctheorem}.  The proof parallels the one in homogeneous case but with two (minor) differences: 1)  In the semiclassical case, we will need to cut-off the Fourier integral operators appearing in Proposition \ref{VTDECOMPa} in order to apply the compactly-supporrted semiclassical Fourier integral operator calculus in \cite{GuSt2}. A key issue is mass concentration for eigenfunctions and their  restrictions to $H.$  For completeness, we review the relevant results here (see \cite{Zw} for more detail). 2) The second difference deals with the role of the $N^*H$. In the homogeneous case, one must remove a conic neighbourhood of $N^*H$ (see (\ref{PSIEP})) to ensure that  $ \chi(x_n) a(s,\sigma)$ is a polyhomogeoneous symbol on $T^*M$. In the semiclassical case, because of mass localization (see Lemma (\ref{mass estimate})), for the proof of Theorem (\ref{sc}) it suffices to consider matrix elements $\langle Op_{h_j}(a) \phi_{h_j}, \phi_{h_j} \rangle$ where $a \in C^{\infty}_{0}(T^*H).$  Under the tangential projection $\pi_H: T_H^*M \rightarrow T^*H,$   $\pi_H(N^*H) = (0)_{T^*H}$ and the zero section $0_{T^*H} = \{ (s,\sigma=0); s \in H)\}$ is of no special interest in the semiclassical case. 

\subsection{ Semiclassical symbols}
A natural class of semiclassical symbols \cite{Zw} is given by
\begin{equation} \label{sc1}
S^{m,k}(T^*M \times [0,h_0)) := \{ a \in C^{\infty};  a(x,\xi,h) \sim_{h \rightarrow 0^+} \sum_{j=0}^{\infty} a_{k-j}(x,\xi) h^{m+j}, \,\, a_{k-j} \in S^{k-j}_{1,0}(T^*M) \}. \end{equation}
Here, we recall that $S^{m}_{1,0}$ is the standard H\"{o}rmander class consisting of smooth functions $a(x,\xi)$ satisfying the estimates
$|\partial_{x}^{\alpha} \partial_{\xi}^{\beta} a(x,\xi)| \leq C_{\alpha,\beta} \langle \xi \rangle^{m-|\beta|}$ for all multi-indices $\alpha, \beta \in N^{n}.$
We say that $A(h) \in Op_h(S^{m,k})$ provided its Schwartz kernel is locally of the form
\begin{equation} \label{sc2}
A(h)(x,y) = (2\pi h)^{-n} \int_{\R^n} e^{i \langle x-y,\xi \rangle/h} a(x,\xi,h) \, d\xi \end{equation}
with $a \in S^{m,k}$ and alternatively, we sometimes write $Op_h(a)$ or $a(x,hD_x)$ for the operator $A(h)$ in (\ref{sc2}). 

Let $\phi_{\lambda_j}; j=1,2,...$ be $L^2$ orthonormal basis of Laplace eigenfunctions on $(M,g)$ and $H \subset M$ a hypersurface. From now on, we assume that  $h \in \{ \lambda_j^{-1} \}; j=1,2,...,$ write $h_j = \lambda_j^{-1}$ and denote the corresponding eigenfunction by $\phi_{h_j}.$  Given $0 < \epsilon_0 <1$ an arbitrary small number, let 
\begin{equation} \label{chi}
\chi(x,\xi) \in C^{\infty}_{0}(T^*M), \,\, \chi |_{A(\epsilon_0)} \equiv 1,  \,\, \text{supp}\, \chi \subset A(2\epsilon_0), \end{equation}
with $A(\epsilon_0):= \{ (x,\xi);  (1-\epsilon_0)  < |\xi|_g < (1 + \epsilon).$  Let $\tilde{\chi} \in C^{\infty}_{0}$ be another cutoff equal to one on $A(2\epsilon_0)$ and with supp $\tilde{\chi} \subset A(4\epsilon_0).$   Consider  the eigenfunction equation
$$ (-h^2 \Delta_g - 1) \phi_{h} = 0.$$
Since $P(h):= -h^2 \Delta_g - 1$ is $h$ elliptic for $(x,\xi) \in T^*M - A(\epsilon),$  one can construct  an $h$-mircolocal parametrix  $Q(h) \in Op_{h}(S^{-1}_{0,0})$ so that
$$(1- \tilde{\chi}(h)) Q(h)  P(h) (1-\chi(h)) \phi_{h} = (1-\chi(h)) \phi_{h} + {\mathcal O}(h^{\infty}).$$ Since $P(h) \phi_{h} = 0$  and $\sigma( [ P(h), (1-\chi(h))] (x,\xi) = 0$ for $(x,\xi) \in $supp $ (1-\tilde{\chi}(h)),$  one gets the well-known energy  surface concentration estimate
\begin{equation} \label{con1} \begin{array}{ll}
 \|  (1-\tilde{\chi}(h)) \phi_{h} \|_{L^2} = {\mathcal O}(h^{\infty}). 
 \end{array}
 \end{equation}
 Since $\epsilon_0 >0$ can be chosen arbitrarily small, it follows from (\ref{con1}) that $WF_{h}(\phi_{h}) \subset S^*M. $
 
 A similar argument with the derivatives $\partial_x^{\alpha} \phi_h(x)$ combined with Sobolev lemma implies
 \begin{equation} \label{ptwise}
  \|  (1-\tilde{\chi}(h)) \phi_{h} \|_{C^k} = {\mathcal O}_k(h^{\infty}).\end{equation}

In analogy with (\ref{con1}), for the eigenfunction restrictions  one has the following energy surface mass localization result:

\begin{lem} \label{mass estimate} Let $H \subset M$ be a hypersurface and $u_{h} :=  \phi_{h}|_{H} = \gamma_H \phi_{h}.$ Then,
$$ WF_{h}(u_h) \subset B^*H.$$
\end{lem}

\begin{proof}
 Let $(s,x_n) \in \R^{n-1} \times (-\epsilon_0,\epsilon_0)$ be Fermi coordinates in an $\epsilon_0$ collar neighbourhood  of $H$ with $H= \{ x_n= 0 \}.$
In these coordinates,
\begin{equation} \label{nf1}
 -h^{2} \Delta_g =  h^{2}D_{x_n}^2  - h^{2} \Delta_{H} + {\mathcal O}(x_n) | h D_{s}|^{2} + {\mathcal O}( h^2 ( |D_s|  + |D_{x_n}|) ). \end{equation}
 
 Let $\zeta(s,x_n,\sigma) \in C^{\infty}_{0}(\R^{n-1} \times (-\epsilon_0, \epsilon_0) \times \R^{n-1})$  be equal to $1$ when $ \sigma \in B(1 + \epsilon_0)$ and vanishing outside $B(1+ 2\epsilon_0).$ 
Since $[\gamma_H, \Delta_H] = 0,$
\begin{equation} \label{nf2} \begin{array}{lll}
 (- h^2 \Delta_H - 1) (1-\zeta(s,x_n=0,h D_{s})  ) u_h  = \gamma_H  \, \big[ h^2 D_{x_n}^2   + {\mathcal O}(x_n) | h D_{s}|^{2} \\ \\  + {\mathcal O}( h^2 ( |D_s|  + |D_{x_n}|) ) \, \big] (1-\zeta(s,x_n,h D_{s}) )\phi_{h}   + \gamma_H (-h^{2}\Delta_g  - 1) (1-\zeta(s,x_n,h D_{s}) ) \phi_{h}. \end{array} \end{equation}
Since  $|\sigma|^2 \geq 1+\epsilon_0$ on supp $(1-\zeta)$, obviously $\xi_n^2 + |\sigma|^2 \geq 1 + 2\epsilon_0$ also holds on supp $(1-\zeta)$. But then, by (\ref{ptwise}) it follows that both terms on the RHS of (\ref{nf2}) are ${\mathcal O}(h^{\infty}).$ As for the LHS,  it then follows that
$$ (- h^2 \Delta_H - 1)  \cdot (1-\zeta(s,x_n=0,h D_{s}) ) u_h = {\mathcal O}(h^{\infty}).$$
Then, since $h^2 \Delta_H -1$ is $h$-elliptic on supp $1-\zeta(x_n=0,s,\sigma)$,
by the same kind of parametrix construction used to prove (\ref{con1}), it follows that
$$ \| (1-\zeta(s,x_n=0,h D_{s}) ) u_{h} \|_{L^2(H)} = {\mathcal O}(h^{\infty}).$$

\end{proof}

\subsubsection{\label{QERsc} QER for semiclassical symbols}

\begin{proof}
The proof is very similar to the homogeneous case discussed in the rest of the paper and so we only point out here the relatively  minor differences and how to deal with them.  We use the notation $\gamma_H \phi_{h_j} = \phi_{h_j} |_H$.
First, we note that by $L^2$-boundedness and the $L^2$-restriction bound \cite{BGT} $\|  \gamma_H \phi_h \|^{2}_{L^2(H)} 
= {\mathcal O}(h^{-\frac{1}{2}}),$ it follows  that  for $a \in S^{0,0}(T^*H \times (0,h_0]),$ 
\begin{equation} \label{scthm1} \langle Op_{h_j}(a)
\gamma_H \phi_{h_j}, \gamma_H\phi_{h_j} \rangle_{L^{2}(H)} = \langle Op_{h_j}(a_0)
 \gamma_H \phi_{h_j}, \gamma_H\phi_{h_j} \rangle_{L^{2}(H)} + {\mathcal O}(h_j^{\frac{1}{2}}).\end{equation}
 So without loss of generality we can assume that $a = a_0$, since in view of (\ref{scthm1}),  the lower-order terms $a_{-j}; j \geq 1$ in the symbol expansion  (\ref{scthm1}) do not affect the leading-order asymptotics of the matrix elements. 
The next step is to replace the symbol $a_0$ in the matrix elements $\langle Op_{h_j}(a_0) \gamma_H\phi_{h_j}, \gamma_H\phi_{h_j} \rangle$ by a compactly-supported  cutoff to the interior of $B^*H.$ In the following, we let $B_{r}^{*}H:=  \{(s,\sigma); |\sigma| \leq r \}$. Then, for a fixed small constant $\epsilon >0$  we let  $\chi_{H,in} \in C^{\infty}_{0}(T^*H)$ with supp $\chi_{in} \subset B_{1-\epsilon}^*H, \,\, \chi_{H,\epsilon} \in C^{\infty}_{0}(T^*H)$ with supp $\chi_{\epsilon} \subset B_{1+\epsilon}^*H \setminus B_{1-\epsilon}^*H$ and choose $\chi_{H,out} \in C^{\infty}(T^*H)$ supported outside $B^*H$ with the property that
$$ \chi_{H,in} + \chi_{H,\epsilon} + \chi_{H,out} \equiv 1 \,\,\, \text{on} \,\, T^*H.$$
Due to the large number of semiclassical pseudodifferential cutoffs appearing in the argument, we sometimes denote both a cutoff function $\chi_H \in C^{\infty}_{0}(T^*H)$ and the corresponding operator $Op_{h_j}(\chi_H)$ both by $\chi_H$. By Lemma \ref{mass estimate}, 
\begin{equation} \label{outermass}
\langle \chi_{H,out} Op_{h_j}(a_0) \gamma_H\phi_{h_j}, \gamma_H\phi_{h_j} \rangle = {\mathcal O}(h_j^{\infty}) \end{equation}  and so, the matrix elements 
\begin{equation} \label{scmatrix} \begin{array}{ll}
 \langle Op_{h_j}(a_0) \gamma_H \phi_{h_j}, \gamma_H\phi_{h_j} \rangle = \langle  \chi_{H,in} Op_{h_j}(a_0)
 \gamma_H\phi_{h_j},\gamma_H\phi_{h_j} \rangle_{L^{2}(H)} \\ \\ + \langle \chi_{H,\epsilon} Op_{h_j}(a_0)
 \gamma_H\phi_{h_j}, \gamma_H\phi_{h_j} \rangle_{L^{2}(H)} + {\mathcal O}_{\epsilon}(h_j^{\infty}). \end{array} \end{equation}



 In analogy with the homogeneous case, the main part of the proof of Theorem \ref{sc} is the variance estimate
\begin{equation} \label{scvariance2}
\limsup_{h \rightarrow 0^+} \frac{1}{N(h)}  \sum_{j; h_j \geq h} \vert  \langle  \chi_{H,in} Op_{h_j}(a_0) \gamma_H \phi_{h_j}, \gamma_H \phi_{h_j} \rangle - \omega(\chi_{H,in} a_0)  \vert ^2  = 0,\end{equation} where $N(h):= \# \{ j; h_j \geq h \} \sim_{h \rightarrow 0^+} C_n h^{-n}.$
The averaging argument proceeds as before, except that the constituent homogeneous Fourier integral operators are cut-off using the mass concentration estimates in (\ref{con1}), Lemma \ref{mass estimate} and the reduction (\ref{scvariance2}). The resulting cut-off operators are then compactly supported semiclassical Fourier operators in the sense of \cite{GuSt2}.

Let $\chi \in C^{\infty}_{0}(T^*M)$ be the cutoff in (\ref{chi}). We define the semiclassically cut-off wave operators $U(\cdot,h):C^{\infty}(M) \rightarrow C^{\infty}(M \times \R)$ by 
\begin{equation} \label{wave}
U(\cdot,h):=\chi_{\R}(t,hD_t)  \, [  \chi(x,hD_x) \, U(\cdot) \, \chi(x,hD_x) ]. \end{equation}
where, $\chi_{\R}(t,t',h):= (2\pi h)^{-1} \int_{\R} e^{i(t-t')\tau/h} \chi(\tau-1) \chi(t/T) \, d\tau$ and in the latter $\chi \in C^{\infty}_{0}(\R)$ is a cutoff as in (\ref{VTFORM}).
Similarily, we  define the cut-off restriction operators $\gamma_H(h): C^{\infty}(M) \rightarrow C^{\infty}(H)$ by
\begin{equation} \label{restrict}
\gamma_H(h):= \chi_{in}(s,hD_s) \, \gamma_H \, \chi(x,hD_x).\end{equation} It follows that
\begin{equation} \label{scvariance3} \begin{array}{ll}
\frac{1}{N(h)}  \sum_{j; h_j \geq h} \vert  \langle  \chi_{in} Op_{h_j}(a_0) \gamma_H \phi_{h_j},  \gamma_H \phi_{h_j} \rangle - \omega(\chi_{in} a_0)  \vert ^2  \\ \\ = 
\frac{1}{N(h)}  \sum_{j; h_j \geq h} \vert  \langle  V_{T,\epsilon}(a_0,h)  \phi_{h_j},  \phi_{h_j} \rangle - \omega(\chi_{in} a_0)  \vert ^2 + {\mathcal O}\left( \frac{1}{N(h)} \right),\end{array} \end{equation} 
where, the semiclassical averaging operator 
\begin{equation} \label{average}
V_{T,\epsilon}(a_0,h):= \frac{1}{T} \int_{\R} U(-t,h) \gamma_H(h)^*  \chi_{in} Op_h(a_0) \gamma_H(h) U(t,h) \, \chi (\frac{t}{T} ) \, dt.\end{equation}
Thus, it suffices to take $\limsup_{h \rightarrow 0^+}$ of the RHS in (\ref{scvariance3}). 
In analogy with Proposition \ref{VTDECOMPa} one shows that modulo residual error $V_{T,\epsilon}(a_0,h)$ is the sum of a semiclassical pseudodifferential operator in $Op_h(S^{0,-\infty}(T^*M \times (0,h_0]))$  and a compactly-supported zeroth-order semiclassical Fourier integral operator.  

Given a manifold-Lagrangian pair $(X \times Y,\Lambda)$ and an operator $F(h): C^{\infty}(X) \rightarrow C^{\infty}(Y)$ with $WF'_{h}(F(h)) \subset \Lambda,$ following \cite{GuSt2} we say that    $F(h) \in {\mathcal F}_0^{k}(X \times Y, \Lambda)$ provided it has a Schwartz kernel  locally of the form
$$ F(h)(x,y) = (2\pi h)^{k-\frac{N}{2} -\frac{\dim Y}{2}} \int_{\R^N} e^{i\phi(x,y,\theta)/h} a(x,y,\theta,h) \, d\theta,$$
with $a \in C^{\infty}_{0}$ in all variables. In this case,  we call $F(h):C^{\infty}(X) \rightarrow C^{\infty}(Y)$ a compactly-supported semiclassical Fourier integral operator (scFIO) of order $k.$ We refer the reader to \cite{GuSt2} Chapter 8 for a detailed discussion of composition formulas and symbol calculus for these operators. In particular, given two scFIO's $F_1(h) \in {\mathcal F}_0^{m_1}(X_1 \times X_2,\Lambda_1)$ and $F_{2}(h) \in {\mathcal F}_{0}^{m_2} (X_2 \times X_3, \Lambda_2)$ with associated Lagrangians $\Lambda_1 \subset T^*X_1 \times T^* X_2$ and $\Lambda_2 \subset T^*X_2 \times T^*X_3$ that are  transversally composible, one has
\begin{equation} \label{composition}
F_1(h) \circ F_2(h) \in {\mathcal F}_0^{m_1 + m_2  } (X_1 \times X_3, \Lambda_1 \circ \Lambda_2). \end{equation} Following the argument in section \ref{v operator}, one shows that 
$ U(-t_1,h) \gamma_H(h)^* \chi_{in} Op_h(a_0) \gamma_H(h) U(t_2,h) \in {\mathcal F}^{1/2}_{0}(M \times \R \times M \times \R, (\Gamma^* \circ C_H \circ \Gamma)_{\chi} )$
 where, $(\Gamma^* \circ C_H \circ \Gamma)_{\chi} = (\Gamma^* \circ C_H \circ \Gamma)  \cap ( \text{supp} \chi \times T^*\R \times \text{supp} \chi \times T^*\R)$ and $ \pi_{T_*} \Delta_t^* \in {\mathcal F}_{0}^{ \,-1/2}( (\R \times \R \times M \times M) \times (M \times M),\Gamma_{\chi})$ where $\Gamma$ is the identity graph and $\Gamma_{\chi} = \Gamma \cap (T^*\R \times T^*\R \times \text{supp} \chi  \times \text{supp} \chi \times \text{supp} \chi \times \text{supp} \chi). $ We note that the transversality conditions for all of the scFIO compositions in (\ref{average}) are verified exactly as before since the associated Lagrangians are just  subsets of the corresponding conic Lagrangians in section \ref{v operator}. Since $V_{T,\epsilon}(a_0,h) = ( \pi_{T_*} \Delta_t^{*}) \circ U(-t_1,h) \gamma_H(h)^* \chi_{in} Op_h(a_0) \gamma_H(h) U(t_2,h),$ it follows from (\ref{composition}) that  the $h$-microlocal deomposition of $V_{T,\epsilon}(a_0,h)$ then parallels the one in Proposition \ref{VTDECOMPa} with
 $$V_{T,\epsilon}(a_0,h) = P_{T,\epsilon}(a_{0},h) + F_{T,\epsilon}(a_0,h) + R(a_0,h).$$
 Here, $P_{T,\epsilon}(a_0,h) \in Op_{h}(S^{0,-\infty}(T^*M \times (0,h_0])),$ 
 $F_{T,\epsilon} \in {\mathcal F}_0^0 (M \times M, (\Gamma_{T,\epsilon})_{\chi} )$ and $\| R(a_0,h) \|_{L^2 \rightarrow L^2} = {\mathcal O}(h^{\infty}).$
 The principal symbol formulas also parallel the homogeneous ones in (\ref{aTFORM}) with $\sigma(P_{T,\epsilon}(a_0,h))(x,\xi) = (\chi_{in} a_0)_{T,\epsilon}(x,\xi)$ and likewise for $\sigma(F_{T,\epsilon}(a_0,h))(x,\xi).$ The rest of the proof of Theorem \ref{sctheorem} folows as in Theorem \ref{maintheorem}.   \end{proof}

\section{\label{FIOS} Appendix}


In this appendix, we briefly review the basic facts of symbol
composition that we will use in the calculations. We are working
in the framework of homogeneous pseudo-differential operators on
$H$. Thus, we assume $a(s,\sigma) \in S^{0}_{cl}(T^*H)$ is a
zeroth order classical polyhomogeneous symbol on $H$  with $a \sim
\sum_{j=0}^{\infty} a_j, \, a_j \in S^{-j}_{1,0}(T^*H)$ and
$Op_H(a)$ is its quantization as a pseudo-differential operator on
$L^2(H)$. We refer to \cite{DS,GS} and especially to volume IV of
\cite{HoI-IV} for background on Fourier integral operators. We use
the notation $I^m(X \times Y, C)$ for the class of Fourier
integral operators of order $m$ with wave front set along the
canonical relation $C$, and $WF'(F)$ to denote the canonical
relation of a Fourier integral operator $F$.

We recall that a Fourier integral operator $A: C^{\infty}(X) \to
C^{\infty}(Y)$ is an operator whose Schwartz kernel may be
represented by an oscillatory integral
$$K_A(x,y) = \int_{\R^N} e^{i \phi(x, y, \theta)} a(x, y, \theta) d\theta$$
where the phase $\phi$ is homogeneous of degree one in $\theta$.
The critical set of the phase is given by
$$C_{\phi} = \{(x, y, \theta): d_{\theta} \phi = 0\}. $$
Under ideal conditions, the map
$$\iota_{\phi} : C_{\phi} \to T^*(X, Y), \;\;\; \iota_{\phi}(x, y,
\theta) = (x, d_x \phi, y, - d_y \phi) $$ is an embedding, or at
least an immersion. In this case the phase is called
non-degenerate.  Less restrictive, although still an ideal
situation, is where the phase is clean. This means that the map
$\iota_{\phi} : C_{\phi} \to \Lambda_{\phi} $, where
$\Lambda_{\phi} $ is the image of $\iota_{\phi}$,  is locally a
fibration with fibers of dimension $e$.  From \cite{HoI-IV}
Definition 21.2.5, the number of linearly independent
differentials $d \frac{\partial \phi}{\partial \theta}$ at a point
of $C_{\phi}$ is $N - e$ where $e$ is the excess.

We a recall that the order of $F: L^2(X) \to L^2(Y)$ in the
non-degenerate case  is given in terms of a local oscillatory
integral formula by $m +
 \frac{N}{2} - \frac{n}{4},$, where $n = \dim X + \dim Y, $ where
$m$ is the order of the amplitude, and $N$ is the number of phase
variables in the local Fourier integral representation (see
\cite{HoI-IV}, Proposition 25.1.5); in the general clean case with
excess $e$, the order goes up by $\frac{e}{2}$ (\cite{HoI-IV},
Proposition 25.1.5'). Further, under clean composition of
operators of orders $m_1, m_2$, the order of the composition is
$m_1 + m_2 - \frac{e}{2}$ where $e$ is the so-called excess (the
fiber dimension of the composition); see \cite{HoI-IV}, Theorem
25.2.2.

The symbol $\sigma(\nu)$ of a Lagrangian (Fourier integral)
distributions is a section of the bundle $\Omega_{\half} \otimes
\mcal_{\half}$ of the bundle of half-densities (tensor the Maslov
line bundle). In terms of a Fourier integral representation it is
the square root $\sqrt{d_{C_{\phi}}}$ of the delta-function on
$C_{\phi}$ defined by $\delta(d_{\theta} \phi)$, transported to
its image in $T^* M$ under $\iota_{\phi}$. If $(\lambda_1, \dots,
\lambda_n)$ are any local coordinates on $C_{\phi}$, extended as
smooth functions in  neighborhood, then
$$ d_{C_{\phi}}: = \frac{|d \lambda|}{|D(\lambda,
\phi_{\theta}')/D(x, \theta)|}, $$ where $d \lambda$ is the
Lebesgue density.

\end{document}